\UseRawInputEncoding
\documentclass[11pt,letterpaper]{amsart}
\usepackage{amssymb,enumerate}
\usepackage[colorlinks, linkcolor=black, anchorcolor=black, citecolor=black]{hyperref}
\usepackage{graphicx}

\usepackage{varioref}

\usepackage{mathrsfs}
\usepackage{xypic,amscd,amssymb,latexsym,srcltx}

 \newtheorem{lemma}{Lemma}[section]
\newtheorem{proposition}[lemma]{Proposition}
\newtheorem{corollary}[lemma]{Corollary}
\newtheorem{theorem}[lemma]{Theorem}
\newtheorem{remark}[lemma]{Remark}
\newtheorem{definition}[lemma]{Definition}
\newtheorem{example}[lemma]{Example}

\begin{document}

\thispagestyle{empty}
\title{On polynomial representation of finite groups}
\author{\bf
  Lizhong Wang and Jiping Zhang}

\address{School of Mathematics\\
Peking University\\ Beijing 100871\\ P.R.China}
 \address{lwang@math.pku.edu.cn}
\address{jzhang@pku.edu.cn}

\keywords{partition algebra, definition equations of characters, Frobenius polynomial, degree polynomial, McKay conjecture\\
2020 Mathematics Subject Classifications:20C15,20C20}
\begin{abstract}
By generalizing Frobenius' polynomial method to good partition algebra, we will develop new character theories for a finite group $G$. A uniform defining equations are derived for these kinds of character theories. The new character theories leads to various factorizations of the group determinant. We will show that these new character theories are equivalent to the Frobenius polynomials of the correspondent good partition algebras. In particular, the character table of a finite group $G$ can be replaced by the Frobenius polynomial of $G$ which is a kind of degenerate of the group determinant. As applications, we  find a new series of invariants $p_{ijl}$ for a finite group. In particular, a finite simple group is determined by these invariants. As further applications, the degrees of all irreducible characters can also be realized as the solutions of a polynomial of $G$ and we can
combine the refined McKay conjecture and part of Galois conjecture of Navarro on degrees of irreducible characters into a new general conjecture.

\end{abstract}

\maketitle

\begin{center}
  \tableofcontents
\end{center}

\section{Introduction} \label{sec:intro}
The representation theory of finite groups was created by Frobenius in order to factor group determinant $\Theta_G(x)$ into irreducible factors over complex field. The characters was invented to derive the irreducible factors by formulae(see \cite{F1} and \cite{F2}). After that, it was realized that characters are also powerful in investigating the structure of finite groups. In fact, a finite group is determined completely by its group determinant or its  $1-$,$2-$ and $3$-characters (see \cite{FS} and \cite{HJ}). The theory of factorization of a group determinant has appeared in diverse areas. There are natural connections among the original character theory of Frobenius, geometry and analysis. The main purpose of this paper is to generalize character theory of Frobenius and  study representations by polynomials.

Assume that $\Theta_G(x)$ has the following factorization over complex field: $$\Theta_G(x)=\prod_{i=1}^s \Phi_i(x)^{e_i},$$ where $\Phi_1(x),\cdots,\Phi_s(s)$ are all co-prime irreducible factors and $e_i$ is the multiplicity of $\Phi_i(x)$. Let $f_i$ be the degree of $\Phi_i(x)$. If $g$ is not the identity of $G$, we denote the coefficient of the term $x_e^{f_i-1}x_g$ in $\Phi_i(x)$ by $\chi_i(g)$. If we set $\chi_i(1)=f_i$, $\chi_i:G \rightarrow \mathbb{C}$ is a function of $G$. It is called a character of $G$ by Frobenius and it is in fact an irreducible character $G$ over the field $\mathbb{C}$. It was shown in \cite{F1} and \cite{F2} that $s$ is equal to the number of conjugate classes of $G$ and $\chi_1,\cdots,\chi_s$ are all irreducible characters of $G$ over $\mathbb{C}$.

 A partition $\mathcal{P}=\cup_{i=1}^n\mathcal{P}_i$ of Cl($G$) is called good partition if it satisfies conditions in Definition \ref{def:goodpart}. The $\mathcal{P}_i$ is called a partition class. Assume that $\hat{C_i}\hat{C_j}=\sum_{l=1}^n a_{lij}\hat{C_l}$, where $\hat{C}_i=\sum_{g\in \mathcal{P}_i}g$. Given a good partition $\mathcal{P}=\cup_{i=1}^n \mathcal{P}_i$.  Then we can define a function $\chi_i^\mathcal{P}$ on $G$ for each $\Phi_i(x)$ such that  $\chi_i^\mathcal{P}(g)=\sum_{y\in \mathcal{P}_j}\chi_i(y)$, where $\mathcal{P}_j$ is the partition class to which $g$ belongs. These functions are called $\mathcal{P}$-characters. The $\mathcal{P}$-characters are new ones not super characters defined in \cite{A} and \cite{IN} as we have shown in Section \ref{sec:realization}. We can prove that these characters are determined by the following series of equations.

{\bf Theorem A.}{\it The $\mathcal{P}$-characters of  a good partition $\mathcal{P}$ of Cl($G$) are common roots of the following equations:
\[ x_ix_j=\frac{1}{\lambda_\mathcal{P}}x_1\sum_l a_{lij}x_l, 1\leq i,j,l\leq n,\] where $\lambda_\mathcal{P}$ is the identity constant of
$k\mathcal{P}$.}

The above defining equations can afford almost all known characters.
If we choose $\mathcal{P}$ to be the trivial partition with each partition class containing only one conjugate class of $G$, then above equations are definition equations of  characters defined by Frobenius in \cite{F1}.
 
The variety defined by the equations in Theorem A is called character variety with respect to the good partition and it is denoted by $V_\mathcal{P}(G)$. In fact, the variety $V_\mathcal{P}(G)$ is decided by the irreducible representations of the partition algebra $k\mathcal{P}$.
Given a good partition $\mathcal{P}={\mathcal{P}_1\cup\cdots\cup\mathcal{P}_n}$.  Let $\mathcal{L}:k\mathcal{P}\rightarrow M_n(k)(\hat{C}_j\mapsto \mathcal{A}_j)$ be the regular representation of $k\mathcal{P}$.
Set $\mathfrak{A}=\sum_{i=1}^n x_i\mathcal{A}_i=(\alpha_{ij})$.
Let $\gamma_1,\cdots,\gamma_n$ be all irreducible representations of $k\mathcal{P}$ and each $\gamma_t$ is written as $(\lambda_{1t},\cdots,\lambda_{nt})$. The representation variety $V_k(k\mathcal{P})$ of $k\mathcal{P}$ consists of irreducible representations and zero representation of $k\mathcal{P}$.

{\bf Theorem B} {\it The character variety $V_\mathcal{P}(G)$ consists of all multiples of elements in  $V_k(k\mathcal{P})$ i.e.
$V_\mathcal{P}(G)=\{\alpha\gamma|\alpha\in k, \gamma\in V_k(k\mathcal{P}) \}$. In particular, $V_\mathcal{P}(G)$  consists of $n$ irreducible components and the dimension of $V_\mathcal{P}(G)$ is $1$. }

The $\mathcal{P}$-characters are closely related to the decomposition of the group determinant. A formula to construct a power factor of the group determinant by its power function is presented in section \ref{sec:powerfactors}. We also give transfer formulae among $\mathcal{P}$-character,super characters and power functions. So we can construct a factor of the group determinant by a  $\mathcal{P}$-character.

 With defining equations, we can prove the orthogonal relations  and invariants of the $\mathcal{P}$-characters. In particular, we  give a formula for the multiplicity of the factors in group determinant.
Set $\gamma_t(x)=\sum_{i=1}^n \lambda_{it}x_i$ for $1\leq t \leq n$. Assume that $x_A=x_B$ in $\Theta_G(x)$ whenever $A$ and $B$ belong to the same partition class of $\mathcal{P}$. Then  $\Theta_G(x)=\prod_{i=1}^n\gamma_i^{m_i}(x)$, where $m_i$ is a nonnegative integer. We call these $m_is$ squares of the degrees  of characters. In the case $\mathcal{P}$ is the trivial partition of Cl($G$), these $m_is$ are exactly  squares of degrees of irreducible character of $G$ indeed.
Let $\varsigma_1,\cdots,\varsigma_n$ be all solutions of the following equations:
\begin{equation*}\label{equ:orthogonalM} x_{i}\gamma_t(x)=\sum_ja_{ij}(x)x_{j}, \leqno(2)\end{equation*} where $a_{ij}(x)$ is the element on $(ij)$-entry of $\mathfrak{A}$ and $1\leq t \leq n$. Let $\mathcal{P}_{j'}$ consists of all inverse of elements in $\mathcal{P}_j$. Define $p_{ij}$ to be the number of solutions  of the equation \begin{equation*}\label{def:commutators} xy=g_i\mbox{ or $g_j$}\end{equation*} where
$x\in \mathcal{P}_l$,$y\in \mathcal{P}_{l'}$ for  $1\leq l \leq n$ and $g_i$ and $g_j$ are representatives of $\mathcal{P}_{i}$ and $\mathcal{P}_{i}$ respectively.

{\bf Theorem C.}{\it
The polynomial $D_\mathcal{P}(x)=|xI-Diag(\ell\ell_1,\cdots,\ell\ell_n)M\mathcal{SS}'|$ has the following factorization:
$D_\mathcal{P}(x)=(x-m_1)\cdots(x-m_n)$, where $M$ the permutation matrix such that $(p_{ij'})M=(p_{ij})$ and $\mathcal{S}=(\varsigma_1,\cdots,\varsigma_n)$.
 }

 The polynomial $D_\mathcal{P}(x)$ is called the degree polynomial associate with $\mathcal{P}$. Let $k$ be a subfield of $\mathbb{C}$. Let $\xi$ be $m$-th primitive root of unitary. Then $k(\xi)$ is a splitting field of $G$ and its subgroups whenever $m$ is the exponent of $G$.   For any $\sigma\in \mbox{Gal}(k(\xi)/k)$, there exists a unique $t \in \mathbb{Z}/m\mathbb{Z}$ such that $\sigma(\omega)=\omega^t$ for such $\omega$ with $\omega^m=1$. Two elements $x,y \in G$ are called $k$-conjugate if $x$ is conjugate to $y^t$, where $t$ is determined by some $\sigma\in \mbox{Gal}(k(\xi)/k)$.
Let $\mathcal{P}_1,\cdots,\mathcal{P}_n$ be all $k$-classes of $G$ and let $\hat{{C}}_i$ is the sum of elements in $\mathcal{P}_i$. Then $\mathcal{P}=\cup_{i=1}^n \mathcal{P}_i$ is a good partition.   Consider $D_\mathcal{P}(x)$ as polynomial over a field $F$ with characteristic $p$. Then $\overline{D}_\mathcal{P}(x)=x^eD_G^{k,p'}(x)$, where $(x^e,D_G^{k,p'}(x))=1$. The polynomial $D_G^{k,p'}(x)$ is called $p'$-degree polynomial of $G$ with respect to $k$-classes and the prime $p$.

{\bf Conjecture D.} {\it Let $k$ be a field with characteristic zero and let $N$ be the normalizer of some Sylow $p$-subgroup of $G$. Then $D_G^{k,p'}(x)=D_N^{k,p'}(x)$.}

If $k$ is a splitting field of $G$, then $k$-classes are just conjugate classes of $G$. In this case, above conjecture is equivalent to the refined McKay conjecture in \cite{IN}. If Gal($k(\xi)/k)$) is a cyclic $p$-groups, then above conjecture is equivalent to Conjecture A in \cite{N}.
  In \cite{SL}, $D_G(x)$,$D_G^{k,p'}(x)$ and $D_{N}^{k,p'}(x)$ of $G$ are calculated out for prime divisors of $G$ in case that $k$ is splitting field of $G$, where $G$ is one of the following groups: $M_{11},M_{12},M_{22},M_{23},M_{24},Sz(8),Sz(32)$. By comparing $D_G^{k,p'}(x)$ and $D_{N}^{k,p'}(x)$, the  Isaacs-Navarro-McKay conjecture holds true for these groups.

Let $\mathfrak{L}$ be the regular representation of an algebra $\mathcal{A}$ over $k$. Let $\alpha_1,\cdots,\alpha_s$ be a basis of $\mathcal{A}$. Then the matrix $\lambda I-\sum_i x_i\mathfrak{L}(\alpha_i)$ is similar to  \[\mbox{ Diag}(\epsilon_1(\lambda,x_1,\cdots,x_r),\cdots,\epsilon_r(\lambda,x_1,\cdots,x_r)),\] where $\epsilon_i(\lambda,x_1,\cdots,x_r)$ is polynomial and $\epsilon_i(\lambda,x_1,\cdots,x_r)\mid \epsilon_{i+1}(\lambda,x_1,\cdots,x_r)$. The polynomial $\epsilon_r(0,x_1,\cdots,x_r)$ is called the 
norm form of $\mathcal{A}$. Given an element $a=\sum_{i}l_i\alpha_i\in \mathcal{A}$. Then $|\lambda I-\sum_i l_i\mathfrak{L}(\alpha_i)|$ is called the characteristic polynomial of $a$.
 The norm form of the partition algebra $k\mathcal{P}$ is called Frobenius polynomial with the basis $\hat{C}_1,\cdots,\hat{C}_n$.
The Frobenius polynomial is equivalent to the character table. The advantage of the Frobenius polynomial is that the theory of forms can be applied on it.  For example, as a form, the Frobenius polynomial can be determined by its first three kinds of coefficients. We can prove that the first three kinds coefficients are  the generalized number $p_{ijl}$(see definition \ref{def:pijk}) of  commutators of $G$.

 {\bf Theorem E.} {\it  Let $\mathcal{P}$ be the trivial partition of Cl($G$). Then
\begin{enumerate}
\item[(i)]{The Frobenius polynomial $F_G(x)$ of $G$ and the ordinary characters table $T_G$   are determined by each other.}
\item[(ii)]{The Frobenius polynomial $F_G(x)$ is determined by the numbers $p_{ijl}$ where $1\leq i,j,l\leq n$.}
\end{enumerate} }

It is important to find which intrinsic constants of $G$ can determine the group $G$, especially in the case that $G$ is a finite simple group. There are also some conjectures on this direction such as Thompson conjecture and Huppert conjecture. As corollary of Theorem D, we find the generalized numbers of commutators can determine the normal series of $G$.

{\bf Corollary F.} {\it  Let $G$ and $H$ be finite groups. Let $p_{ijl}(G)$ be the generalized number of commutators  of $G$.
 \begin{itemize}
\item[(i).] {If  $p_{ijl}(G)=p_{ijl}(H)$ for all $1\leq i,j,l\leq n$ then $G$ and $H$ have the same chief factors with multiplicity;}
\item[(ii).] {A finite simple group $G$ is determined by $p_{ijl}(G)$, $1\leq i,j,l\leq n$, i.e. $G$ is isomorphic to $H$ if and only if $p_{ijl}(G)=p_{ijl}(H)$ for all $1\leq i,j,l\leq n$.}
 \end{itemize}}

Here is the outline of sections.
In section \ref{sec:Frobenius theory}, we will define good partitions of Cl($G$) and present some useful examples of good partitions. We will also realize all irreducible representations of the partition algebra $k\mathcal{P}$  into the common solutions of a series equations. In section \ref{sec:polyalg}, we introduce various polynomials of an algebra and different forms of these polynomials under different basis the algebra.  In section \ref{sec:defequ}, we will derive the uniform definition equations of various character theories. Many examples of simplifying the definition equations are also presented.  In section \ref{sec:realization}, the characters of the defining equations in section \ref{sec:defequ} are decided. We also give a comparing formula for super characters and $\mathcal{P}$-characters. In section \ref{sec:polyalg}, we give a formula to construct power factors by power function. In section \ref{sec:orthdegree}, we prove the orthogonal relations of $\mathcal{P}$-characters. We will also prove formulae of the degrees of $\mathcal{P}$-characters.
In section \ref{sec:FrobPoly}, we will define Frobenius polynomial for each partition algebra and show  the partition algebra is determined by its Frobenius polynomial.
We will also define the generalized number of commutators  and show how the numbers  determine the Frobenius polynomial. In section \ref{sec:degreepoly}, we focus on the ordinary characters corresponding to the trivial partition of Cl($G$).
We show that $F_G(x)$ is determined by the numbers of commutators. In particular, a finite simple group is determined by those numbers. A sufficient and necessary condition is proved for two groups having the same character table.  
In section \ref{sec:k-character}, we will use the degree polynomials for $k$-characters to combine various McKay-type conjectures on degrees of irreducible characters into a new conjecture. 


The notions of this paper are standard. We refer to \cite{I} and \cite{S} for character theory of finite groups and \cite{MA} for symmetric polynomials. For more on group determinants and representations, we refer to \cite{J}. The terminologies and notations for basic algebraic geometry are as the same as in \cite{G}.

{\bf Acknowledgement.} We are grateful to Jon F Carlson and Pham Huu Tiep for many corrections and helpful suggestions. We also thank Liu, Yanjun and Liao,pengcheng for reading through this paper.

\section{Partition algebras and its representations}   \label{sec:Frobenius theory}
In this section we will introduce partition algebras as subalgebra of the center of the group algebra $kG$ and give some useful examples of partition algebras.

Let $G$ be a finite group with $\mbox{Cl}(G)=\{\mathscr{C}_i\}_{i=1}^s$, where $\mathscr{C}_i$ is conjugate class of $G$ for $1\leq i\leq s$. Let $\hat{\mathscr{C}}_i=\sum_{g\in \mathscr{C}_i}g$. Let $\mathscr{C}_{i'}=\{g^{-1}|g\in {\mathscr{C}}_{i}\}$.
Given a partition $\mathcal{P}={\mathcal{P}_1\cup\cdots\cup\mathcal{P}_n}$ of  $\mbox{Cl}(G)$, where 
$\mathcal{P}_j=\{\mathscr{C}_{j_1},\cdots,\mathscr{C}_{j_l}\},1\leq j\leq r$. We call $\mathcal{P}_i$ a partition class of $\mathcal{P}$. Let $\mathcal{P}_{j'}=\{\mathscr{C}_{j'_1},\cdots,\mathscr{C}_{j'_l}\}$. We denote the sum $\sum_{\mathscr{C}_{j_s}\in \mathcal{P}_j}\hat{\mathscr{C}}_{j_s}$ of $\mathcal{P}_j$ by $\hat{C_j}$.

\begin{definition}\label{def:goodpart}
A partition $\mathcal{P}={\mathcal{P}_1\cup\cdots\cup\mathcal{P}_n}$ of a subset $\mathcal{C}$ of $\mbox{Cl}(G)$ with the identity $1$ belonging to $\mathcal{P}_1$ is called a good partition if
\begin{enumerate}
\item{$\mathcal{P}_{j'}$ is a partition class of $\mathcal{P}$ for $1\leq j\leq n$,}
\item{There exist elements $b_{ij\ell}$ in $\mathbb{Z}$ such that $\hat{{C}_j}\hat{{C}_l}=\sum_{i=1}^n b_{ijl}\hat{{C}_i}$ for $1\leq j,l\leq n$,}
\item{there is an identity in the algebra $k\mathcal{P}=\oplus_{i=1}^nk\hat{{C}_i}$.}
\end{enumerate}
The algebra  $k\mathcal{P}=\oplus_{i=1}^nk\hat{{C}_i}$ is called a partition algebra with respect to $\mathcal{P}$.
If there exists $\lambda_\mathcal{P}\in k$ such that $\lambda_{\mathcal{P}}\hat{C}_1$ is the identity of $k\mathcal{P}$, then we say that $\lambda_{P}$ is the identity constant of $k\mathcal{P}$.
\end{definition}

\begin{example}\label{exa:goodpart2}

 The trivial partition $\mathcal{P}=\{\{\mathscr{C}_1\},\cdots,\{\mathscr{C}_s\}\}$ is a good partition. The partition algebra with respect to this partition is the center of $kG$. The identity is the identity of the group algebra $kG$. This partition is called the trivial partition.
\end{example}

\begin{example}
 Let $N$ be a normal subgroup $G$. Let $\mbox{Cl($G/N$)}=\{\mathcal{\overline{C}}_i\}_{i=1}^t$. Then $\mathcal{P}=\{\mathcal{C}_iN\}_{i=1}^t$ is a good partition whose identity of the partition algebra is $\frac{1}{|N|}\hat{N}$ and it is isomorphic to the center of $k(G/N)$.
\end{example}

\begin{example}\label{exa:goodpart1}
Let $N$ be a normal subgroup of $G$. Then there exist a subset $\mathcal{C}=\{\mathscr{C}_1,\cdots,\mathscr{C}_r\}$ of Cl($G$) such that $N=\cup_{i=1}^r \mathscr{C}_i$. Then $\mathcal{C}$ is a good partition of Cl($N$) and $k\hat{C}_1+\cdots+k\hat{C}_r$ is a partition algebra.
\end{example}

\begin{example}\cite{A}\cite{DI}
  Let $G=U_n(F)$ be the group of $n \times n$ unimodular upper triangular matrices
over a finite field $F$ of characteristic $p$. Let $J$ be the algebra of strictly upper triangular $n\times n$ matrices over $F$.
Then all subsets of the form $1 + GxG$ consist of a good partition of Cl($G$).
\end{example}

Next we will give a good partition of Cl($G$) with respect to a subfield $k$ of $\mathbb{C}$.

Let $k$ be a subfield of $\mathbb{C}$. Let $K=k(\xi)$ be an extension field of $k$ by adjoining primitive $m$-th root $\xi$ of unity to $F$, where $m=\mbox{exp}(G)$. Then $K$ is a splitting field of $G$ and its subgroups. For any $\sigma\in \mbox{Gal}(K/k)$, there exists a unique $t \in \mathbb{Z}/m\mathbb{Z}$ such that $\sigma(\omega)=\omega^t$ for such $\omega$ with $\omega^m=1$(see \cite{BO}). Hence each element  $\sigma\in \mbox{Gal}(K/k)$ can define a map on $G$  by $g^\sigma=g^t$. In this way $\mbox{Gal}(K/k)$ acts on $G$ as permutations. Two elements $x,y \in G$ are called $\mbox{Gal}(K/k)$-conjugate if $x$ is conjugate to $y^\sigma$ for some $\sigma\in \mbox{Gal}(K/k)$, i.e. $gxg^{-1}=y^\sigma$ for some $g\in G$. A $\mbox{Gal}(K/k)$-conjugate class is called a $F$-class of $G$. Let $\mathcal{P}_1,\cdots,\mathcal{P}_n$ be all $F$-classes of $G$ and let $\hat{{C}}_i$ is the sum of elements in $\mathcal{P}_i$.

Let $Irr_K(G)=\{\chi_1,\cdots,\chi_s\}$ be the set of irreducible characters of $G$ over $K$. Let $e_i=\frac{\chi_i(1)}{|G|}\sum_{j=1}^s\chi_i(g_j^{-1})\hat{C}_j$ be the central primitive idempotent in $Z(KG)$ corresponding to $\chi_i$. Then $\mbox{Gal}(K/k)$ acts on $\{\chi_1,\cdots,\chi_s\}$ and $\{e_1,\cdots,e_s\}$. For any $\sigma\in \mbox{Gal}(K/k)$, $\chi_i^\sigma=\chi_i$ if and only if $e_i^\sigma=e_i$. Let $\mathcal{O}_1,\cdots,\mathcal{O}_\nu$ be the orbits of $\mbox{Gal}(K/k)$ on $\{e_1,\cdots,e_s\}$. Then $\{\varepsilon_j=\sum_{e_\ell\in \mathcal{O}_j}e_\ell\}_{j=1}^\nu$ are all central primitive idempotents of $Z(kG)$.

Since $\mbox{Gal}(K/k)$ is abelian, the inverse map  is an automorphism of $\mbox{Gal}(K/k)$. Define the action of $\mbox{Gal}(K/k)$ on $G$ by $g^\sigma= g^{t'}$, where $g\in G$, $\sigma\in \mbox{Gal}(K/k)$ and $t'$ is the unique element in $\mathbb{Z}/m\mathbb{Z}$ corresponding to $\sigma^{-1}$. With this action on Cl$(G)$, we have  $\chi_i^\sigma({g_j}^\sigma)=\chi_i(g_j)$. By Brauer's permutation lemma, the numbers of orbits of the actions of $\mbox{Gal}(K/k)$ on  $Irr_K(G)$ and Cl$(G)$ are equal,  i.e. $n=\nu$.

\begin{proposition}
Let $\varepsilon_k(G)=k\varepsilon_1\oplus \cdots \oplus k\varepsilon_n$ be the subalgebra of $Z(KG)$. Then $\hat{{C}}_1,\cdots, \hat{{C}}_n$ is a basis of $\varepsilon_k(G)$.

\end{proposition}
\begin{proof}
Let $\mathcal{O}_i=\{\chi_{i_1},\cdots,\chi_{i_\tau}\}$. Then $\chi_{i_1}(1)=\cdots=\chi_{i_\tau}(1)$ and $m_{\chi_{i_1}}=\cdots=m_{\chi_{i_\tau}}=m_i$, where $m_{\chi_{i_j}}$ is the Schur index of $\chi_{i_j}$. Furthermore, we have
\begin{eqnarray*}
\varepsilon_i&=&e_{i_1}+\cdots+e_{i_\tau}\\
&=&\frac{\chi_{i_1}(1)}{|G|}\sum_{j=1}^n (\chi_{i_1}+\cdots+\chi_{i_\tau})(g_j^{-1})\hat{C}_j.
\end{eqnarray*}
By Theorem 6.2 of Chapter 2 in \cite{NT},  $m_i(\chi_{i_1}+\cdots+\chi_{i_\tau})$ is afforded by an irreducible module $\tilde{V}$ of $G$ over $F$. By Corollary 2 of Chapter 12.4 in \cite{S}, $m_i(\chi_{i_1}+\cdots+\chi_{i_\tau})$ is constant on $F$-classes of $G$, so does $(\chi_{i_1}+\cdots+\chi_{i_\tau})$. This implies that \[\varepsilon_i=\sum_{i=1}^n \ell_i\hat{{C}}_i,\] where $\ell_i\in k$.
So $\varepsilon_i(i=1,\cdots,\nu)$ is in the vector space $\oplus_{i=1}^n k\hat{{C}}_i$. By comparing the dimension, we have $\varepsilon_k(G)=\oplus_{i=1}^n k\hat{{C}}_i$.

\end{proof}

\begin{corollary}\label{lem:kclass}
For $1\leq i,j \leq n$, we have
\[\hat{{C}}_i\hat{{C}}_j=\sum_{l=1}^s{a}_{l i j}\hat{{C}}_l, \] where ${a}_{l i j}$ are rational integers.

\end{corollary}
\begin{proof}
Since $\hat{{C}}_i,\hat{{C}}_j$ are in the subalgebra $\varepsilon_k(G)$ and $\hat{{C}}_1,\cdots, \hat{{C}}_s$ is a basis of $\varepsilon_k(G)$, it follows that
\[\hat{{C}}_i\hat{{C}}_j=\sum_{l=1}^s{l}_{l i j}\hat{{C}}_l, \] where ${l}_{l i j}$ is in $F$. It follows from that the number of time that an element of $\hat{{C}}_l$ appears in the product $\hat{{C}}_i\hat{{C}}_j$ must be rational integer.
\end{proof}

\begin{example}
The $k$-classes $\mathcal{P}_1,\cdots,\mathcal{P}_n$  consist of a good partition of Cl($G$) by Corollary \ref{lem:kclass}.
\end{example}


In the next, we will
realize the irreducible representations of $k\mathcal{P}$ as the points of an affine variety. This variety is used to determine character variety in  section \ref{sec:realization}.

Let $\mathcal{A}$ be a commutative algebra over $k$. Assume that $\mathfrak{a}_1,\cdots,\mathfrak{a}_n$ is a basis of $\mathcal{A}$ such that $\mathfrak{a}_i\mathfrak{a}_j=\sum_{l=1}^na_{lij}\mathfrak{a}_l$. Define matrices $\mathcal{A}_j=(a_{lij})_{1\leq l,i\leq n}$ for $j=1,\cdots,n$. Then $\mathfrak{L}:\mathcal{A}\rightarrow M_n(k)$ is the regular representation of $\mathcal{A}$, where $\mathfrak{L}(\mathfrak{a}_j)=\mathcal{A}_j$.
If $k$ is large enough then  for $1\leq i \leq n$ there exists a common invertible matrix $\mathcal{U}$   such that
\[\mathcal{U}\mathcal{A}_i\mathcal{U}^{-1}=\begin{pmatrix}
\lambda_{i1}& *& *&\cdots &*\\
0& \lambda_{i2}& *&\cdots &*\\
\cdot& \cdot& \cdot&\cdots &\cdot\\
0& 0& 0&\cdots &\lambda_{in}\\
\end{pmatrix}.\leqno{(*)}\]
Then the map $\mathfrak{r}_j:\mathcal{A}\longrightarrow k$  sending $\mathfrak{a}_i$ to $\lambda_{ij}$ is an irreducible representation of $\mathcal{A}$ for $1\leq j\leq n$.
Let $\mathfrak{P}$ be a matrix such that $\mathfrak{P}=(\mathfrak{p}_{ij})$ with $\mathfrak{p}_{ij}=\mbox{Tr}(\mathcal{A}_i\mathcal{A}_j)$. Define $\mathcal{R}$ to be the matrix with $\mathfrak{r}_i(\mathfrak{a}_j)$ in the $(ij)$-entry of $\mathcal{R}$. Then $\mathcal{R}'\mathcal{R}=\mathfrak{P}$.

Let $\gamma_t=(\lambda_{1t},\cdots,\lambda_{nt})$, where $1\leq t\leq n$. Since $\tau_t$ is a representation,
$\gamma_t$ satisfies the following series of equations:
\begin{equation}\label{equ:RegRep}x_jx_l=\sum_{i=1}^n a_{ijl} x_i,1\leq j,l\leq n.\end{equation}
The common solutions of these equations consist of an affine variety. It is called representation variety of $\mathcal{A}$ and we denote it by $V_k(\mathcal{A})$.

Let $\gamma_t(x)=\sum_{i=1}^n \lambda_{it}x_i$ for $1\leq t \leq n$.

\begin{theorem}\label{Property:A}
Let $\mathfrak{A}=\sum_{i=1}^n x_i\mathcal{A}_i=(\alpha_{ij})$, where $\alpha_{ij}=\sum_{l=1}^n a_{ijl}x_l$. Then
\begin{enumerate}
\item[(i)]{$\mid \mathfrak{A}\mid=\prod_{t=1}^n\gamma_t(x)=\prod_{t=1}^n(\sum_{i=1}^n\lambda_{it}x_i)$.}
\item[(ii)]{$\gamma_t(x)$ is the vector of eigenvalues of $\mathfrak{A}$ and $\gamma_t(x)\gamma_t=\gamma_t\mathfrak{A}$ for $1\leq t\leq n$.}
\item[(iii)]{$\gamma_1,\cdots,\gamma_n$ are determined by the following equations:
\begin{equation}\label{equ:charII} x_{it}\gamma_t(x)=\sum_j\alpha_{ji}x_{jt}.\end{equation} }
\end{enumerate}


\end{theorem}

\begin{proof}
 It follows from (*) that $\gamma_1(x),\cdots,\gamma_n(x)$ are eigenvalues of the matrix $\mathfrak{A}$. Then $\mid \mathfrak{A}\mid=\prod_{t=1}^n\gamma_t(x)=\prod_{t=1}^n(\sum_{i=1}^n\lambda_{it}x_i)$.

By equations (\ref{equ:RegRep}), we have
\begin{equation*}
\begin{split}
\lambda_{jt}\gamma_t(x)=\sum_{l=1}^n\lambda_{jt}\lambda_{lt}x_l=\sum_{l=1}^n\sum_{i=1}^n a_{ijl}\lambda_{it}x_l=\sum_{i=1}^n(\sum_{l=1}^n a_{ijl}x_l)\lambda_{it}=\sum_{i=1}^n\alpha_{ij}\lambda_{it}.
\end{split}
\end{equation*}
So $\gamma_t(x)\gamma_t=\gamma_t\mathfrak{A}$ for $1\leq t\leq n$.


\end{proof}

\begin{corollary}\label{cor:rvariety}
The variety $V_k(\mathcal{A})$ consists of $\gamma_1,\cdots,\gamma_n$ and $\gamma_0=(0\cdots 0)$.
\end{corollary}

\begin{proof}
We have shown that
$\gamma_1,\cdots,\gamma_n$ are solutions of equations (\ref{equ:RegRep}).  Assume $0\not=\epsilon=(\epsilon_1,\cdots,\epsilon_n)$ is a solution of equations (\ref{equ:RegRep}). Let $\epsilon(x)=\sum_{i=1}^n\epsilon_i x_i$. Then we have $(\epsilon_1\epsilon(x),\cdots,\epsilon_n\epsilon(x))=(\epsilon_1,\cdots,\epsilon_n)\mathfrak{A}$ i.e. $(\epsilon_1,\cdots,\epsilon_n)(\epsilon(x) I-\mathfrak{A})=0$. So $|\epsilon(x) I-\mathfrak{A}|=0$. This implies that $\epsilon(x)$ is a root of the equation:
\begin{equation}\label{eleDeterm} |xI-\mathfrak{A}|=\prod_{i=t}^n (x-\sum_{i=1}^n\lambda_{it}x_i) .\end{equation} So $\epsilon(x)$ is equal to some $\sum_{i=1}^n\lambda_{it}x_i$ by Theorem \ref{Property:A}. And so $\gamma_t=(\epsilon_1,\cdots,\epsilon_n)$. This means that equations (\ref{equ:RegRep}) has exactly $n$ nonzero solutions: $\gamma_1,\cdots,\gamma_n$.
\end{proof}

\begin{theorem}\label{thm:crsemi}
The following conditions are equivalent:
\begin{enumerate}
\item[(i)]{ The commutative algebra $\mathcal{A}$ is semisimple.}
\item[(ii)]{There is no non-zero nilpotent element in $\mathcal{A}$ .}
\item[(iii)]{The matrix $\mathfrak{P}$ is invertible.}
\end{enumerate}
\end{theorem}
\begin{proof} If $\mathcal{A}$ is semisimple, then the radical of $\mathcal{A}$ is zero. If there exists a non-zero element $a$ in $\mathcal{A}$ is nilpotent, then the ideal $I$ generated by $a$ which is  non-zero nilpotent. It is contradicting to the semisimplicity of $\mathcal{A}$. If there is no non-zero nilpotent element in $\mathcal{A}$, then the radical of $\mathcal{A}$ is zero, since the radical $J(\mathcal{A})$ is nilpotent. So (i) and (ii) are equivalent.

Now we show that (i) is equivalent to (iii). If $\mathcal{A}$ is semisimple and the field $k$ is finite or it is of characteristic zero, then $\mathcal{A}$ is separable. So we can assume that $k$ is large enough for $\mathcal{A}$. Then $\mathfrak{P}=\mathcal{R}'\mathcal{R}$ as above. If $\mathfrak{P}$ is not invertible, then there exists a nonzero vector $(\ell_1,\cdots,\ell_n)$ such that $\mathcal{R}(\ell_1,\cdots,\ell_n)'=0$. This implies that $\mathcal{U}\sum_{i=1}^n\ell_i\mathcal{A}_i\mathcal{U}^{-1}$ is of the form:
\[\begin{pmatrix}
0& *& *&\cdots &*\\
0& 0& *&\cdots &*\\
\cdot& \cdot& \cdot&\cdots &\cdot\\
0& 0& 0&\cdots &0\\
\end{pmatrix}.\]

So the element $\sum_{i=1}^n\ell_i\mathcal{A}_i=\mathfrak{L}(\sum_{i=1}^n \ell_i\mathfrak{a}_i)$ is nilpotent. This implies that $$\sum_{i=1}^n\ell_i\mathfrak{a}_i\not=0$$ is nilpotent. This contradicts the semisimplicity of $\mathcal{A}$.

If $\mathfrak{P}$ is invertible, then the bilinear form $(\mathfrak{a}_i,\mathfrak{a}_j)=\mathfrak{p}_{ij}$ on $\mathcal{A}$ is nonsingular. If there is a non-zero nilpotent element $\mathfrak{a}\in \mathcal{A}$, then $(\mathfrak{a},\mathfrak{b})$=Tr($\mathfrak{L}(\mathfrak{b}\mathfrak{a}))=0$ for any $\mathfrak{b}\in \mathcal{A}$. This contradicts  the non-singularity of the bilinear form. So (i) and (iii) are equivalent.
\end{proof}

\begin{corollary} If $k$ is a splitting field of $\mathcal{A}$ and $\mathcal{A}$ is semisimple then $\mathfrak{r}_1,\cdots,\mathfrak{r}_n$ are all different irreducible representations of $\mathcal{A}$.
\end{corollary}

\begin{proof}
By Theorem \ref{thm:crsemi}, $\mathfrak{P}=\mathcal{R'R}$ is invertible. So all rows of $R$ are linearly independent. Since the $i$-row determines the representation $\mathfrak{r}_i$, it follows that $\mathfrak{r}_1,\cdots,\mathfrak{r}_n$ are all non-isomorphic irreducible representations of $\mathcal{A}$.
\end{proof}

In the rest of this section, we will show that all partition algebras are semisimple.
Let $\mathcal{P}=\{\mathcal{P}_i\}_{i=1}^n$ be a good partition of Cl$(G)$ i.e.
Let $\hat{C}_1,\cdots,\hat{C}_n$ be  class sums of the partition $\mathcal{P}$. Denotes by $\ell_i$ the number of elements in $\mathcal{P}_i$. Let $\mathcal{P}_{i'}$ be the set consisting of the inverse of elements in $\mathcal{P}_i$. Then
$\ell_{i'}=\ell_i$.  Denotes by $\ell_{i_1i_2\cdots i_r}$ the number of solutions of the equation $g_{i_1}g_{i_2}\cdots g_{i_r}=1$ with $g_{i_j}\in \mathcal{P}_{i_j}$ for $1\leq j \leq r$.

\begin{lemma}\label{basicEqs}
The following equations hold:

\begin{enumerate}
\item[(i)]{$\ell_{i_1i_2\cdots i_r}=\ell_{i_1' i_2'\cdots i_r'}$.}
\item[(ii)]{$\ell_{i_1i_2\cdots i_r}=\ell_{i_1^\tau i_2^\tau\cdots i_r^\tau}$,
where $\tau$ is a permutation on the set $\{i_1,i_2,\cdots ,i_r\}$.}
\item[(iii)]{$\frac{1}{\ell_{j'}}\ell_{j'k_1\cdots k_t}=a_{jk_1\cdots k_t}$ where $a_{jk_1\cdots k_t}$ is the coefficient of the equation $$\hat{C_{k_1}}\hat{C_{k_2}}\cdots \hat{C_{k_t}}=\sum_{j=1}^n a_{jk_1\cdots k_t}\hat{C_j}.$$}
\item[(iv)]{$\sum_{i_{r+1}=1}^n\ell_{i_1i_2\cdots i_ri_{r+1}}=\ell_{i_1}\cdots \ell_{i_r}$.}
\item[(v)]{$\ell_{i_1i_2\cdots i_ai_{a+1}\cdots i_r}=\sum_{j=1}^n \frac{1}{\ell_j} \ell_{i_1i_2\cdots i_aj}\ell_{j'i_{a+1}\cdots i_r}.$}
\end{enumerate}

\end{lemma}
\begin{proof}
We will prove above equations in case $r=3$ or $r=4$. The general cases follows from induction on $r$.
Since
$abc=1$ is equivalent to $cab=1$ or $bca=1$,
we have $\ell_{ijk}=\ell_{kij}=\ell_{jki}$.
There is a bijection
$P_iP_j \to P_jP_i$ given by $(a,b) \mapsto (b, b^{-1}ab)$ such that the
products coincide. So we have
$\ell_{ijk}=\ell_{jik}$. There is also an obvious bijection $P_{i_1} \dots P_{i_t}$
to $P_{i^\prime_t} \dots P_{i^\prime_1}$ by taking inverses.
Let $S_{\{i,j,k\}}$ be the symmetric group on $\{i,j,k\}$ which is generated by $(i,j)$ and $(i,j,k)$. By applying above discussions repeatedly, we have
$\ell_{ijk}=\ell_{i^\sigma j^\sigma k^\sigma}=\ell_{k'j'i'}=\ell_{i'j'k'}$, for $\sigma\in S_{\{i,j,k\}}$.

It follows from the definition that $\ell_{ijk}=|\{(a,b,c)\mid a\in \mathcal{P}_i,b\in \mathcal{P}_j,c\in \mathcal{P}_k,abc=1\}|$.
Since $abc=1$ is equivalent to $bc=a^{-1}$, we have
$\ell_{ijk}=|\{(b,c)\mid b\in \mathcal{P}_j,c\in \mathcal{P}_k,bc\in \mathcal{P}_{i'}\}|$.
Given $a\in \mathcal{P}_i$. Then $a_{i'jk}=|\{b\in \mathcal{P}_j,c\in \mathcal{P}_k|abc=1\}|$.
So $\ell_{ijk}=\ell_{i'} a_{i'jk}=\ell_{i} a_{i'jk}$.

Since $\{(a,b)|a\in \mathcal{P}_i,b\in \mathcal{P}_j\}=\bigsqcup_{k=1}^n \{(a,b)|a\in \mathcal{P}_i,b\in \mathcal{P}_j,ab\in \mathcal{P}_{k'}\}$,
we have $\ell_i\ell_j=\sum_{k=1}^{n} \ell_{ijk}$.

Given \ $x\in \mathcal{P}_\lambda$. Then the number of solutions of the equation $abcd=1,a\in \mathcal{P}_i,b\in \mathcal{P}_j,c\in \mathcal{P}_k,d\in \mathcal{P}_l$  afforded by equations $ab=x^{-1},cd=x,a\in \mathcal{P}_i,b\in \mathcal{P}_j,c\in \mathcal{P}_k,d\in \mathcal{P}_l$
is $\frac{\ell_{ij\lambda}}{\ell_\lambda}\frac{\ell_{\lambda'kl}}{\ell_\lambda}=\frac{\ell_{ij\lambda}\ell_{\lambda'kl}}{\ell_\lambda^2}$.
When $x$ runs over  $G$, it follows
$$\ell_{ijkl}=\ell_\lambda a_{ijkl}=\sum_{\lambda=1}^{n}\frac{\ell_{ij\lambda}\ell_{\lambda'kl}}{\ell_\lambda} .$$

\end{proof}

Since $\hat{{C}_j}\hat{{C}_l}=\hat{{C}_l}\hat{{C}_j}$, we have $a_{jk_1\cdots k_t}=a_{jk_1^\tau\cdots k_t^\tau}$ for each permutation $\tau$ on $k_1,\cdots,k_t$. Define a matrix
$\mathcal{M}_j=(a_{klj})$ of degree $n$ for each $1\leq j \leq n$. Then $\mathcal{L}:k\mathcal{P}\rightarrow M_n(k)(\hat{C}_j\mapsto \mathcal{M}_j)$ is a regular representation of the partition algebra $k\mathcal{P}$. So $\mathcal{M}_i\mathcal{M}_j=\mathcal{M}_j\mathcal{M}_i$ for $1\leq i,j\leq n$. Assume that $\mathcal{L}(\hat{C_1})$ is the identity matrix.

\begin{lemma}\label{property:p} Let $p_{ij}=\mbox{\rm Tr}(\mathcal{M}_i\mathcal{M}_j)$ and $p_i=p_{i1}=p_{1i}=\sum_{t=1}^n\frac{\ell_{itt'}}{\ell_t}$ for $1\leq i,j\leq n$. Then
\begin{equation}\label{equ:pij}
p_{ij}=\sum_{t=1}^n \frac{1}{\ell_t}\ell_{ijtt'}=\sum_{l,t=1}^na_{ttl}a_{lij}=\sum_{l=1}^n\frac{1}{\ell_l}\ell_{ijl}p_l.
\end{equation}
\begin{equation}\label{equ:pij'}
p_{ij'}=\sum_{s,t=1}^n \frac{\ell_{sti}\ell_{stj}}{\ell_s\ell_t}=p_{ji'}=p_{i'j}=p_{j'i}.
\end{equation}

\end{lemma}

\begin{proof}
By definition, $p_{ij}=\sum_{s,t=1}^na_{sti}a_{tsj},p_i=p_{i'}$. By (3) of Lemma \ref{basicEqs}, we can replace $a_{sti},a_{tsj}$ by $\frac{1}{\ell_{s}}\ell_{s'ti}$ and $\frac{1}{\ell_{t}}\ell_{t'sj}$  respectively in the definition of $p_{ij}$ and we get
\begin{equation*}
\begin{split}
p_{ij}=&\sum_{s,t=1}^n\frac{\ell_{s'ti}\ell_{t'sj}}{\ell_{s}\ell_{t}}
=\sum_{t=1}^n\frac{1}{\ell_{t}}\sum_{s=1}^n\frac{\ell_{s'ti}\ell_{t'sj}}{\ell_{s}}
=\sum_{t=1}^n\frac{1}{\ell_{t}}\ell_{ijtt'}\\
=&\sum_{t=1}^n\sum_{l=1}^n\frac{1}{\ell_{l}}\ell_{ijl'}\frac{1}{\ell_{t}}\ell_{ltt'}
=\sum_{t=1}^n\sum_{l=1}^n a_{ttl}a_{lij}\\
=&\sum_{l=1}^n\frac{1}{\ell_{l}}\ell_{ijl'}\sum_{t=1}^n\frac{1}{\ell_{t}}\ell_{ltt'}=\sum_{l=1}^n\frac{1}{\ell_{l}}\ell_{ijl'}p_l.
\end{split}
\end{equation*}
Furthermore,
\begin{equation*}
\begin{split}
p_{ij}=&\sum_{t=1}^n\frac{1}{\ell_{t}}\ell_{ijtt'}=\sum_{l=1}^n\frac{1}{\ell_{l}}\ell_{ijl}\sum_{t=1}^n\frac{1}{\ell_{t}}\ell_{l'tt'}\\
=&\sum_{l=1}^n\frac{1}{\ell_{l}}\ell_{ijl}p_{l'}=\sum_{l=1}^n\frac{1}{\ell_{l}}\ell_{ijl}p_{l}.
\end{split}
\end{equation*}
So equation (\ref{equ:pij}) holds.

When $s'$ runs over $n$ classes, so does $s$. So we can exchange $s$ and $s'$ in the following equation $p_{ij'}=\sum_{s,t=1}^n\frac{\ell_{s'ti}\ell_{t's j'}}{\ell_{s}\ell_{t}}$ and we get  $p_{ij'}=\sum_{s,t=1}^n\frac{\ell_{s ti}\ell_{t's'j'}}{\ell_{s}\ell_{t}}$. It follows from $l_{ts j}=l_{t's'j'}$ that $p_{ij'}=\sum_{s,t=1}^n\frac{\ell_{s ti}\ell_{ts j}}{\ell_{s}\ell_{t}}$, and the rest equations in (\ref{equ:pij'}) follows from it.
\end{proof}
\begin{lemma}\label{determNot0}

The determinant $\mid p_{ij}\mid_{1\leq i,j\leq n}$ is not zero.
\end{lemma}

\begin{proof}
By a series of swapping columns of $\mid p_{ij}\mid$, we have $|p_{ij'}|$. It suffices to show $|p_{ij'}|\not=0$.
Define \[M=\begin{pmatrix}\cdot& \cdot& \cdot&\cdots &\cdot\\
\frac{\ell_{ij-11}}{\sqrt{\ell_i\ell_{j-1}}}&\frac{\ell_{ij-12}}{\sqrt{\ell_i\ell_{j-1}}}&\frac{\ell_{ij-13}}{\sqrt{\ell_i\ell_{j-1}}}&
\cdots&\frac{\ell_{ij-1n}}{\sqrt{\ell_i\ell_{j-1}}}\\
\frac{\ell_{ij1}}{\sqrt{\ell_i\ell_j}}&\frac{\ell_{ij2}}{\sqrt{\ell_i\ell_j}}&\frac{\ell_{ij3}}{\sqrt{\ell_i\ell_j}}&
\cdots&\frac{\ell_{ijn}}{\sqrt{\ell_i\ell_j}} \\
\frac{\ell_{ij+11}}{\sqrt{\ell_i\ell_{j+1}}}&\frac{\ell_{ij+12}}{\sqrt{\ell_i\ell_{j+1}}}&\frac{\ell_{ij+13}}{\sqrt{\ell_i\ell_{j+1}}}&
\cdots&\frac{\ell_{ij+1n}}{\sqrt{\ell_i\ell_{j+1}}}\\
\cdot& \cdot& \cdot&\cdots &\cdot\end{pmatrix}\] to be a matrix with $n^2$ rows and $n$ columns, where $1\leq i,j\leq n$. By equation (5), it follows that
\[\mid p_{ij'}\mid= \mid M'M\mid.\] By Binet-Cauchy formula for determinant, we have
\[ \mid p_{ij'}\mid=\sum_{1\leq t_1<\cdots<t_n\leq n^2}\mid M\begin{pmatrix}t_1 & t_2&t_3&\cdots &t_n\\
1& 2& 3&\cdots &n\end{pmatrix}\mid^2.\] In particular, we have some \[M\begin{pmatrix}t_1 & t_2&t_3&\cdots &t_n\\
1& 2& 3&\cdots &n\end{pmatrix}=\begin{pmatrix}\frac{\ell_{11'1}}{\sqrt{\ell_1\ell_{1'}}}&\frac{\ell_{11'2}}{\sqrt{\ell_1\ell_{1'}}}&\frac{\ell_{11'3}}{\sqrt{\ell_1\ell_{1'}}}&
\cdots&\frac{\ell_{11'n}}{\sqrt{\ell_1\ell_{1'}}}\\
\frac{\ell_{12'1}}{\sqrt{\ell_1\ell_{2'}}}&\frac{\ell_{12'2}}{\sqrt{\ell_1\ell_{2'}}}&\frac{\ell_{12'3}}{\sqrt{\ell_1\ell_{2'}}}&
\cdots&\frac{\ell_{12'n}}{\sqrt{\ell_1\ell_{2'}}}\\
\cdot& \cdot& \cdot&\cdots &\cdot\\
\frac{\ell_{1n'1}}{\sqrt{\ell_1\ell_{n'}}}&\frac{\ell_{1n'2}}{\sqrt{\ell_1\ell_{n'}}}&\frac{\ell_{1n'3}}{\sqrt{\ell_1\ell_{n'}}}&
\cdots&\frac{\ell_{1n'n}}{\sqrt{\ell_1\ell_{n'}}}\\\end{pmatrix}.\] It is the diagonal matrix $\mbox{diag}(\sqrt{\ell_1},\cdots,\sqrt{\ell_n})$. So $\mid p_{ij'}\mid\not=0$.
\end{proof}


It is clear that the next Theorem follows from  Maschke's Theorem easily. But the constants introduced above are enough for the proof.
\begin{theorem}\label{thm:semisimPart}
All partition algebras are semisimple over a field $k$ with characteristic zero.
\end{theorem}
\begin{proof}
It follows from Theorem \ref{thm:crsemi} and Lemma \ref{determNot0}.
\end{proof}

\section{Polynomials of algebras}\label{sec:polyalg}

 Let $\mathcal{A}$ be an algebra over $k$ with identity. Let $\alpha_1,\cdots, \alpha_n$ be a basis of $\mathcal{A}$. Given $\alpha \in \mathcal{A}$, then $\alpha \alpha_i=\sum_{j=1}^n \ell_{ji}^\alpha \alpha_j$. The left regular representation $\mathfrak{L}$ of $\mathcal{A}$ is defined by $\mathfrak{L}(\alpha)=(\ell_{ji}^\alpha)_{1\leq j,i\leq n}$. A generic element $\alpha \in \mathcal{A}$ can be written in the form $\alpha=x_1\alpha_1+\cdots +\alpha_n$ and
$\mathfrak{L}(\alpha)=x_1\mathfrak{L}(\alpha_1)+\cdots +x_n\mathfrak{L}(\alpha_n)$, where $x_1,\cdots, x_n$ are indeterminates. The matrix $\lambda I - \mathfrak{L}(\alpha) $
is similar to \[\mbox{ Diag}(\epsilon_1(\lambda,x_1,\cdots,x_n),\cdots,\epsilon_n(\lambda,x_1,\cdots,x_n)),\] where the polynomials $\epsilon_i(\lambda,x_1,\cdots,x_n)$ is a factor of $\epsilon_{i+1}(\lambda,x_1,\cdots,x_n)$. The polynomial $\epsilon_n(\lambda,x_1,\cdots,x_n)$ is called the minimal polynomial of $\mathcal{A}$ with respect to $\alpha_1,\cdots,\alpha_n$ and it is denoted by $m_{\mathcal{A}}(\lambda,x_1,\cdots,x_n)$. The polynomial $\prod_{i=1}^r \epsilon_i(\lambda,x_1,\cdots,x_n)=|\lambda I-\mathfrak{L}(\alpha)|$ is called the characteristic polynomial of $\mathcal{A}$ with respect to $\alpha_1,\cdots,\alpha_n$ and it is denoted by $c_{\mathcal{A}}(\lambda,x_1,\cdots,x_n)$. The polynomial $(-1)^nc_{\mathcal{A}}(0,x_1,\cdots,x_n)$ is called generic polynomial of $\mathcal{A}$ with respect to the basis $\alpha_1,\cdots,\alpha_n$ and it is denoted by $g_\mathcal{A}(x_1,\cdots,x_n)$.
Elements of $\mathcal{A}$ are annihilated by $m_{\mathcal{A}}(\lambda,x_1,\cdots,x_n)$ and $c_{\mathcal{A}}(\lambda,x_1,\cdots,x_n)$. These polynomials have the following forms respectively
\begin{eqnarray*}&&c_{\mathcal{A}}(\lambda,x_1,\cdots,x_n)=|\lambda I -M|\\
&=&\lambda^n-\lambda^{n-1}\tau_1(x_1,\cdots,x_n)+\cdots + (-1)^n \tau_n(x_1,\cdots,x_n),\end{eqnarray*}
\[m_{\mathcal{A}}(\lambda,x_1,\cdots,x_n)=\lambda^r-\lambda^{r-1}\rho_1(x_1,\cdots,x_n)+\cdots + (-1)^r \rho_r(x_1,\cdots,x_n),\] where $\rho_i(x_1,\cdots,x_n)$  and $\tau_i(x_1,\cdots,x_n)$ are homogenous polynomial with degree $i$. The polynomial $\rho_r(x_1,\cdots,x_n)$ is usually called norm form of $\mathcal{A}$.

Given $\beta=\sum_{i=1}^n d_i\alpha_i \in \mathcal{A}$. Then $c_{\mathcal{A}}(\lambda,d_1,\cdots,d_n)=|\lambda I-\mathfrak{L}(\beta)|$ is called the characteristic polynomial of $\beta$ and it is denoted by $c_\beta(\lambda)$.

Let $\mathcal{A}$ be a group algebra of a finite group $G=\{g_1,\cdots,g_\ell\}$. Then then generic polynomial $(-1)^\ell g_\mathcal{A}(x_{g_1},\cdots,x_{g_\ell})$ with respect to basis $g_1,\cdots,g_\ell$ is called group determinant of $G$ over $k$. The group determinant can also be presented as

\begin{definition}\cite{F2}
Let $G=\{g_1,\cdots, g_\ell\}$ be a finite group with $g_1$ as identity of $G$. Let $x_{g_1},\cdots,x_{g_\ell}$ be variables indexed by elements of $G$. Then the matrix
$$M_G( x)=\begin{pmatrix}
x_{g_1g_1^{-1}}&x_{g_1g_2^{-1}}&\cdots &x_{g_1g_\ell^{-1}}\\
\cdots&\cdots&\cdots&\cdots\\
x_{g_\ell g_1^{-1}}&x_{g_\ell g_2^{-1}}&\cdots &x_{g_\ell g_\ell^{-1}}
\end{pmatrix}$$ is called the matrix of $G$, the determinant $\Theta_G(x)$ of $M_G(x)$ is called group determinant of $G$.
\end{definition}

Let $c_\mathcal{A}(\lambda,x_1,\cdots,x_n)$ be the characteristic polynomial of $\mathcal{A}$ with respect to the basis $\alpha_1,\cdots,\alpha_n$ and let $c_\mathcal{A}(\lambda,y_1,\cdots,y_n)$ be  the characteristic polynomial of $\mathcal{A}$ with respect to the basis $\beta_1,\cdots,\beta_n$. Assume that
$(\beta_1,\cdots,\beta_n)=(\alpha_1,\cdots,\alpha_n)A$ for an invertible matrix $A$.

\begin{theorem}\label{thm:bccpoly}
Let $A^{-1}=(a'_{ij})$. Then the polynomial $c_\mathcal{A} (\lambda,x_1,\cdots,x_n)$ is equal to the polynomial $c_\mathcal{A}(\lambda,y_1,\cdots,y_n)$ after replacing $y_i$ by $a'_{i1}x_1+\cdots +a'_{in}x_n$ for $1\leq i \leq n$.
\end{theorem}

\begin{proof}
Let $(y_1,\cdots,y_n)=A^{-1}x$, where $x=(x_1,\cdots,x_n)'$. We denote $x_1\alpha_1+\cdots +x_n\alpha_n$ and $y_1\beta_1+\cdots +y_n\beta_n$ by $x_\alpha$ and $y_\beta$ respectively. Then $y_\beta=(\alpha_1,\cdots,\alpha_n)AA^{-1}x=x_\alpha$. Assume that $(x_\alpha\alpha_1,\cdots,x_\alpha\alpha_n)=(\alpha_1,\cdots,\alpha_n)M_\alpha$ and $(y_\beta\beta_1,\cdots,y_\beta\beta_n)=(\beta_1,\cdots,\beta_n)M_\beta$. Then $M_\beta=A^{-1}M_\alpha A$. So we have
$f(\lambda,x_1,\cdots,x_n)=|\lambda I-M_\alpha|=|\lambda I-M_\beta|=g(\lambda,y_1,\cdots,y_n)$.

\end{proof}
\begin{remark}
As in Theorem \ref{thm:bccpoly}, we have similar relations among the polynomials: $\tau_j(x_1,\cdots,x_n)$ and  $\tau_j(x_1,\cdots,x_n)$, $\rho_j(x_1,\cdots,x_n)$ and $\rho_j(y_1,\cdots,y_n)$, $m_{\mathcal{A}}(\lambda,x_1,\cdots,x_n)$ and $m_{\mathcal{A}}(\lambda,y_1,\cdots,y_n)$.

\end{remark}

\section{Defining equations of $\mathcal{P}$-characters}\label{sec:defequ}

In this section,  we will prove uniform definition equations for various characters. For the reader's conveniences, we will reformulate some basic properties of the group determinant and its irreducible factors in \cite{F1} and \cite{F2}.

Let $\mathcal{P}$ be a good partition of Cl($G$) with the identity contained in the partition class $\mathcal{P}_1$. Let $\hat{C}_1$ be the class sum of $\mathcal{P}_1$. To simplify the discussion, we will assume that $\lambda \hat{C}_1$ is the identity of $k\mathcal{P}$ for some $\lambda\in k$ in the rest.

Let $\mathscr{L}:kG\rightarrow M_\ell(k)$ be the left regular representation of $kG$.  Then  the matrix $M_G(x)$ of $G$ is the generic element $\sum_{g\in G} \mathscr{L}(g)x_g$ in $\mathscr{L}(kG)$. We assume that $k$ is the splitting field of $kG$ and its center $Z(kG)$. Let $\mbox{Cl}(G)=\{\mathscr{C}_i\}_{i=1}^s$ as in Section \ref{sec:Frobenius theory}. And $\mathfrak{t}_i$ is the representative of $\mathscr{C}_i$. Then there are $s$ irreducible representations $\mathfrak{l}_1,\cdots,\mathfrak{l}_s$ of $Z(kG)$ over $k$ and the degree of each $\mathfrak{l}_i$ is $1$. Since $Z(kG)$ is semisimple, the restriction of $\mathscr{L}$ to $Z(kG)$ is equivalent to $\oplus_{i=1}^s m_i \mathfrak{l}_i$, where $m_i$ is a positive integer.

Let $M$ be a square matrix. We will denote the determination of $M$ by $|M|$.
For elements $P,Q\in G$, we will denote $M_G(x)$ by $(x_{PQ^{-1}})$ and  denote the function $\Phi(x_{g_1},x_{g_2},,\cdots,x_{g_\ell})$ by $\Phi(x_R)$ or $\Phi(x)$.

\begin{definition}
We say that a matrix $M$ over a ring $S$ has  the same symmetry as $M_G(x)$ if $M$ is a specialization of $M_G(x)$ i.e. there is a map $f: \{x_g|g\in G\}\rightarrow S$ with the property that $M=f(M_G(x))=(f(x_{g_ig_j^{-1}}))$, where $S$ is set of entries of $M$.

\begin{example}\label{exa:symmetry}
Assume that the matrix $Y=(y_{g_ig_j^{-1}})_{g_i,g_j\in G}$ satisfies $y_{g_ig_j}=y_{g_jg_i}$ i.e. if $g_i$ is conjugate to $g_l$ then $y_{g_i}=y_{g_j}$.  Then there are $s$ independent variables $\{y_{d_i}\}_{i=1}^s$, where $d_i$ is the representative of $\mathscr{C}_i$. Let $S=\{y_{d_i}:i=1,\cdots s\}$ and we define $f:\{x_g|g\in G\}\rightarrow S$ by $f(x_g)=y_{d_i}$ if $g$ is conjugate to $d_i$. Then $Y=f(M_G(x))=(f(x_{g_ig_j^{-1}}))$. So $Y$ has the same symmetry as $M_G(x)$. In fact, $Y$ is also equal to $\sum_{i=1}^s \mathscr{L}(\hat{\mathscr{C}_i})y_{d_i}$.

\end{example}

\begin{lemma}\label{lem:detofxy=yx} Let $Y=(y_{g_ig_j^{-1}})_{g_i,g_j\in G}$ be as in above example. Then $|\lambda I_\ell - Y|=\prod_{j=1}^s (\lambda -\sum_{i=1}^s\mathfrak{l}_j(\hat{\mathscr{C}}_i)y_{d_i})^{m_j}$.
\end{lemma}
\begin{proof}
The matrix $Y$ can be written into the form $\sum_{i=1}^s \mathscr{L}(\hat{\mathscr{C}_i})y_{d_i}$. Since $\mathscr{L}$ to  is equivalent to $\oplus_{i=1}^s m_i \mathfrak{l}_i$, there is an invertible matrix $T$ such that \[T\mathscr{L}(\hat{\mathscr{C}_i})T^{-1}=Diag(\mathfrak{l}_1^{m_1}(\hat{\mathscr{C}_i}),\cdots,\mathfrak{l}_s^{m_s}(\hat{\mathscr{C}_i}))\]
for $1\leq i \leq s$. So \[|\lambda I_\ell -Y|=|\lambda I_\ell -\sum_{i=1}^s T\mathscr{L}(\hat{\mathscr{C}_i})T^{-1}y_{d_i}|=\prod_{j=1}^s (\lambda -\sum_{i=1}^s\mathfrak{l}_j(\hat{\mathscr{C}}_i)y_{d_i})^{m_j}.\]
\end{proof}

\end{definition}

Let $$\Theta_G(x)=|(x_{PQ^{-1}})|=\prod_i\Phi_i^{e_i}(x)$$ be a factorization of irreducible polynomials over the complex field $\mathbb{C}$.
Let $f_i$ be the degree of $\Phi_i(x)$. Then the coefficient of $x_{g_1}^{f_i}$ in $\Phi_i(x)$ is $1$(see section 2 in \cite{F2}). Denote by $\chi_i(g_j)$ the coefficient of $x_{g_1}^{f_i-1}x_{g_j}$ in $\Phi_i(x)$ for $g_j\neq g_1$ and set $\chi_i(g_1)=f_i$. In this way, we define a function \[\chi_i:G\rightarrow \mathbb{C}\mbox{ }(g\mapsto \chi_i(g))\] for each irreducible factor $\Phi_i(x)$ of $\Theta_G(x)$. As in \cite{F1}, we call this function $\chi_i$ the character of $\Phi_i(x)$.

Let $(z_{PQ^{-1}})=(x_{PQ^{-1}})(y_{PQ^{-1}})$, where $z_{PQ^{-1}}=\sum_{R\in G} x_{PR^{-1}}y_{RQ^{-1}}$.  Then  \[\Theta_G(z)=\Theta_G(x)\Theta_G(y).\] Furthermore, the matrix $(z_{PQ^{-1}})$ has the same symmetry as $M_G(x)$. So $$\Theta_G(z)=|(z_{PQ^{-1}})|=\prod_i\Phi_i^{e_i}(z).$$ With these notations, we have the following

\begin{lemma}\label{lem:composable}
 Let $\Phi(x)$ be a homogenous polynomial in $k[x_{g_1},\cdots,x_{g_\ell}]$ with $\Phi(\varepsilon)=1$, where $\varepsilon=(1,0,\cdots,0)$. Then $\Phi(x)$ is a product of powers of irreducible factors of $\Theta_G(x)$ if and only if
  $\Phi(z)=\Phi(x)\Phi(y)$.
\end{lemma}
\begin{proof}
At first we assume that $\Phi(x)$ is irreducible factors of $\Theta_G(x)$. By above equations, we have
  $$\prod_i\Phi_i^{e_i}(z)=\prod_i\Phi_i^{e_i}(x)\Phi_i^{e_i}(y).$$
  This implies that $\Phi(z)=\Lambda(x)\Delta(y)$, where $\Lambda(x),\Delta(y)$ are products of those $\Phi_i(x)$ and $\Phi_j(y)$ respectively.
  Set $y_R=\delta_{ER}$. Then $z_R=\sum_S x_{RS^{-1}}y_S=x_R$ and $\Phi(z)=\Phi(x)=\Lambda(x)\Delta(\varepsilon)$, where $\varepsilon=(1,0,\cdots,0)$. Similarly, we have $\Phi(y)=\Lambda(\varepsilon)\Delta(y)$. Hence $$\Phi(z)=\Lambda(x)\Delta(y)=\Lambda(x)\Delta(\varepsilon)\Delta(y)\Lambda(\varepsilon).$$ Since
  $\Lambda(\varepsilon)\Delta(\varepsilon)=\Phi(\varepsilon)=1$, we have $\Phi(z)=\Phi(x)\Phi(y)$. So we have $\Phi(z)=\Phi(x)\Phi(y)$ in the case that $\Phi(x)$ is a product of powers of irreducible factors of $\Theta_G(x)$.

   If $\Phi(\varepsilon)=1$ and $\Phi(z)=\Phi(x)\Phi(y)$, then we take the matrix $(y_{PQ^{-1}})$ to be the adjoint $(x_{PQ^{-1}})^*$ of the matrix $(x_{PQ^{-1}})$. So
   \begin{eqnarray*}
     (z_{PQ^{-1}})&=&(x_{PQ^{-1}})(y_{PQ^{-1}})=(x_{PQ^{-1}})(x_{PQ^{-1}})^\ast\\
     &=&\mbox{Diag}(\Theta_G(x),\cdots,\Theta_G(x)) =\Theta_G(x)I.
   \end{eqnarray*} Then $\Phi(z)=\Phi(\Theta_G(x),0\cdots,0)=\Theta_G^f(x)$, where $f$ is the degree of $\Phi(x)$. On the other hand, we have
   $\Phi(z)=\Phi(x)\Phi(y)$. So $\Phi(x)$ is a factor of $\Theta_G^f(x)$ and it is  a product of powers of irreducible factors of $\Theta_G(x)$.

  \end{proof}

\begin{theorem}\label{thm:classfunction}\cite{F2}
Let $\chi$ be the character of an irreducible factor $\Phi(x)$ of $\Theta_G(x)$. Then $\chi(AB)=\chi(BA)$ for $A,B\in G$.
\end{theorem}
\begin{proof}
Let $(y_{g_1},\cdots,y_{g_\ell})=(0,\cdots,0,y_{g_s},0,\cdots,0)$, i.e. $y_R=\delta_{A,R}y_A$ in the matrix $(y_{PQ^{-1}})$, where $g_s=A$. Then $(y_{PQ^{-1}})$ is the product of $y_A$ and a permutation matrix $\mathscr{L}(A)$. So $$\Theta_G(0,\cdots,y_A,\cdots,0)=\Phi_1(y_A)\Phi_2(y_A)\cdots=|(y_{PQ^{-1}})|=\varepsilon y_A^\ell,$$ where $\varepsilon$ is $1$ or $-1$, $\ell$ is the order of $G$. Let $\Phi(y)$ be any irreducible factor of $\Theta_G(y)$. We also take $y_R=\delta_{A,R}y_A$ in $\Phi(y)$. Then $\Phi(y_A)=\theta(A)y_A^f\not=0$, where $f$ is the degree of $\Phi(z)$ and $\theta(A)$ is a constant. By Lemma \ref{lem:composable}, $\Phi(z)=\Phi(x)\Phi(y)$. We take $y_R=\delta_{A,R}y_A$ in this equation. Then $z_R=\sum_Bx_{RB^{-1}}y_B=x_{RA^{-1}}y_A$. Since $\Phi(z)$ is a homogenous polynomial, we have $\Phi(x_ {RA^{-1}}y_A)=y_A^f\Phi(x_{RA^{-1}})$ and
\[\Phi(z_R)=\Phi(x_{RA^{-1}}y_A)=y_A^{f}\Phi(x_{RA^{-1}})=\theta(A)\Phi(x_R)y_A^f.\] It follows the equation
$\Phi(x_{RA^{-1}})=\theta(A)\Phi(x_R)$. Similarly, $$\Phi(x_{A^{-1}R})=\theta(A)\Phi(x_R).\leqno(\star)$$
Given $B\in G$. Replacing $x_R$ by $x_{B^{-1}R}$ in equation ($\star$), it follows
$$\Phi(x_{B^{-1}A^{-1}R})=\theta(A)\Phi(x_{B^{-1}R})=\theta(A)\theta(B)\Phi(x_R).$$
Replacing $A$ by $AB$ in ($\star$), it follows
$\Phi(x_{B^{-1}A^{-1}R})=\theta(AB)\Phi(x_R).$
So we have  $\theta(AB)=\theta(A)\theta(B)$. In particular, $\theta(A)\theta(A^{-1})=1$.

Replacing $x_R$ by $x_{RA}$ in ($\star$) it follows
$$\Phi(x_{A^{-1}RA})=\theta(A)\Phi(x_{RA})=\theta (A) \theta(A^{-1}) \Phi(x_R) $$
This means
$$ \Phi(x_{A^{-1}RA})=\Phi(x_R).$$
Comparing the coefficients of $x_E^{f-1}x_{A^{-1}RA}$ on both sides of the last equation, it follows
$ \chi(AB)=\chi(BA), $ since $AB$ and $BA$ are conjugate in $G$.

\end{proof}

\begin{lemma}\label{lem:commut} \cite{F2}
Let $y_{g_1},\cdots,y_{g_\ell}$ be another system of variables. Then the matrix $(y_{P,Q})$ with $y_{P,Q}=y_{Q^{-1}P}$
commutes with the matrix $M_G(x)=(x_{PQ^{-1}})$.i.e.\[\sum_R x_{PR^{-1}}y_{Q^{-1}R}=\sum_S y_{S^{-1}P}x_{SQ^{-1}}.\]
\end{lemma}

\begin{proof}
Let $P,Q$ be fixed elements. If $x_{PR^{-1}}=x_{SQ^{-1}}$, then $SQ^{-1}=PR^{-1}$ and $S=PR^{-1}Q$. Hence $S^{-1}P=Q^{-1}R$ and $y_{S^ {-1}P}=y_{Q^{-1}R}$. This means that if $x_{PR^{-1}}=x_{SQ^{-1}}$, then $ x_{PR^{-1}}y_{Q^{-1}R}= y_{S^{-1}P}x_{SQ^{-1}}$. So $\sum_R x_{PR^{-1}}y_{Q^{-1}R}=\sum_S y_{S^{-1}P}x_{SQ^{-1}}$.
\end{proof}

It is well known that the regular representation $\mathscr{L}$ of $kG$ is equivalent to $\oplus_{i=1}^s L_i^{e_i}$, where $L_1,\cdots,L_s$ are non-equivalent irreducible representations of $kG$, $L_i^{e_i}=\overbrace{L_i\oplus\cdots\oplus L_i}^{e_i}$ and $\sum_{i=1}^s e_i^2=|G|$.

\begin{lemma}\label{lem:irrFacRep}
Let $\Theta_i(x)=|\sum_{g\in G}L_i(g)x_g|$. Then $\Theta_i(x)=\Phi_i(x)$ after suitable arranging the indexes.
\end{lemma}

\begin{proof}
By Lemma \ref{lem:composable}, $\Theta_i(x)$ is a power factor of $\Theta_G(x)$. By definition of group determinant, we have  $\Theta_G(x)=|\sum_{g\in G}\mathscr{L}(g)x_g|=\prod_{i=1}^s \Theta_i^{e_i}(x)$. By section 3 in \cite{F1} and section 9 in \cite{F2}, there are $s$ irreducible factors in the factorization of $\Theta_G(x)$ and the sum of the square of the multiplicity of  each irreducible factors is $|G|$. If $\Theta_i(x)$ is not irreducible for some $1\leq i \leq s$, the sum of the square of multiplicity of  each irreducible factors must be larger than $|G|$. So $\Theta_i(x)$ is irreducible and $\Theta_i(x)=\Phi_i(x)$ after suitable arranging the indexes.
\end{proof}

\begin{lemma}\label{lem:linearfactors}
Let $\mathcal{P}=\coprod_{i=1}^r \mathcal{P}_r$ be a good partition. If $x_g=x_{g'}$ when $g,g'$ belong to the same partition class $\mathcal{P}_j$, then $\Phi_i(x)$ is a power of a linear factor.
\end{lemma}

\begin{proof}
Let $g_j$ be the representative of $\mathcal{P}_j$.
If $x_g=x_{g'}$ when $g,g'$ belong to the same partition class $\mathcal{P}_j$, then $|\sum_{g\in G}L_i(g)x_g|=|\sum_{j=1}^r L_i(\hat{C_j})x_{g_j}|$, where $\hat{C_j}=\sum_{g\in \mathcal{P}_j} g$. Since $L_i$ is irreducible and $\hat{C_j} \in Z(kG)$, $L_i(\hat{C_j})$ is a multiple of the identity matrix. So $\Phi_i(x)$ is a power of a linear factor.
\end{proof}

\begin{lemma}\label{lem:charCorespondence} If $\Phi_i(x)=|\sum_{g\in G}L_i(g)x_g|$, then the function $\chi_i$ defined by $\Phi_i(x)$ is the character of the irreducible representation $L_i$.
\end{lemma}

\begin{proof}
By the orthogonal relations in section 3 of \cite{F2} and Theorem \ref{thm:classfunction},  the functions $\chi_i s$ defined by irreducible factors $\Phi_is$ are exactly the irreducible characters of $G$ over  $k$. By definition of $\chi_i$, the factor $\Phi_i(x)$ has the following expression: $\Phi_i(x)=x_{g_1}^{f_i}+\sum_{g\not= g_1}\chi_i(g)x_{g_1}^{f_i-1}x_g+*$. Let $g_i$ be the representative of the conjugate class $\mathscr{L}_i$. If $x_{g}=x_{g'}$ when $g,g'$ are conjugate in $G$, by Theorem \ref{thm:classfunction}, $\Phi_i(x)$ can be written into $\Phi_i(x)=x_{g_1}^{f_i}+\sum_{j=2}^s|\mathscr{L}_j|\chi_i(g_j)x_{g_1}^{f_i-1}x_{g_j}+*$. If we take the good partition to be the conjugate classes itself, by Lemma \ref{lem:linearfactors}, $\Phi_i(x)$ is a power of a monic linear factor by identifying $x_g=x_{g'}$ when $g,g'$ belong to the same conjugate class i.e. $\Phi_i(x)=(x_{g_1}+\sum_{j=1}^2\frac{|\mathscr{L}_j|\chi(g_j)}{\chi(1)}x_{g_j})^{f_i}$, where $\chi$ is the character of $L_i$. Since $f_i$ is the dimension of $L_i$, we have $f_i=\chi(1)$. By comparing the coefficient of $x_{g_1}^{f_i-1}x_g$ in two expressions of $\Phi_i(x)$, we have $\chi_i=\chi$.

\end{proof}
\begin{corollary}\label{thm:abbaExp}\cite{F2}
If we identity $x_{g}=x_{g'}$ when $g$ and $g'$ are conjugate in $G$, then $\Phi(x)=(\frac{1}{f}\sum\chi(R)x_R)^f$, where $f$ is the degree of $\Phi(x)$.
\end{corollary}
\begin{proof}
It follows from Lemma \ref{lem:linearfactors} and  Lemma \ref{lem:charCorespondence} directly.
\end{proof}

Define matrix $M_i$ to be $M_G(x)$ by taking $x_g=1$ whenever $g\in \mathcal{P}_i$ and $x_g=0$ otherwise. Then the map $$\mathfrak{r}:k\mathcal{P}\rightarrow M_l(k)$$ defined by $\mathfrak{r}(\hat{C}_i)= M_i$ is a faithful representation of the partition algebra $k\mathcal{P}$. In fact, $\mathfrak{r}$ is the restriction of $\mathscr{L}$ to the subalgebra $k\mathcal{P}$. Let $g_i$ be a representative in $\mathcal{P}_i$.

\begin{proposition}\label{lem:samesym} Let $(z_{PQ^{-1}})=(x_{PQ^{-1}})(y_{PQ^{-1}})$. Let $\mathcal{P}={\mathcal{P}_1\cup\cdots\cup\mathcal{P}_n}$ be a good partition of Cl$(G)$. If the matrices $(x_{PQ^{-1}})$ and $(y_{PQ^{-1}})$ satisfy the condition $x_R=x_T$ and $y_R=y_T$ whenever $R,T$ belong to the same partition class, then $z_{PQ^{-1}}=\sum_R x_{PR^{-1}}y_{RQ^{-1}}=\sum_S y_{PS^{-1}}x_{SQ^{-1}}$ and $(z_{PQ^{-1}})$ has the same symmetry as $(x_{PQ^{-1}})$ i.e. $z_{R}=z_{T}$ whenever $R,T$ belong to the same partition class.
\end{proposition}
\begin{proof} The condition $y_R=y_T$ whenever $R,T$ belongs to the same partition class implies $ y_{AB}=y_{BA}$. So $(y_{PQ^{-1}})=(y_{Q^{-1}P})$. By  Lemma \ref{lem:commut}, $(x_{PQ^{-1}})$ commutes with $(y_{PQ^{-1}})$ i.e. $z_{PQ^{-1}}=\sum_R x_{PR^{-1}}y_{RQ^{-1}}=\sum_S y_{PS^{-1}}x_{SQ^{-1}}$.

  If the matrices $(x_{PQ^{-1}})$ and $(y_{PQ^{-1}})$ satisfy the condition $x_R=x_T$ and $y_R=y_T$ whenever $R,T$ belong to the same partition class, then $(x_{PQ^{-1}})=\sum_{l=1}^n x_{g_l}M_l$ and $(y_{PQ^{-1}})=\sum_{l=1}^n y_{g_l}M_l$. So
\begin{align*}
(z_{PQ^{-1}})&=(x_{PQ^{-1}})(y_{PQ^{-1}})\\
&=\sum_i\sum_j x_{g_i}y_{g_j}M_iM_j\\
&=\sum_i\sum_j x_{g_i}y_{g_j}\sum_{l=1}^n a_{lij} M_l\\
&=\sum_{l=1}^n (\sum_i\sum_j x_{g_i}y_{g_j}a_{lij}) M_l.
\end{align*}
This implies that $(z_{PQ^{-1}})$ has the same symmetry as $(x_{PQ^{-1}})$ i.e. $z_{R}=z_{T}$ whenever $R,T$ belong to the same partition class.

\end{proof}

\begin{definition}
 Define function $\chi_i^{\mathcal{P}}$ of $\Phi_i(x)$ relative to the good partition $\mathcal{P}$ as following:
\[\chi_i^{\mathcal{P}}(g)=\sum_{y\in \mathcal{P}_i} \chi_i(y),\]if $g$ belongs to $\mathcal{P}_i$.
We will call the function $\chi_i^{\mathcal{P}}$ the $\mathcal{P}$-character of $G$ relative to the partition $\mathcal{P}$.
\end{definition}

By Theorem \ref{thm:abbaExp}, if  we assume that  $x_R=x_S$ whenever $R,S$ belong to the same partition class, then $$\Phi_i(x)=(\frac{1}{f_i}\sum_{j=1}^n\chi_i^\mathcal{P}(g_{j})x_{g_{j}})^{f_i},\mbox{ and } \eta=\frac{1}{f_i}\sum_{j=1}^n\chi_i^\mathcal{P}(g_{j})x_{g_{j}},$$ where $\Phi(x)$ is an irreducible factor of $\Theta_G(x)$ with degree $f$, $\mathcal{P}={\mathcal{P}_1\cup\cdots\cup\mathcal{P}_n}$ is a good partition of Cl$(G)$ and $g_j$ is a representative of $\mathcal{P}_j$.

\begin{theorem}\label{thm:defeqn}
Let $\mathcal{P}={\mathcal{P}_1\cup\cdots\cup\mathcal{P}_n}$ be a good partition of Cl$(G)$ and let $g_i$ be a representative  of  $\mathcal{P}_i$. Then $${\chi_i^\mathcal{P}(g_t)\chi_i^\mathcal{P}(g_j)}= f_i\sum_{l=1}^n a_{ltj}\chi_i^{\mathcal{P}}(g_l) (1\leq t,j,l\leq n)$$ for a character $\chi_i^\mathcal{P}$ with respect to the partition $\mathcal{P}$, where $a_{ltj}$ is defined in Lemma \ref{basicEqs}.
\end{theorem}

\begin{proof}
Let $\Theta_G(x)=|(x_{PQ^{-1}})|=\prod_{i=1}^{s'}\Phi_i^{e_i}(x)$ be the factorization into irreducible polynomials over $k$. Then each $\Phi_i(x)$ define a character $\chi_i$. By Theorem \ref{thm:classfunction}, $\chi_i$ is class function of $G$. So $s'$ is less than the number $s$ of conjugate classes of $G$.  If we set $x_{BA}=x_{AB}$ in $\Phi_i(x)$, by Theorem \ref{thm:abbaExp}, $\Phi_i(x)=(\frac{1}{f_i}\sum_{j=1}^s|\mathscr{C}_j|\chi_i(d_j)x_{d_j})^{f_i}$, where $d_j$ is the representative of the conjugate class $\mathscr{C}_j$. Let $\eta_i=\frac{1}{f_i}\sum_{j=1}^s|\mathscr{C}_j|\chi_i(d_j)x_{d_j}$. Since $\chi_1,\cdots,\chi_{s'}$ are orthogonal to each other by (2) of Section 3 in \cite{F1} , $\eta_1,\cdots,\eta_{s'}$ are coprime to each other. If we set $x_{BA}=x_{AB}$ in $\Theta_G(x)$, then $\Theta_G(x)=\eta_1^{f_1}\cdots \eta_{s'}^{f_{s'}}$. By Lemma \ref{lem:detofxy=yx}, we also have  $\Theta_G(x)=\prod_{j=1}^s (\eta_j')^{m_j}$, where $\eta_j'=\sum_{i=1}^s\mathfrak{l}_j(\hat{\mathscr{C}}_i)x_{d_i}$.  Since $\mathfrak{l}_1,\cdots,\mathfrak{l}_s$ are all inequivalent representations of $Z(kG)$, $\eta_1',\cdots,\eta_s'$ are coprime to each other. So, by suitable arrangement, we have $s=s'$, $\eta_i=\eta_i'$ and $m_i=f_i$ for $1\leq i \leq s$. By comparing coefficients of $\eta_i$ and $\eta_i'$, we have $\frac{|\mathscr{C}_j|\chi_i(d_j)}{f_i}=\mathfrak{l}_i(\hat{\mathscr{C}}_j)$. Since $k\mathcal{P}$ is a subalgebra of $Z(kG)$, the restriction of $\mathfrak{l}_i$ to $k\mathcal{P}$ is also an irreducible representation of $k\mathcal{P}$. So we have
\begin{align*}
\frac{\chi_i^\mathcal{P}(g_t)}{f_i}\frac{\chi_i^\mathcal{P}(g_j)}{f_i}&=\mathfrak{l}_i(\hat{C}_t)\mathfrak{l}_i(\hat{C}_j)
=\mathfrak{l}_i(\hat{C}_t\hat{C}_j)\\
&=\mathfrak{l}_i(\sum_{l=1}^n a_{ltj}\hat{C}_l)\\
&=\sum_{l=1}^n a_{ltj}\mathfrak{l}_i(\hat{C}_l)\\
&=\sum_{l=1}^n a_{ltj}\frac{\chi_i^\mathcal{P}(g_l)}{f_i}.
\end{align*}
And it follows that ${\chi_i^\mathcal{P}(g_t)\chi_i^\mathcal{P}(g_j)}= f_i\sum_{l=1}^n a_{ltj}\chi_i^{\mathcal{P}}(g_l)$.
\end{proof}

\begin{definition}\label{def:char}
The definition equations of characters relative to a good partition $\mathcal{P}={\mathcal{P}_1\cup\cdots\cup\mathcal{P}_n}$
are  the following series of equations:
\[ x_ix_j=\frac{1}{\lambda_\mathcal{P}}x_1\sum_l a_{lij}x_l, 1\leq i,j,l\leq n,\] where  $\lambda_\mathcal{P}$ is the identity constant of $k\mathcal{P}$. The affine variety defined by these equations are called $\mathcal{P}$-character variety with respect to the good partition $\mathcal{P}$. We denote it be $V_\mathcal{P}(G)$.
\end{definition}

The definition equations in \ref{def:char} unify various definition equations of characters. Next we will show some definition equations of characters in case that the corresponding good partition $\mathcal{P}$ satisfes additional conditions. These definition equations can be derived from  definition equations in \ref{def:char}.


 Let $\mathcal{P}={\mathcal{P}_1\cup\cdots\cup\mathcal{P}_n}$ be a good partition with $g_i$ as a representative of $\mathcal{P}_i$ and $g_1=E$.
 We will keep these notations in the following remarks. If $\mathcal{P}_1=\{E\}$, then $\lambda_{\mathcal{P}}=1$.

\begin{remark}
If each partition class $\mathcal{P}_j$ consists of conjugate classes with the same length, then $\chi^\mathcal{P}(g_j)= h_j\sum_{t=1}^l\chi(g_{j_t})$, where $\mathcal{P}_j=\{\mathscr{C}_{j_1},\cdots,\mathscr{C}_{j_l}\}$ and $h_j$ is the common length of conjugate classes contained in $\mathcal{P}_j$ and $\lambda_\mathcal{P}=1$.  Then the definition equations in \ref{def:char} can be written in the following form: $h_ix_ih_jx_j=x_1\sum_l a_{lij}h_lx_l, 1\leq i,j,l\leq r$. In particular, we take $\mathcal{P}$ to be the trivial partition of Cl($G$) i.e. each partition class is exactly to be a conjugate class of $G$. Then we get the definition equations in \cite{F2} $$h_ix_ih_jx_j=x_1\sum_l h_{l'ij}x_l, 1\leq i,j,l\leq r,$$
where $h_{l'ij}=\sharp\{(R,S,T)\in \mathscr{C}_i\times \mathscr{C}_j\times \mathscr{C}_{l'}\mid RST=E\}=a_{lij}h_l$.
\end{remark}

\begin{remark}
Let $\mathcal{P}$ be the partition consistting of $k$-classes. Then each $\mathcal{P}$ is a $k$-class and the conjugate classes in a $\mathcal{P}_i$ have the same length. In this case, $\lambda_\mathcal{P}=1$. Then the definition equations in \ref{def:char} can be written into $$h_ix_ih_jx_j=x_1\sum_l h_{l'ij}x_l, 1\leq i,j,l\leq r \leqno(13),$$ where $h_i$ is the length of some conjugate class in $\mathcal{P}_i$ and $h_{l'ij}=\sharp\{(R,S,T)\in \mathcal{P}_i\times \mathcal{P}_j\times \mathcal{P}_{l'}\mid RST=E\}=a_{lij}h_l$.

\end{remark}

\begin{remark}
Let $\mathcal{P}$ be partition in \ref{exa:goodpart2}. Then $\lambda_\mathcal{P}=\frac{1}{|N|}$. In order to see the characters define by the partition $\mathcal{P}$ as character of $G$, we set $\chi^P(g_i)=q_i\chi(g_i)$, where $q_i$ is the number of elements contained in $\mathcal{P}_i$. Then the equations in \ref{def:char} can be write into the following form:
\[q_ix_iq_jx_j=|N|x_1\sum_l q_{l'ij}x_l, 1\leq i,j,l\leq r,\] where $q_{l'ij}=\sharp\{(R,S,T)\in \mathcal{P}_i\times \mathcal{P}_j\times \mathcal{P}_{l'}\mid RST=E\}=a_{lij}q_l$. In fact, they are definition equations of characters of $G/N$.

\end{remark}


\section{Realization of $\mathcal{P}$-characters}\label{sec:realization}

In this section, we will determine the character variety by the representation variety of the correponding algebra of the good partition. Then orthogonal relations of these characters are derived from the representation matrix and its invertible matrix. Based on the definition equations and the orthogonal relations, we will give the formula for $m_i$ of the characters and prove that $m_i$ is the multiplicity of a linear factor in $\Theta_G(x)$ in case that $x_A=x_B$ whenever $A$ and $B$ belong to the same partition class of $\mathcal{P}$.

Let $\mathcal{P}={\mathcal{P}_1\cup\cdots\cup\mathcal{P}_n}$ be a good partition of Cl($G$) with $g_i$ as a representative of $\mathcal{P}_i$ and $g_1=E$. As in section \ref{sec:Frobenius theory}, the partition class sums  $\hat{C}_1,\cdots,\hat{C}_n$ consist of a basis of the partition algebra  $k\mathcal{P}$. Let $\mathcal{L}:k\mathcal{P}\rightarrow M_n(k)(\hat{C}_j\mapsto \mathcal{A}_j)$ be regular representation of $k\mathcal{P}$.
By Theorem \ref{thm:semisimPart}, $k\mathcal{P}$ is semisimple. Since $k\mathcal{P}$ is commutative and semisimple,  there exists an invertible matrix $\mathcal{M}$ such that
\[\mathcal{M}\mathcal{A}_i\mathcal{M}^{-1}=\begin{pmatrix}
\lambda_{i1}& 0& 0&\cdots &0\\
0& \lambda_{i2}& 0&\cdots &0\\
\cdot& \cdot& \cdot&\cdots &\cdot\\
0& 0& 0&\cdots &\lambda_{in}\\
\end{pmatrix}\] for $1\leq i \leq n$.
Then \[\mathfrak{r}_j:k\mathcal{P}\longrightarrow k(\hat{C}_i\mapsto \lambda_{ij}),1\leq j\leq n\] is an irreducible representation of $k\mathcal{P}$. Let $\gamma_t=(\lambda_{1t},\cdots,\lambda_{nt})$, where $1\leq t\leq n$.
Then $\gamma_t$ is a solution of the following series of equations:
\begin{equation*}x_jx_l=\sum_{i=1}^n a_{ijl} x_i,1\leq i,j,l\leq n.\leqno(1)\label{equ:defirrep}\end{equation*}
Let $\gamma_t(x)=\sum_{i=1}^n \lambda_{it}x_i$ for $1\leq t \leq n$.
By Theorem \ref{Property:A}, $\gamma_1,\cdots,\gamma_n$ are the exact $n$ solutions of above series of equations.
Let $\chi_t^\mathcal{P}$ be a character with respect to $\mathcal{P}$. By Theorem \ref{thm:defeqn}, $(\chi_t^\mathcal{P}(g_1),\cdots,\chi_t^\mathcal{P}(g_n))$ is a solution of the equations
\[ x_ix_j=f_t\sum_l a_{lij}x_l, 1\leq i,j,l\leq n£¬\leqno(2)\label{equ:Relizationchar}\]where $f_t$ is the degree of the corresponding irreducible factor of $\Theta_G(x)$.
 So $$(\frac{\chi_t^\mathcal{P}(g_1)}{f_t},\cdots,\frac{\chi_t^\mathcal{P}(g_n)}{f_t})$$ is  a solution of equations (1).   By Corollary \ref{cor:rvariety}, there exists  $\gamma_t$ such that $\gamma_t=(\frac{\chi_t^\mathcal{P}(g_1)}{f_t},\cdots,\frac{\chi_t^\mathcal{P}(g_n)}{f_t})$ . This implies that $\chi_t^\mathcal{P}(g_j)=f_t\lambda_{jt}$ for $1\leq j\leq n$.

\begin{theorem} The character variety $V_\mathcal{P}(G)$ consists of all multiples of elements in  $V_k(k\mathcal{P})$ i.e.
$V_\mathcal{P}(G)=\{\alpha\gamma|\alpha\in k, \gamma\in V_k(k\mathcal{P}) \}$. In particular, $V_\mathcal{P}(G)$  consists of $n$ irreducible components and the dimension of $V_\mathcal{P}(G)$ is $1$.

\end{theorem}
\begin{proof}
 Let $\chi_t=(f_t\lambda_{1t},\cdots,f_t\lambda_{nt})$ for $1\leq t\leq n$. Then
 the statement $\chi_{t}$ is a solution of the equations  in Definition \ref{def:char}  follows from above analysis and Theorem \ref{thm:defeqn}.
          Let $\beta=(\beta_1,\cdots,\beta_n)$ be the any nonzero solution of the following equations
 \[ x_ix_j=\frac{1}{\lambda_{\mathcal{P}}}x_1\sum_l a_{lij}x_l, 1\leq i,j,l\leq n.\leqno(3)\label{equ:Relizationchar}\]  Then there some $\beta_i\not= 0$ and  $0\not=\beta_i\beta_i=\frac{1}{\lambda_{\mathcal{P}}}\beta_1\sum_l a_{lij}\beta_l$. This implies that $\beta_1\not= 0$.
 So $\frac{1}{\beta_1}\beta$ is a solution of (1). By Corollary \ref{cor:rvariety}, $\beta$ is a multiple of $\chi_t$. Thus
  any solution of  the equations  in Definition \ref{def:char} is of the form $\alpha g_t$ for some $\alpha\in k$ and $1\leq t\leq n$.
  So $V_\mathcal{P}(G)=\{\alpha\gamma|\alpha\in k, \gamma\in V_k(k\mathcal{P}) \}$. Since $V_\mathcal{P}(G)$ is the union of $n$ $1$-dimensional vector spaces, the other statements follow.
\end{proof}

 By equation (\ref{eleDeterm}) in section \ref{sec:Frobenius theory} and Lemma \ref{basicEqs}, the following  equation of determinant holds true
\begin{equation*}\label{equ:chi1}|\sum_{t}\ell_{ij't}x_t-\ell_{ij'}x|_{1\leq i,j\leq n}=\prod_{t=1}^n\ell_t(\gamma_t(x)-x).\end{equation*}And so
 \begin{equation*}\label{equ:basicdeterminant}|\sum_{t}\frac{\ell_{ij't}}{\ell_i}x_t-\frac{\ell_{ij'}}{\ell_i}x|_{1\leq i,j\leq n}=\prod_{t=1}^n(\gamma_t(x)-x).\end{equation*}

Assume that $$\Theta_G(x)=|M_G(x)|=\prod_i\Phi_i^{c_i}(x)\leqno(Dequ)$$ is a factorization of irreducible polynomials over complex field $\mathbb{C}$ and the degree of $\Phi_i(x)$ is $f_i$.


\begin{theorem}\label{thm:satequ}
For $1\leq t\leq n$, the following hold:
\begin{itemize}
\item[(i)] $\chi^\mathcal{P}_t(g_1)=\lambda_\mathcal{P}f_t$;
\item[(ii)] $\lambda_{ij}$ is an algebraic integer for $1\leq i,j\leq s$;
\end{itemize}
\end{theorem}
\begin{proof}
 By taking $x_1=1$ and $x_2=\cdots=x_n=0$ in both sides of equation (3) in section \ref{sec:Frobenius theory}, we have $(\lambda_{\mathcal{P}}-x)^n=\prod_{t=1}^n(\frac{\chi^\mathcal{P}_t(g_1)}{f_t}-x)$. So  $\chi^\mathcal{P}_t(g_1)=\lambda_\mathcal{P}f_t$ and follows from (i) directly.

By taking $x_t=1$ and $x_i=0$ for $i\not=t$ on both sides of equation (3), it follows
$|(a_{i'j't})-x I|= \prod_{\alpha=1}^n(\lambda_{t\alpha}-x)$. This implies $\lambda_{t\alpha}$ is a root of monic polynomial with rational integers as coefficients since $a_{i'j't}$ is rational integer and (ii) holds.

\end{proof}

Let  $\mathcal{P}=\coprod_{i=1}^n\mathcal{P}_i$ be a good partition of  $\mbox{Cl}(G)$, where 
$\mathcal{P}_j=\{\mathscr{C}_{j_1},\cdots,\mathscr{C}_{j_l}\}$ for $1\leq j\leq n$. Then the partition algebra $k\mathcal{P}$ is a subalgebra of the center $Z(kG)$ of the group algebra $kG$. Let $\chi_1,\cdots,\chi_s$ be all irreducible characters of $G$ over $k$. Then each irreducible character $\chi_i$ defines an primitive idempotent $\mathfrak{e}_i=\frac{\chi_i(1)}{|G|}\sum_{x\in G} \chi_i(x^{-1})x(1\leq i \leq s)$ in $Z(kG)$ and $1=\mathfrak{e}_1+\cdots +\mathfrak{e}_s$. Let $E_j(1\leq j\leq n)$ be all primitive idempotents of $k\mathcal{P}$. Since $k\mathcal{P}$ and $Z(kG)$ have the same identity, then $E_j= \mathfrak{e}_{j_1}+\cdots +\mathfrak{e}_{j_{\alpha_j}}$ and $\alpha_1+\cdots +\alpha_n=s$. According to the coefficient of $g$ in $E_j$, we define a function $\theta_j:G\rightarrow k$ such that $\theta_j(g)=\sum_{i=1}^{\alpha_j} \chi_{j_i}(1)\chi_{j_i}(g^{-1})$. Then $\theta_1,\cdots,\theta_n$ are super characters determined by the good partition $\mathcal{P}$( see \cite{DI}). Since $E_1,\cdots,E_n$ and $\hat{C_1},\cdots,\hat{C_n}$ are two basis of $k\mathcal{P}$, $\theta_j$ is constant on a partition class $\mathcal{P}_i$, i.e. $\theta_j(g)=\theta_j(g')$ if $g,g'$ belongs to the same partition class. We have the following transition matrix

\[(E_1,\cdots,E_n)=(\hat{C}_1,\cdots,\hat{C_n})A,\] where \[A=\frac{1}{|G|}\begin{pmatrix}
\theta_1(g_1)& \theta_2(g_1)& \theta_3(g_1)&\cdots &\theta_n(g_1)\\
\theta_1(g_2)& \theta_2(g_2)& \theta_3(g_2)&\cdots &\theta_n(g_2)\\
\cdot& \cdot& \cdot&\cdots &\cdot\\
\theta_1(g_n)& \theta_2(g_n)& \theta_3(g_2)&\cdots &\theta_n(g_n)\\
\end{pmatrix},\] where $g_i$ is a representative of $\mathcal{P}_i$.
By the orthogonal relations of irreducible characters, the  matrix $A^{-1}$ is

\[
\begin{pmatrix}
\frac{\ell_1\theta_1(g_1^{-1})}{\theta_1(1)}& \frac{\ell_2\theta_1(g_2^{-1})}{\theta_1(1)}& \frac{\ell_3\theta_1(g_3^{-1})}{\theta_1(1)}&\cdots &\frac{\ell_n\theta_1(g_n^{-1})}{\theta_1(1)}\\
\frac{\ell_1\theta_2(g_1^{-1})}{\theta_2(1)}& \frac{\ell_2\theta_2(g_2^{-1})}{\theta_2(1)}& \frac{\ell_3\theta_2(g_3^{-1})}{\theta_2(1)}&\cdots &\frac{\ell_n\theta_2(g_n^{-1})}{\theta_2(1)}\\
\cdot& \cdot& \cdot&\cdots &\cdot\\
\frac{\ell_1\theta_n(g_1^{-1})}{\theta_n(1)}& \frac{\ell_2\theta_n(g_2^{-1})}{\theta_n(1)}& \frac{\ell_3\theta_n(g_3^{-1})}{\theta_n(1)}&\cdots &\frac{\ell_n\theta_n(g_n^{-1})}{\theta_n(1)}\\
\end{pmatrix},\] where $\ell_i$ is the length of $\mathcal{P}_i$.

The generic polynomial $g_{k\mathcal{P}}(y_1,\cdots,y_n)$ of $k\mathcal{P}$ with respect to the basis $E_1,\cdots,E_n$ is $y_1y_2\cdots y_n$. Applying Theorem \ref{thm:bccpoly}, replacing $y_i$ by $\frac{\ell_1\theta_i(g_1^{-1})}{\theta_i(1)}x_1+\frac{\ell_2\theta_i(g_2^{-1})}{\theta_i(1)}x_2+ \frac{\ell_3\theta_i(g_3^{-1})}{\theta_i(1)}x_3+\cdots +\frac{\ell_n\theta_i(g_n^{-1})}{\theta_i(1)}x_n$ we get the generic polynomial  of $k\mathcal{P}$ with respect to the basis $\hat{C}_1,\cdots,\hat{C}_n$ i.e. $$g_{k\mathcal{P}}(x_1,\cdots,x_n)=\prod_{i=1}^n g_i(x_1,\cdots,x_n),$$ where $g_i(x_1,\cdots,x_n)=\frac{\ell_1\theta_i(g_1^{-1})}{\theta_i(1)}x_1+\frac{\ell_2\theta_i(g_2^{-1})}{\theta_i(1)}x_2+ \frac{\ell_3\theta_i(g_3^{-1})}{\theta_i(1)}x_3+\cdots +\frac{\ell_n\theta_i(g_n^{-1})}{\theta_i(1)}x_n$. Since $k\mathcal{P}$ is a commutative algebra, each $g_i(x_1,\cdots,x_n)$ define the irreducible representation $\mathfrak{r}_i$ of $k\mathcal{P}$ (arranging the indexes if necessary) such that $\mathfrak{r}_i(\hat{C}_j)=\frac{\ell_j\theta_i(g_j^{-1})}{\theta_i(1)}$ for $1\leq i \leq n$. This implies that each row of the matrix $A^{-1}$ is a solution of equations (1). For a good partition $\mathcal{P}$ of Cl($G$), we have the following relations between $\mathcal{P}$-characters and super characters determined by $\mathcal{P}$.

\begin{theorem}\label{thm:diffpchaAscha}
 Let $\Phi(x)$ be an irreducible factor of $\Theta_G(x)$ and let $\chi^\mathcal{P}$ be the $\mathcal{P}$-character defined by $\Phi(x)$. Then there is a super character $\theta$ determined by $\mathcal{P}$ such that $\chi^\mathcal{P}(g)=\frac{\chi^{\mathcal{P}}(1)\ell_g\theta(g^{-1})}{\theta(1)}$,where $\ell_g$ is the length of the partition class containing $g$.
\end{theorem}
\begin{proof}
Let $g_i$ be the representative of the partition class $\mathcal{P}_i$ for $1\leq i \leq n$. By section 9 in \cite{F2}, $\chi^{\mathcal(P)}(1)$ is the degree of $\Phi(x)$. Thus we have  shown that $(\frac{\chi^{\mathcal{P}}(g_1)}{\chi^{\mathcal{P}}(1)},\cdots,\frac{\chi^{\mathcal{P}}(g_n)}{\chi^{\mathcal{P}}(1)})$ is a solution of the equations (1) at the beginning of this section. Since $|A^{-1}|\not=0$, $n$ rows of $A^{-1}$ are exactly the solutions of equations (1). So
$(\frac{\chi^{\mathcal{P}}(g_1)}{\chi^{\mathcal{P}}(1)},\cdots,\frac{\chi^{\mathcal{P}}(g_n)}{\chi^{\mathcal{P}}(1)})$ is some row of $A^{-1}$.

\end{proof}

\begin{remark}
This Theorem shows that a $\mathcal{P}$-character is not a super character, even not a multiple of some super character.
\end{remark}

For each irreducible character $\chi_i$ of $G$ over $\mathbb{C}$, by Lemma \ref{lem:charCorespondence}, there is a unique irreducible factor of $\Phi_i$ of $\Theta_G(x)$ such that the coefficient of $x_{g_1}^{\chi_i(1)-1}x_g$ in $\Phi_i$ is $\chi_i(g)$ in the case that $g_1\not=g$ and the coefficient of $x_{g_1}^{\chi_i(1)}$ is $1$. The multiplicity of $\Phi_i$ in $\Theta_G(x)$ is $\chi_i(1)$ and $\Theta_G(x)=\prod_{i=1}^s\Phi_i^{\chi_i(1)}$(see\cite{F2}). According to the decomposition $E_j= e_{j_1}+\cdots +e_{j_{\alpha_j}}$, we define $\Theta_{E_j}(x)=\prod_{t=1}^{\alpha_j}\Phi_{j_t}^{\chi_{j_t}(1)}$. Then the degree of $\Theta_{E_j}(x)$ is $d_j=\sum_{t=1}^{\alpha_j}\chi_{j_t}^2(1)$ and $\Theta_G(x)=\prod_{j=1}^n\Theta_{E_j}(x)$.  The coefficient of $x_{g_1}^{d_j-1}x_g$ in $\Theta_{E_j}(x)$ is $\theta_j(g^{-1})$.

\begin{theorem}
 If $x_g=x_{g'}$ when $g$ and $g'$ belong to the same partition class, then $\Theta_{E_j}(x)=(\sum_{i=1}^n\frac{\ell_i\theta_j(g_i^{-1})}{\theta_j(1)}x_{g_i})^{d_j}$ for $1\leq j\leq n$, where $g_i$ is the representative of the partition class $\mathcal{P}_i$.
\end{theorem}
\begin{proof}
Let $\mathscr{L}$ be the regular representation  of $kG$. Then $\mathscr{L}$ is equivalent to $\oplus_{i=1}^s L_i^{e_i}$, where $L_1,\cdots,L_s$ are non-equivalent irreducible representations of $kG$, $L_i^{e_i}=\overbrace{L_i\oplus\cdots\oplus L_i}^{e_i}$. Assume that $E_j= \mathfrak{e}_{j_1}+\cdots +\mathfrak{e}_{j_{\alpha_j}}$ be the decomposition of primitive idempotents in $Z(kG)$ and $L_{j_t}(\mathfrak{e}_i)=\delta_{j_t,i}I$. Set $\mathscr{L}_j=\oplus_{t=1}^{\alpha_j}L_{j_t}^{e_{j_t}}$. By Lemma \ref{lem:irrFacRep}, we have $|\sum_{g\in G} \mathscr{L}_j(g)x_g|=\Theta_{E_j}(x)$.  If $x_g=x_{g'}$ when $g$ and $g'$ belong to the same partition class, then $|\Theta_{E_j}(x)=\sum_{i=1}^n \mathscr{L}_j(\hat{C}_i)x_{g_i}|$. On the other hand, we have
\[(\hat{C}_1,\cdots,\hat{C_n})=(E_1,\cdots,E_n)A^{-1}.\]  Since $\mathscr{L}_j(E_i)=\delta_{ij}I$, we have \[|\Theta_{E_j}(x)=\sum_{i=1}^n \mathscr{L}_j(\hat{C}_i)x_{g_i}|=|\sum_{i=1}^n\mathscr{L}_j(\frac{\ell_i\theta_j(g_i^{-1})}{\theta_j(1)}E_j)x_{g_i}|
=(\sum_{i=1}^n\frac{\ell_i\theta_j(g_i^{-1})}{\theta_j(1)}x_{g_i})^{d_j}.\]
\end{proof}



\section{The power factors of Group Determinant}\label{sec:powerfactors}
In this section, we assume that $k$ is a subfield of the complex field $\mathbb{C}$.

Let $\Phi(x)$ be a power factors of the group determinant which is a product of powers of irreducible factors of $\Theta_G(x)$. Define a function $\chi: G\rightarrow k$ such that $\chi(g)$ is the coefficient of the monomial $x_{g_1}^{f-1}$ in $\frac{\partial(\Phi(x))}{\partial x_g}$, where $f$ is the degree of $\Phi(x)$. The function $\chi$ is called power function of $\Phi(x)$. We will show in this section that $\Phi(x)$ can be constructed from its power function $\chi$. The method is due to Frobenius.

Let $u$ be an independent variable. We denote by $\Phi(x+u\varepsilon)$ the polynomial $\Phi(x_E+u,x_{g_2},x_{g_3},\cdots,x_{g_\ell})$. Then
$\Phi(x+u\varepsilon)$ can be written into the form:
$$\Phi(x+u\varepsilon)=u^f+\Phi_1u^{f-1}+\Phi_2u^{f-2}+\cdots+\Phi_f \leqno(1)$$
and $\Phi_r$ is a homogenous polynomial in $n$ variables $x_E,x_{g_2},x_{g_3},\cdots,x_{g_\ell}$. If $u=0$, then $\Phi_f=\Phi$.

\begin{proposition} For $1\leq r \leq f$, we have
  $$(i)\ \Phi_r=\frac{1}{(f-r)!}\frac{\partial^{f-r}\Phi}{\partial x_E^{f-r}},\quad (ii)\ \frac{\partial\Phi_{r+1}}{\partial x_E}=(f-r)\Phi_r. \leqno(2)$$
\end{proposition}
\begin{proof}
If we replace $x_{g_1}$ by $x_{g_1}+u$ in $\Phi(x)$, then only the monomial of $\Phi(x)$ with the form $x_E^a x_A^{l_A} x_B^{l_B}\cdots,\ a\geq f-r$ can affords $u^{f-r}$. In $(x_E+u)^a x_A^{l_A} x_B^{l_B}\cdots$, the monomial containing $u^{f-r}$ is $C_a^{f-r}u^{f-r}x_E^{a-f+r} x_A^{l_A} x_B^{l_B}\cdots$. On the other hand, we have $$C_a^{f-r}x_E^{a-f+r} x_A^{l_A} x_B^{l_B}\cdots=\frac{1}{(f-r)!}\frac{\partial^{f-r}(x_E^a x_A^{l_A} x_B^{l_B}\cdots)}{\partial x_E^{f-r}}.$$ Thus $\frac{1}{(f-r)!}\frac{\partial^{f-r}\Phi}{\partial x_E^{f-r}}$ is the coefficient of $u^{f-r}$ in $\Phi(x+u\varepsilon)$. And (i) follows by comparing with equation (1).
By (i), we have $\Phi_{r+1}=\frac{1}{(f-r-1)!}\frac{\partial^{f-r-1}\Phi}{\partial x_E^{f-r-1}}$. So $\frac{\partial\Phi_{r+1}}{\partial x_E}=\frac{1}{(f-r-1)!}\frac{\partial^{f-r}\Phi}{\partial x_E^{f-r}}=\frac{f-r}{(f-r)!}\frac{\partial^{f-r}\Phi}{\partial x_E^{f-r}}=(f-r)\Phi_r$.
\end{proof}

\begin{proposition}
  $$\Phi_1=\sum_{R\in G}\chi(R)x_R. \leqno(3)$$
\end{proposition}
\begin{proof}
  In $\Phi(x+u\varepsilon)$ only $(x_E+u)^f$ and $\chi(R)(x_E+u)^{f-1}x_R,\ R\neq E$ can attribute to monomials containing $u^{f-1}$.
The coefficient of $x_Eu^{f-1}$ in $(x_E+u)^f$ is $f$ and the coefficient of $u^{f-1}x_R$ in $\chi(R)(x_E+u)^{f-1}x_R$ is $\chi(R)$. Thus $\Phi_1=\sum_{R}\chi(R)x_R$.
\end{proof}

Assume that $\Phi(x+u\varepsilon)=u^f+\Phi_1(x)u^{f-1}+\cdots+\Phi_f(x)$ has the following factorization
$$\Phi(x+u\varepsilon)=(u+u_1)(u+u_2)\cdots(u+u_f). \leqno(4)$$

Let $g(u)$ be a polynomial with variable $u$ with $g(u)=a(u+v_1)(u+v_2)\cdots(u+v_r)$, where $a,v_1,v_2,\ldots,v_r$ are constants. Set \[(y_{PQ^{-1}})=g((x_{PQ^{-1}}))=a((x_{PQ^{-1}})+v_1I)((x_{PQ^{-1}})+v_2I)\cdots((x_{PQ^{-1}})+v_rI),\]where $I$ is the identity matrix.
 Then $g((x_{PQ^{-1}}))$ has the same symmetry as $M_G(x)$.   By Lemma \ref{lem:composable}, we have
$$\Phi(y)=a^f\Phi(x+v_1\varepsilon)\Phi(x+v_2\varepsilon)\cdots\Phi(x+v_r\varepsilon). \leqno(4)^\prime$$
By equation (4),  $\Phi(x+v_i\varepsilon)=(v_i+u_1)(v_i+u_2)\cdots(v_i+u_f)$. Since $g(u_j)=a(v_1+u_j)(v_2+u_j)\cdots(v_n+u_j)$,
\begin{eqnarray*}
  \Phi(y)&=& a^f\prod_{i=1}^r\Phi(x+v_2\varepsilon)\\
  &=&a^f\prod_{i=1}^r (v_i+u_1)(v_i+u_2)\cdots(v_i+u_f)\\
  &=&\prod_{j=1}^f\prod_{i=1}^ra(v_i+u_j)\\
  &=&g(u_1)g(u_2)\cdots g(u_f).
\end{eqnarray*}

Thus we prove the following theorem.

\begin{theorem}\label{thm:phiy}
   With notations above, we have
  $\Phi(y)=g(u_1)g(u_2)\cdots g(u_f)$.
\end{theorem}

 Let $g(u)=u^r+v$. Then $g((x_{PQ^{-1}}))=(x_{PQ^{-1}})^r+v I$. Applying Theorem \ref{thm:phiy} to polynomial $g(u)=u^r+v$ we have
$\Phi(y)=(v+u_1^n)(v+u_2^n)\cdots(v+u_f^n) $. On the other hand, the $(P,Q)$-entry of $(x_{PQ^{-1}})^r$ is
$$x_{PQ^{-1}}^{(r)}=\sum_{S_1\cdots S_{r-1}\in\mathfrak{H}}x_{PS_1}x_{S_1^{-1}S_2}\cdots x_{S_{r-1}^{-1}Q^{-1}}.$$ And $g((x_{PQ^{-1}}))=(x_{PQ^{-1}}^{(r)})+v I$. Since $g((x_{PQ^{-1}}))$ has the same symmetry as $M_G(x)$, it follows $$\Phi(x^{(r)}+v\varepsilon)\triangleq\Phi(y+v\varepsilon)=(v+u_1^r)(v+u_2^r)\cdots(v+u_f^r). \leqno(\ast)$$
We can also write $\Phi(x^{(r)}+v\varepsilon)$ into the following form
\[\Phi(x^{(r)}+v\varepsilon)=v^f+\Phi_1(x^{(r)})v^{f-1}+\Phi_2(x^{(r)})v^{f-2}+\cdots+\Phi_f(x^{(r)}).\leqno(\ast\ast)\]
By equation (3), we have $\Phi_1(x^{(r)})=\sum_{R\in G} \chi(R)x_R^{(r)}$. Let $S_r=\Phi_1(x^{(r)})$. Then $$S_r=\sum_{R_1,\ldots,R_r}\chi(R_1\cdots R_r)x_{R_1}\cdots x_{R_r},\leqno(5)$$ where $R_1,\cdots,R_r$ runs over $G$ independently. By comparing the coefficient of $v^{f-1}$ in the equations ($\ast$) and ($\ast\ast$), we also have $$S_n=\sum_R \chi(R)x_R^{(n)}=u_1^n+u_2^n+\cdots+u_f^n. \leqno(6)$$

By comparing coefficient of $u^i$ of (4) and (6), we have

\begin{equation}\nonumber
  \begin{array}{ll}
  \Phi_1=u_1+\cdots+u_f & S_1=u_1+\cdots+u_f\\
  \Phi_2=\sum_{1\leq i\neq j\leq f}u_iu_j & S_2=u_1^2+\cdots+u_f^2\\
 \vdots & \vdots\\
  \Phi_n=\sum_{i_1\neq i_2\neq\cdots\neq i_n}u_{i_1}u_{i_2}\cdots u_{i_n} & S_n=u_1^n+\cdots+u_f^n\\
  \vdots & \vdots\\
  \Phi_f=u_1\cdots u_f & S_f=u_1^f+\cdots+u_f^f
  \end{array}
\end{equation}
By the formula on power sum and product sum in \cite{MA}, we have

\begin{theorem}\label{thm:formulaI} For $1\leq r\leq f$,
  $$(-1)^r\Phi_r=\sum\frac{(-1)^{a+b+c+\cdots}S_1^a S_2^b S_3^c\cdots}{1^a2^b3^c\cdots a!b!c!\cdots} \leqno(7)$$
  where $a,b,c,\ldots$ run over all positive integers such that satisfies $a+2b+3c+\cdots=r$.
 If $r>f$, then $\Phi_r=0$.
\end{theorem}

Set $G^r=\overbrace{G\times G\times\cdots\times G}^r$. Let $\mathfrak{S}_r$ be the symmetric group on $\{1,\cdots,r\}$.  Let $\tau=(i_1,i_2,\cdots,i_{t_1})(i_{t_1+1},\cdots,i_{t_2})\cdots(i_{t_{s-1}+1},\cdots,i_{t_s})$ be an element of $\mathfrak{S}_r$.   For $(g_1,\cdots,g_r)\in G^r$, we define
\[\chi((g_1,\cdots,g_r)^\tau)=sgn(\tau)\chi(g_{i_1}g_{i_2}\cdots g_{i_{t_1}})\chi(g_{i_{t_1+1}}\cdots g_{i_{t_2}})\cdots \chi(g_{i_{t_{s-1}+1}}\cdots g_{i_{t_s}}).\]

\begin{definition}\label{def:l-char}
The $r$-character $\chi^{(r)}$ of $\chi$ is a function on $G^r$ which is defined by
\[\chi^{(r)}((g_1,g_2,\cdots,g_r))=\sum_{\tau \in \mathfrak{S}_r}\chi((g_1,g_2,\cdots,g_r)^\tau),\] for any $(g_1,g_2,\cdots,g_r)\in G^r$.
The $1$-character $\chi^{(1)}$ is $\chi$.
\end{definition}

Let $\tau$ be an element in $\mathfrak{S}_r$. The type of decomposition of cycles of $\tau$ has  $c_1$ $1$-cycles,$\cdots$,$c_t$ $t$-cycles and $1c_1+\cdots+tc_t=r$. Then \begin{eqnarray*}sgn(\tau)
&=&(-1)^{c_1(1-1)+c_2(2-1)+c_3(3-1)+\cdots+c_t(t-1)}\\
&=&(-1)^{(1c_1+2c_2+3c_3+\cdots+tc_t)-(c_1+c_2+c_3+\cdots+c_t)}\\
&=&(-1)^{r}(-1)^{c_1+c_2+c_3+\cdots+c_t},\end{eqnarray*} and the length of the conjugate class containing $\tau$ is
$ \frac{r!}{1^{c_1} 2^{c_2} 3^{c_3}\cdots t^{c_t} c_1!c_2!c_3!\cdots c_t!}$.
Be identity (6), we have
\begin{eqnarray*}
&&\sum_{(g_1,g_2,\cdots,g_r)\in \overbrace{G\times G\cdots\times G}^r}\chi((g_1,\cdots,g_r)^\tau)x_{g_1} x_{g_2} \cdots x_{g_r}\\
&=&(-1)^r(-1)^{c_1+c_2+c_3+\cdots+c_t}S_1^{c_1} S_2^{c_2} \cdots S_t^{c_t}.
\end{eqnarray*}
Let $\tau,\tau'\in \mathfrak{S}_r$. Then \begin{eqnarray*}
&&\sum_{(g_1,g_2,\cdots,g_r)\in \overbrace{G\times G\cdots\times G}^r}\chi((g_1,\cdots,g_r)^\tau)x_{g_1} x_{g_2} \cdots x_{g_r}\\
&=&\sum_{(g_1,g_2,\cdots,g_r)\in \overbrace{G\times G\cdots\times G}^r}\chi((g_1,\cdots,g_r)^{\tau'})x_{g_1} x_{g_2} \cdots x_{g_r} \end{eqnarray*}if and only if $\tau$ and $\tau'$ are conjugate in $\mathfrak{S}_r$.

\begin{theorem}\label{phin} For $1\leq r\leq f$,
  $$ r!\Phi_r(x)=\sum_{g_1,g_2,\cdots g_r} \chi(g_1,g_2,\cdots g_r)x_{g_1}x_{g_2}\cdots x_{g_r}\leqno(8) $$
\end{theorem}

\begin{proof}
Let CS$_r$ be the set of representatives of conjugate classes of $\mathfrak{S}_r$. Set $l(\tau)=\frac{r!}{1^{c_1} 2^{c_2} 3^{c_3}\cdots t^{c_t} c_1!c_2!c_3!\cdots c_t!}$. Then
\begin{eqnarray*}
&&(-1)^r \sum_{(g_1,g_2,\cdots,g_r)\in \overbrace{G\times G\cdots\times G}^r} \chi(g_1,g_2,\cdots g_r)x_{g_1} x_{g_2} \cdots x_{g_r}\\
&=& (-1)^r \sum_{(g_1,g_2,\cdots,g_r)\in \overbrace{G\times G\cdots\times G}^r} \sum_{\tau\in \mathfrak{S}_n}\chi((g_1,\cdots,g_r)^\tau)x_{g_1} x_{g_2} \cdots x_{g_r}\\
&=&(-1)^r \sum_{\tau\in \mathfrak{S}_r}\sum_{(g_1,\cdots,g_r)\in \overbrace{G\times G\cdots\times G}^r}\chi((g_1,\cdots,g_r)^\tau)x_{g_1} x_{g_2} \cdots x_{g_r}\\
&=&(-1)^r  \sum_{\tau\in CS_r}l(\tau)\sum_{(g_1,\cdots,g_r)\in \overbrace{G\times G\cdots\times G}^r} \chi((g_1,g_2,\cdots,g_r)^\tau)x_{g_1} x_{g_2} \cdots x_{g_r}\\
&=&(-1)^r  \sum_{\tau\in \mathfrak{S}_r}l(\tau)(-1)^r(-1)^{a+b+c+\cdots}(-1)^r(-1)^{c_1+c_2+c_3+\cdots+c_t}S_1^{c_1} S_2^{c_2} \cdots S_t^{c_t}\\
&=&r!(-1)^r\Phi_r.
\end{eqnarray*}
\end{proof}

\section{Orthogonal relations and invariants of $\mathcal{P}$-characters}\label{sec:orthdegree}
In the rest, $\chi_t[l]$ will be denoted by $\chi_{lt}$.


Let $\mathfrak{p}_{ij}=\mbox{Tr}(\mathcal{A}_i\mathcal{A}_j)$ and let $\mathfrak{P}=(\mathfrak{p}_{ij})$ be a matrix. Let $\mathcal{R}$ be the matrix with $\mathfrak{r}_j(\mathfrak{a}_i)$ in the $(ij)$-entry. Then $\mathcal{R}\mathcal{R'}=\mathfrak{P}$ and $\mathcal{R}$ is invertible by Theorem \ref{thm:crsemi}.
 Let $\mathcal{S}=(\varsigma_{ij})$ be a matrix such that $\mathcal{RS'}=I$. Then $\mathcal{S'R}=I$ and it follows equations:
\begin{equation*}\label{equ:orthog} \sum_{t=1}^n \lambda_{it}\varsigma_{jt}=\delta_{ij}\mbox{ and }\sum_{i=1}^n\lambda_{ij}\varsigma_{ik}=\delta_{jk}.\leqno(4)\end{equation*}

\begin{theorem}\label{Othmatrix}
 Set $\varsigma_l=(\varsigma_{1l},\varsigma_{2l},\cdots,\varsigma_{nl})$ for $l=1,\cdots,n$.
Then
\begin{itemize}
\item[(i)]{$\varsigma_1,\cdots,\varsigma_n$ are eigenvectors of $\mathfrak{A'}$ belonging to  $\gamma_1(x),\cdots,\gamma_n(x)$ respectively i.e. $\varsigma_l\mathfrak{A'}=\gamma_l(x)\varsigma_l$ where $\mathfrak{A}=\sum_{i=1}^sx_i\mathcal{A}_i$;}
\item[(ii)]{$\varsigma_1,\cdots,\varsigma_n$  are determined by equations:
\begin{equation*}\label{equ:orthogonalM} x_{i}\gamma_l(x)=\sum_ja_{ij}(x)x_{j}, \leqno(5)\end{equation*} where $a_{ij}(x)$ is the element on $(ij)$-entry of $\mathfrak{A}$.}
\end{itemize}
\end{theorem}
\begin{proof} Since $\lambda_{jt}\lambda_{kt}=\sum_{s} a_{sjk}\lambda_{st}$, by equations (4), the following equation hold:
\begin{equation*} a_{ijk}=\sum_t\varsigma_{it}\lambda_{jt}\lambda_{kt}.\end{equation*} Then $\sum_j a_{ijk}\varsigma_{jl}=\sum_j\varsigma_{jl}\sum_t\varsigma_{it}\lambda_{jt}\lambda_{kt}=\sum_t\varsigma_{it}\lambda_{kt}(\sum_j\lambda_{jt}\varsigma_{jl})=\varsigma_{il}\lambda_{kl}.$ So \begin{equation*}\label{equ:orthog2}\varsigma_{il}\gamma_l(x)=\sum_k\varsigma_{il}\lambda_{kl}x_k=\sum_k\sum_j a_{ijk}\varsigma_{jl}x_k=\sum_j(\sum_k a_{ijk}x_k)\varsigma_{jl}.\leqno(6)\end{equation*} So $\varsigma_l$ is solution of the equation (5) and it is eigenvalue vectors of $\mathfrak{A'}$ with respect to the eigenvalue $\gamma_l(x)$ for $1\leq l\leq n$ since  $\varsigma_l\gamma_l(x)=\varsigma_l \mathfrak{A'}$. Similar to the proof in \ref{Property:A}, we can show that $\varsigma_1,\cdots,\varsigma_n$  are determined by
$ x_{i}\gamma_l(x)=\sum_ja_{ij}(x)x_{j}$.

\end{proof}

By comparing the coefficient of $x_k$ on both side of equations (6), it follows :
\begin{equation*}\varsigma_{il}\frac{\chi_{kl}}{\chi_l(1)}=\varsigma_{il}\lambda_{kl} =\sum_j a_{ijk}\varsigma_{jl}=\sum_j \frac{\ell_{i'jk}}{\ell_i}\varsigma_{jl}. \mbox{ i.e. } \ell_i\varsigma_{il}{\chi_{kl}}=\chi_l(1)\sum_j a_{j'i'k}\ell_j\varsigma_{jl}.\end{equation*}
By multiplying $x_k$ on both sides of above equations and taking sum over $F$, it follows $$\ell_i\varsigma_{il}\sum_k\frac{\chi_{kl}}{\chi_l(1)}x_k=\sum_j \sum_k a_{j'i'k}x_k\ell_j\varsigma_{jl}.$$ This equation can be written into the follow version of matrices
$$(\ell_1\varsigma_{1l},\cdots,\ell_n\varsigma_{nl})\gamma_l(x)=(\ell_1\varsigma_{1l},\cdots,\ell_n\varsigma_{nl})(\sum_ka_{j'i'k}x_k)_{1\leq j,i\leq n}.$$ So $(\ell_1\varsigma_{1l},\cdots,\ell_n\varsigma_{nl})$ is the eigenvector of the matrix $(\sum_ka_{j'i'k}x_k)_{1\leq j,i\leq n}$ belonging to $\gamma_l(x)$.
Since $\chi_{i'l}{\chi_{kl}}=\chi_l(1)\sum_j a_{j'i'k}\chi_{j'l}$,($\chi_{1'l},\cdots,\chi_{n'l}$) is also the eigenvector of the matrix $(\sum_ka_{j'i'k}x_k)_{1\leq j,i\leq n}$ belonging to $\gamma_l(x)$. By similar proof in Theorem \ref{Othmatrix}, we can choose $\ell_i\varsigma_{il}=z_{l}\chi_{i'l}$ for some   $z_{l}$ and $z_{l}$ can be taken to be the form $z_l=\frac{e_l}{\ell}$. Then  $\mathcal{S}=(\frac{e_l}{\ell\ell_i}\chi_{i'l})$.

\begin{theorem}\label{thm:orth}(Orthogonal relation)

\begin{itemize}
\item[(i).]{$\sum_{i} \frac{\chi_{ij}\chi_{i'l}}{\ell_i}=0$,}
\item[(ii).]{$\sum_{i} \frac{\chi_{ij}\chi_{i'j}}{\ell_i}=\frac{\ell \chi_j(1)}{e_j}$,}
\item[(iii).]{$\sum_l{\frac{e_l}{\chi_l(1)}}\chi_{il}\chi_{jl}=0(i\neq j')$,}
\item[(iv).]{$\sum_l{\frac{e_l}{\chi_l(1)}}\chi_{il}\chi_{i'l}=\ell\ell_i$,}
\item[(v).]{$\sum_l{e_l}\chi_{il}=0$,}
\item[(vi).]{$\sum_le_l\chi_l(1)=\ell$.}

\end{itemize}
\end{theorem}

\begin{proof} Equations (i),(ii) follows from $\mathcal{S'R}=I$,equations (iii),(iv) follows from $\mathcal{RS'}=I$. Equations (v) and (vi) follows by taking $i=1$ in (iii) and (iv) respectively. \end{proof}

\begin{theorem}\label{thm:multox}

For $1\leq i_1,i_2,\cdots ,i_r\leq n$, the following equation holds true

$$ \ell\ell_{i_1i_2\cdots i_r}=\sum_l\frac{e_l}{\chi_l^{r-1}(1)}\chi_{i_1l}\cdots \chi_{i_rl}.$$

\end{theorem}

\begin{proof}

Multiplying by $\frac{e_t\chi_{i_3t}}{\chi^2_{1t}}$ on both sides  $\chi_{i_1t}\chi_{i_2t}=\chi_{1t}\sum_{l=1}^{s}a_{li_1i_2}\chi_{lt}$ and summing over $t$, it follows
\begin{align*}
    \sum_t\frac{e_t}{\chi_{1t}^{2}}\chi_{i_1t} \chi_{i_2t}\chi_{i_3t} &= \sum_t\frac{e_t}{\chi_{1t}}\sum_l a_{li_1i_2}\chi_{lt}\chi_{i_3t}\\
    & = \sum_l a_{li_1i_2}\sum_t\frac{e_t}{\chi_{1t}}\chi_{lt}\chi_{i_3t}\\
    \mbox{By Theorem \ref{thm:orth}(iii)(iv)} &= \sum_l\ell_{l'i_1i_2}\frac{\ell\ell_{li_3}}{\ell_l} \\
    \mbox{By Lemma \ref{basicEqs}(ii)}&=\ell\ell_{i_1i_2i_3}
  \end{align*}

Next using induction on $n$. If
$$\ell \ell_{i_1i_2\cdots i_n}=\sum_l\frac{e_l}{\chi_l^{n-1}(1)}\chi_{i_1l}\cdots \chi_{i_nl},$$
then
\begin{align*}
  \sum_l\frac{e_l}{\chi_l^{n}(1)}\chi_{i_1l}\cdots \chi_{i_nl}\chi_{i_{n+1}l} &= \sum_l\frac{e_l}{\chi_l^{n}(1)}\chi_{i_1l}\cdots \chi_{i_{n-1}l}\sum_s\frac{\chi_{1l}\ell_{s'i_ni_{n+1}}}{\ell_s}\chi_{sl} \\
  &=\sum_s\frac{\ell_{s'i_ni_{n+1}}}{\ell_s}\sum_l\frac{e_l}{\chi_l^{n-1}(1)}\chi_{i_1l}\cdots \chi_{i_{n-1}l}\chi_{sl} \\
  &=\sum_s\frac{\ell_{s'i_ni_{n+1}}}{\ell_s}\ell \ell_{i_1i_2\cdots i_{n-1}s} \\
   \mbox{Lemma \ref{basicEqs}(ii)}&= \ell \ell_{i_1i_2\cdots i_{n+1}}
\end{align*}
\end{proof}

\begin{lemma}\label{lem:pij}
 For $1\leq i,j\leq n$,
$p_{ij}=\sum_{t=1}^{n}\frac{\ell_{ijtt'}}{\ell_t}$. In particular, $p_i=p_{i'}=p_{i1}=\sum_l \frac{\chi_{il}}{\chi_l(1)}$.
\end{lemma}

\begin{proof} The first equation follows from Lemma \ref{property:p}.
By Theorem \ref{thm:multox} and (iv) of \ref{thm:orth}, it follows
 \[p_{ij}=\sum_{t=1}^{n}\frac{\ell_{ijtt'}}{\ell_t}=\sum_{t=1}^{n}\frac{1}{\ell_t\ell}\ell\ell_{ijtt'}=\sum_{t=1}^{n}\frac{1}{\ell_t\ell}\sum_l\frac{e_l}{\chi_l^3(1)}\chi_{il}\chi_{jl}\chi_{tl}\chi_{t'l} =\sum_l\frac{\chi_{il}\chi_{jl}}{\chi_l^2(1)}.\] In particular, by taking $j=1$ in above equation, we have $p_i=p_{i1}=p_{i1'}=p_{i'1}=p_{i'}=\sum_l \frac{\chi_{il}}{\chi_l(1)}$.

\end{proof}

\begin{theorem}\label{thm:degree} The following equations hold for $1\leq t\leq n$,
$\sum_{\gamma=1}^n \frac{1}{\ell}\lambda_{\gamma t}\frac{p_\gamma}{\ell_\gamma}=\frac{1}{\chi_t(1)e_t}$
\end{theorem}

\begin{proof}

It follows from Lemma \ref{lem:pij} that

 \begin{align*}
 &\sum_{\gamma=1}^n \frac{1}{\ell}\lambda_{\gamma t}\frac{p_\gamma}{\ell_\gamma}
 = \sum_{\gamma=1}^n \frac{1}{\ell}\frac{\chi_{\gamma t}}{\chi_t(1)}\frac{p_\gamma}{\ell_\gamma}\\
 &= \sum_{\gamma=1}^n \frac{1}{\ell\ell_\gamma}\frac{\chi_{\gamma t}}{\chi_t(1)}\sum_j\frac{\chi_{\gamma'j}}{\chi_j(1)}
 = \sum_{j=1}^n \frac{1}{\ell}\frac{\sum_\gamma\chi_{\gamma t}\chi_{\gamma'j}}{\chi_t(1)\chi_j(1)}\\
\mbox{Theorem\ref{thm:orth}(i)(ii)} &= \frac{1}{e_t\chi_t(1)}.\end{align*}
\end{proof}

Let $M$ be the permutation matrix such that $(p_{ij'})M=(p_{ij})$. Then $M^2=I$.

\begin{theorem}\label{Mainthm}
The polynomial $D_\mathcal{P}(x)=|xI-Diag(\ell\ell_1,\cdots,\ell\ell_n)M\mathcal{SS}'|$ has the following factorization:
$D_\mathcal{P}(x)=(x-\chi_1(1)e_1)\cdots(x-\chi_n(1)e_n)$.

\end{theorem}
\begin{proof}
Since $\mathcal{R'R}\mbox{Diag}(\frac{1}{\ell\ell_1},\cdots,\frac{1}{\ell\ell_n})\mbox{Diag}(\ell\ell_1,\cdots,\ell\ell_n)\mathcal{SS'}=1$ and $\mathcal{R'R}=(p_{ij'})M$, it suffices to show the characteristic polynomial of the matrix of the form: $(p_{ij'})\mbox{Diag}(\frac{1}{\ell\ell_1},\cdots,\frac{1}{\ell\ell_n})$ has the following factorization \[\mid xI-(p_{ij'})\mbox{Diag}(\frac{1}{\ell\ell_1},\cdots,\frac{1}{\ell\ell_n})\mid =(x-\frac{1}{\chi_1(1)e_1})\cdots(x-\frac{1}{\chi_n(1)e_n}).\]


By equation (\ref{equ:pij}) in section \ref{sec:Frobenius theory} and Lemma \ref{lem:pij}, we have \[\frac{p_{ij'}}{\ell\ell_{j'}}=\sum_{\gamma}\frac{\ell_{ij'\gamma}}{\ell\ell_j}\frac{p_\gamma}{\ell_\gamma}=\frac{1}{\ell}\sum_\gamma a_{ji\gamma}\frac{p_\gamma}{\ell_\gamma}.\] So $(\frac{p_{ij'}}{\ell\ell_{j'}})=\frac{1}{\ell}\sum_\gamma \mathcal{A'}_\gamma\frac{p_\gamma}{\ell_\gamma}$. By taking $x_\gamma=\frac{p_\gamma}{\ell_\gamma}$ in the matrix $\mathfrak{A}$, it follows from  (i) of Theorem \ref{Property:A} and \ref{thm:degree} that\begin{eqnarray*}
 && \mid xI-(\frac{p_{ij'}}{\ell\ell_{j'}})\mid
 = \prod_t(x-\sum_{\gamma=1}^n \frac{1}{\ell}\lambda_{\gamma t}\frac{p_\gamma}{\ell_\gamma})\\
 &=& \prod_t(x- \frac{1}{\chi_t(1)e_t}).\end{eqnarray*}


\end{proof}
\begin{remark}
In the next section we will show that $\prod_t(x- \frac{1}{\chi_t(1)e_t})$ is the characteristic polynomial of the Casimir element of $k\mathcal{P}$.
\end{remark}

Let $\Theta_G(x)=|M_G(x)|=\prod_i\Phi_i^{c_i}(x)$ be a factorization of irreducible polynomials over complex field $\mathbb{C}$ with $f_i$ as the degree of $\Phi_i(x)$.
Then each $\Phi_i(x)$ defines a character $\chi_i^\mathcal{P}$ and $(\frac{\chi_i^\mathcal{P}(g_1)}{\chi_i^\mathcal{P}(1)},\cdots,\frac{\chi_i^\mathcal{P}(g_n)}{\chi_i^\mathcal{P}(1)})$ is a solution of equations (1). So there is a unique $\gamma_t$ such that $\gamma_t=(\frac{\chi_i^\mathcal{P}(g_1)}{\chi_i^\mathcal{P}(1)},\cdots,\frac{\chi_i^\mathcal{P}(g_n)}{\chi_i^\mathcal{P}(1)})$. In this case, we call $\Phi_i(x)$ a $\gamma_t$-factor of $\Theta_G(x)$. Since $\gamma_1,\cdots,\gamma_n$ is linearly independent, each $\Phi_i(x)$ belongs to only one $\gamma_t$ for some $1\leq t\leq n$ and all $\gamma_t$-factors consist of the set of all irreducible factors of $\Theta_G(x)$.
If $\Phi_i(x)$ is $\gamma_t$-factor, then $\Phi_i(x)=\gamma_t^{f_i}(x)$ by assuming that $x_A=x_B$ whenever $A$ and $B$ belong to the same partition class of $\mathcal{P}$. Furthermore, we have  $$\Theta_G(x)=|M_G(x)|=\prod_i\Phi_i^{c_i}(x)=\prod_{i=1}^n\gamma_i^{m_i}(x)$$ with the assumption that $x_A=x_B$ whenever $A$ and $B$ belong to the same partition class of $\mathcal{P}$.

\begin{proposition}\label{prop:corthogonal}
  $$\sum_i m_i=\ell, \sum_i m_i\frac{\chi^\mathcal{P}_i(g_j)}{\chi^\mathcal{P}_i(1)}=0$$ for $R\not=E$.
\end{proposition}
\begin{proof}
  Replace $x_E$ by $x_E+u$ on both side of the equation $\Theta_G(x)=\prod_i(\Phi_i^{ e_i})$. Then
  $$\Theta_G(x)=
    \left|\begin{array}{ccccc}
      x_E+u & x_A & x_B & \cdots & \cdots\\
       & x_E+u & x_A & \cdots & \cdots\\
       & & x_E+u & & \\
       & \ast & & \ddots & \vdots\\
       & & & & x_E+u\\
    \end{array}\right|$$$=\prod_i(\Phi_i^{ e_i}(x_E+u,x_{g_2},\cdots,x_{g_m}))$
  .

 Assume that $x_A=x_B$ whenever $A$ and $B$ belong to the same partition class of $\mathcal{P}$ in above equation. By  Theorem \ref{thm:abbaExp}, it follows
 $$\Theta_G(x)=
    \left|\begin{array}{ccccc}
      x_E+u &  &  & \cdots & \cdots\\
       & x_E+u &  & \cdots & \cdots\\
       & & x_E+u & & \\
       & \ast & & \ddots & \vdots\\
       & & & & x_E+u\\
    \end{array}\right|= \prod_{i=1}^n (u-\gamma_i(x))^{m_i} )
  .$$
  By comparing the coefficients of $x_E^{h-1}x_{g_i}$ on both sides, it follows $$0=\sum_i m_i\lambda_{ji}=\sum_i m_i\frac{\chi_i^\mathcal{P}(g_j)}{\chi_i^\mathcal{P}(1)}$$ for $g_i\neq E$. Since the degree of $\Theta(x)$ is the order $\ell$ of $G$, we have $\sum_i m_i=\ell$.
\end{proof}

\begin{theorem}\label{cor:etchi1}  Let $e_i,\chi_i(1)$ be as in Theorem \ref{thm:orth}. Then
$e_i\chi_i(1)=m_i$ for $1\leq i\leq n$.

  \end{theorem}
\begin{proof}
Since the determinant of $\mathcal{R}$ is nonzero, the equations $$(x_1,\cdots,x_n)\mathcal{R}=(\ell,0,\cdots,0)$$ has a unique solution. By (v) and (vi) of Theorem \ref{thm:orth}, $(e_1\chi_1(1),\cdots,e_n\chi_n(1))$ is the solution of above equations. By Proposition \ref{prop:corthogonal}, $(m_1,\cdots,m_n)$ is also the solution of above equations. So $(e_1\chi_1(1),\cdots,e_n\chi_n(1))=(m_1,\cdots,m_n)$.
\end{proof}

\section{Frobenius polynomial}\label{sec:FrobPoly}
In this section, we will introduce the definition of characteristic polynomial of an element in an algebra,  Frobenius polynomial and generalized Casimir element of a partition algebra $k\mathcal{P}$. We will prove a formula of the characteristic polynomial of the generalized Casimir element and show that the partition algebra is determined by its Frobenius polynomial.

Let $\mathcal{P}={\mathcal{P}_1\cup\cdots\cup\mathcal{P}_n}$ be a good partition of Cl($G$) with $g_i$ as a representative of $\mathcal{P}_i$ and $g_1=E$. As in section \ref{sec:Frobenius theory}, the partition class sums  $\hat{C}_1,\cdots,\hat{C}_n$ consist of a basis of the partition algebra  $k\mathcal{P}$.

Let $\mathcal{A}=k\mathcal{P}$. We will keep the same assumption on $k$ as in \ref{sec:Frobenius theory}. Then $\mathfrak{L}(\hat{C_j})$ is equal to the matrix $\mathcal{A}_j$ defined in section \ref{sec:Frobenius theory}. Let
$\mathfrak{A}=\sum_{i=1}^n x_i\mathcal{A}_i$ as defined in Theorem \ref{Property:A}.

The element $\sum_{i=1}^n \frac{1}{\ell_i\ell}\hat{C}_i\hat{C}_{i'}$ is called generalized Casimir element of the partition algebra $\mathcal{A}=k\mathcal{P}$.

\begin{theorem}
  The characteristic polynomial of the generalized Casimir element of $k\mathcal{P}$ is $\prod_t(\lambda- \frac{1}{\chi_t(1)e_t})$.
\end{theorem}

\begin{proof}
Let $\mathfrak{L}$ be the regular representation of $k\mathcal{P}$. Assume that $\mathcal{M}$ is the invertible matrix such that \[\mathcal{M}\mathcal{A}_i\mathcal{M}^{-1}=\begin{pmatrix}
\lambda_{i1}& 0& 0&\cdots &0\\
0& \lambda_{i2}& 0&\cdots &0\\
\cdot& \cdot& \cdot&\cdots &\cdot\\
0& 0& 0&\cdots &\lambda_{is}\\
\end{pmatrix}\] for $1\leq i \leq s$. Then the characteristic polynomial of the Casimir element is
$|\lambda I- \mathfrak{L}(\sum_{i=1}^n \frac{1}{\ell_i\ell}\hat{C}_i\hat{C}_{i'})|$ and

\begin{align*}
  &|\lambda I- \mathfrak{L}(\sum_{i=1}^n \frac{1}{\ell_i\ell}\hat{C}_i\hat{C}_{i'})|
  = |\lambda I- (\sum_{i=1}^n \frac{1}{\ell_i\ell}\mathcal{A}_i\mathcal{A}_{i'})| \\
  &= |\lambda I- (\sum_{i=1}^n \frac{1}{\ell_i\ell}\sum_{t=1}^n a_{tii'}\mathcal{A}_t)|\\
  &=|\lambda I- (\sum_{i=1}^n \frac{1}{\ell_i\ell}\sum_{i=1}^n a_{tii'}\mbox{Diag}(\lambda_{t1},\cdots,\lambda_{tn})|\\
  &=|\lambda I- \mbox{Diag}((\sum_{i=1}^n \frac{1}{\ell_i\ell}\sum_{t=1}^n a_{tii'}\lambda_{t1},\cdots,(\sum_{i=1}^n \frac{1}{\ell_i\ell}\sum_{t=1}^n a_{tii'}\lambda_{tn})|.
\end{align*}

On the other hand, we have
\begin{align*}
  \sum_{i=1}^n \frac{1}{\ell_i\ell}\sum_{t=1}^n a_{tii'}\lambda_{tj}
  &= \sum_{i=1}^n \frac{1}{\ell_i\ell^2}\sum_{t=1}^n \frac{\ell\ell_{t'ii'}\chi_{tj}}{\ell_t\chi_j(1)} \\
  &= \sum_{i=1}^n \frac{1}{\ell_i\ell^2}\sum_{t=1}^n \frac{\chi_{tj}}{\ell_t\chi_j(1)}\sum_{l=1}^n\frac{e_l}{\chi_l^2(1)}\chi_{t'l}\chi_{il}\chi_{i'l}\\
 &= \sum_{i=1}^n \frac{1}{\ell_i\ell^2\chi_j(1)}\sum_l\frac{e_l}{\chi_l^2(1)}\chi_{il}\chi_{i'l}\sum_t\frac{\chi_{tj}\chi_{t'l}}{\ell_t}\\
 &= \sum_{i=1}^n \frac{1}{\ell_i\ell^2\chi_j(1)}\frac{e_j\chi_{ij}\chi_{ij'}}{\chi_j^2(1)}\frac{\ell \chi_j(1)}{e_j}\\
 &=\frac{1}{\ell\chi_j^2(1)}\sum_i\frac{\chi_{ij}\chi_{ij'}}{\ell_i}\\
 &=\frac{1}{\chi_j(1)e_j}.
\end{align*}

So we have $$|\lambda I- \mathfrak{L}(\sum_{i=1}^n \frac{1}{\ell_i\ell}\hat{C}_i\hat{C}_{i'})|=\prod_t(\lambda- \frac{1}{\chi_t(1)e_t}).$$

\end{proof}

\begin{definition}
The determinant $|\mathfrak{A}|$ is called {\bf Frobenius polynomial} of the partition algebra $\mathcal{P}$. It is denoted by $F_G^\mathcal{P}(x)$. We denote by $F_G(x)$ the Frobenius polynomial of the trivial partition algebra $Z(kG)$.
\end{definition}

By Theorem \ref{thm:semisimPart}, $kP$ is semisimple. Let $\mathfrak{r}_l:k\mathcal{P}\rightarrow k (\hat{C}_i\mapsto\lambda_{il})$ be simple representation of $k\mathcal{P}$ and let $\gamma_l(x)=\sum_{i=1}^n\lambda_{il}x_i)$  for $1\leq l\leq n$.

\begin{theorem}
The Frobenius polynomial $F_G^\mathcal{P}(x)$ of $k\mathcal{P}$ is equal to  the norm form of $k\mathcal{P}$.
\end{theorem}

\begin{proof}
Since $kP$ is semisimple and , the regular representation $\mathcal{L}$ is equivalent to $\oplus_{j=1}^n \mathfrak{r}_j$. This implies that there exists an invertible matrix $U$ such that $U\mathcal{A}_iU^{-1}=\mbox{Diag}(\lambda_{i1},\cdot,\lambda_{in})$. So
 $\mid \mathfrak{A}\mid=\prod_{t=1}^n\gamma_t(x)$. It follows from Theorem \ref{thm:semisimPart} that \[\gamma_1(x)=\sum_{i=1}^n\lambda_{i1}x_i,\cdots, \gamma_n(x)=\sum_{i=1}^n\lambda_{in}x_i\] are prime to each other. The definition of minimal polynomial implies that $F_G^\mathcal{P}(x)=\prod_{t=1}^n\gamma_t(x)=\prod_{t=1}^n(\sum_{i=1}^n\lambda_{it}x_i)$ is the norm form of $k\mathcal{P}$.
\end{proof}


\begin{proposition}\label{contructureant}
The Frobenius polynomial $F_G^\mathcal{P}(x)$ and the matrices $\{\mathcal{A}_i\}_{i=1}^s$  are determined by each other.
\end{proposition}
\begin{proof} It is obvious that $F_G^\mathcal{P}(x)$ is determined by $\{\mathcal{A}_i\}_{i=1}^s$. Given $F_G^\mathcal{P}(x)$. Then it can be decomposed into product of irreducible polynomials $F(x)=\prod_{t=1}^n\gamma_t(x)$.  Let $\gamma_t(x)=\sum_{i=1}^n\lambda_{it}x_i$ for $1\leq i \leq n$.
Since $k\mathcal{P}$ is semisimple, by Theorem \ref{Property:A} and Theorem \ref{thm:crsemi}, the matrix
\[\begin{pmatrix}
\lambda_{1,1}&\lambda_{2,1}&.&.&.&\lambda_{n,1}\\
\lambda_{1,2}&\lambda_{2,2}&.&.&.&\lambda_{n,2}\\
.&.&.&.&.&.\\
\lambda_{1,i}&\lambda_{2,i}&.&.&.&\lambda_{n,i}\\
.&.&.&.&.&. \\
\lambda_{1,n}&\lambda_{2,n}&.&.&.&\lambda_{n,n}
\end{pmatrix}
\] is reversible. So all $a_{\ell ij}(1\leq \ell,i,j\leq n)$ are the unique solution of the equations
\begin{equation}\label{equ:1RegRep2}\lambda_{it}\lambda_{jt}=\sum_{\ell=1}^n a_{\ell ij} \lambda_{it},1\leq i,j,t\leq n.\end{equation}
By definition of $\mathcal{A}_i$,
So $\{\mathcal{A}_i\}_{i=1}^s$ is determined by $F_G^\mathcal{P}(x)$.
\end{proof}

In the next, we will introduce the generalized commutators and the definition of the number $p_{i_1\cdots i_r}$ and prove that the Frobenius polynomial is determined by $p_{ijl}$.



Let $\mathcal{P}={\mathcal{P}_1\cup\cdots\cup\mathcal{P}_n}$ be a good partition of Cl($G$).

Let $\mathcal{A}=k\mathcal{P}$ and let $\mathcal{L}$ be its regular representation  with respect to the basis $\hat{C_1},\cdots,\hat{C_n}$. Then $\mathcal{L}(\hat{C_i})=\mathcal{A}_i$ is a $n\times n$-matrix. Assume that \[m_{\mathcal{A}}(\lambda,x_1,\cdots,x_n)=\lambda^n-\lambda^{n-1}\rho_1(x_1,\cdots,x_n)+\cdots + (-1)^n \rho_n(x_1,\cdots,x_n)\] is the minimal polynomial. Then \begin{equation}\label{equ:elesym}\rho_i(x_1,\cdots,x_n)=\sigma_i(\gamma_1(x),\cdot,
\gamma_n(x))(i=1,\cdots,n)\end{equation} and $\rho_n(x_1,\cdots,x_n)=F_G^\mathcal{P}(x)$, where $\sigma_1(y_1,\cdots,y_n)=y_1+\cdots+y_n$,$\cdots$,$\sigma_n=y_1y_2...y_n$ are the elementary symmetric polynomials.

\begin{lemma}\cite{MA}\label{lemm:psum}
Let $\mathfrak{s}_i(y_1,\cdots,y_n)=y_1^i+\cdots+y_n^i$ for $i=1,\cdots,n$. Then $$(-1)^l\sigma_l(y_1,\cdots,y_n)=\sum_\lambda \frac{\varepsilon_\lambda \prod \mathfrak{s}_i^{m_i}}{z_\lambda},$$ where $\lambda=1m_1+2m_i+\cdots $ runs over all partitions of $l$,  $\varepsilon_\lambda=(-1)^{m_1+\cdots+m_l}$ and $z_\lambda=\prod i^{m_i}m_{i}£¡$.
\end{lemma}

\begin{definition}\label{def:gpijk}
An element $g$ of $G$ is called generalized commutator with respect to $\mathcal{P}$ if $g=xy$ with $x\in \mathcal{P}_i$ and $y\in \mathcal{P}_{i'}$ for some $1\leq i \leq n$. Assume that $\hat{C}_{t_1},\cdots,\hat{C}_{t_s}$ are all partition classes appearing in the product $\hat{C}_{i_1}\cdots \hat{C}_{i_r}$ with non-zero coefficient, where $1\leq i_1,\cdots,i_r \leq n$. Then  $p_{i_1,\cdots,i_r}$ is the number of solutions of the equation \begin{equation}\label{def:commutators} xy=g_{t_l}\end{equation} where
$x\in \mathcal{P}_i$,$y\in \mathcal{P}_{i'}$ for  $1\leq i \leq n$ and $g_{t_l}$ is a representative of $ \mathcal{P}_{i_l}$ for $1\leq l\leq s$.

\end{definition}

If $r=1\mbox{ or }2$, the definitions of $p_{i_1},p_{i_1,i_2}$ above are the same as definitions in Lemma \ref{property:p}.

\begin{lemma}
 \[\mathfrak{s}_r(\gamma_1(x),\cdots,\gamma_n(x))=\sum_{1\leq i_1,\cdots,i_r \leq n} \frac{1}{\ell}p_{i_1,\cdots,i_r}x_{i_1}\cdots x_{i_r}.\]
\end{lemma}

\begin{proof}
The coefficient of monomial $x_{i_1}x_{i_2}\cdots x_{i_r}$ in $\mathfrak{s}_r(\gamma_1(x),\cdots,\gamma_n(x))$ is the trace of $\mathcal{A}_{i_1}\cdots \mathcal{A}_{i_r}$.


The coefficient of $\hat{C}_j$ in $\hat{C}_{i_1}\hat{C}_{i_2}\cdots \hat{C}_{i_r}$ is $\frac{\ell_{j'i_1,\cdots,i_r}}{\ell_j}$. By Lemma \ref{lem:ncommutators}, each $g\in \mathcal{P}_j$ can be written into $\sum_{i=1}^n\frac{\ell_{j i i'}}{\ell_i}$ generalized commutators. So $\mathcal{P}_j$ contributes $\frac{\ell_{j'i_1,\cdots,i_r}\sum_{i=1}^n\frac{\ell\ell_{j i i'}}{\ell_i}}{\ell_j}$\par  to $p_{i_1,\cdots,i_r}$. Thus we have the following equation:
\begin{eqnarray*}
&&p_{i_1,\cdots,i_r}= \sum_{j=1}^n\frac{\ell_{j'i_1,\cdots,i_r}}{\ell_j}\sum_{i=1}^n\frac{\ell_{j i i'}}{\ell_i}\\
&=&\sum_{i=1}^n \frac{1}{\ell_i}\sum_{j=1}^n \frac{\ell_{j'i_1,\cdots,i_r}\ell_{j i i'}}{\ell_j}\\
&=&\sum_{i=1}^n \frac{\ell_{ii'i_1,\cdots,i_r}}{\ell_i}\\
&=&\sum_{i=1}^n \frac{1}{\ell\ell_i}\sum_{l=1} \frac{e_l}{\chi_l(1)^{r+1}}\chi_{i_1l}\cdots\chi_{i_rl}\chi_{il}\chi_{i'l}\\
&=&\frac{1}{\ell}\sum_{l=1} \frac{e_l}{\chi_l(1)^{r+1}}\chi_{i_1l}\cdots\chi_{i_rl}\sum_{i=1}^n\frac{1}{\ell_i}\chi_{il}\chi_{i'l}\\
&=& \sum_{l=1} \frac{\chi_{i_1l}}{\chi_l(1)}\cdots \frac{\chi_{i_rl}}{\chi_l(1)}\\
&=&\sum_{l=1}^n \lambda_{i_1 l}\lambda_{i_2 l}\cdots \lambda_{i_r l}=\mbox{Tr}(A_{i_1}\cdots A_{i_r})
\end{eqnarray*}
\end{proof}

\begin{theorem}\label{thm:pijk}

The Frobenius polynomial $F^\mathcal{P}_G(x)$ is determined by the numbers $p_{ijl}$ where $1\leq i,j,l\leq n$.

\end{theorem}

\begin{proof}
We call polynomials  \[\rho_1(x_1,\cdots,x_n),\rho_2(x_1,\cdots,x_n)\mbox{ and }\rho_3(x_1,\cdots,x_n)\] the first-three-power-sum polynomials.

By \cite{H1}, polynomials $\rho_i(x_1,\cdots,x_n)(i\geq 4)$ can be determined by  first-three-power-sum polynomials. Since $F_G^\mathcal{P}(x)=\rho_n(x_1,\cdots,x_n)$  is determined by first-power-sum polynomials. By lemma \ref{lemm:psum} and equation (\ref{equ:elesym}),  first-power-sum polynomials are determined by \[\mathfrak{s}_1(\gamma_1(x),\cdots,\gamma_n(x)),\mathfrak{s}_2(\gamma_1(x),\cdots,\gamma_n(x)) \mbox{, } \mathfrak{s}_3(\gamma_1(x),\cdots,\gamma_n(x)),\] and they are determined by $p_{ijl}$. So  $F_G^\mathcal{P}(x)$ is determined by $p_{ijl}$.
\end{proof}

\section{On ordinary characters}\label{sec:degreepoly}

Assume that $\mathcal{P}$ is a trivial partition of Cl($G$) i.e. each partition class $\mathcal{P}_i$ of the good partition $\mathcal{P}$ consists of only one conjugate class of $G$ . The characters determined by the trivial partition are called ordinary characters in this paper. The definition equations of the ordinary characters can be simplified. So the polynomials related to the ordinary characters have different form and some special properties. We will list these differences in this section at first. Then we will introduce a series of invariants of a finite groups and show the chief factors and its multiplicity of a finite group is determined by these invariants.


 The  ordinary characters are defined by the following equations
 $$\ell_ix_i\ell_jx_j=x_1\sum_{l=1}^s \ell_{l'ij}x_l, 1\leq i,j,l\leq s .$$

Let $\Theta_G(x)=\prod_{i=1}^s\Phi_i^{e_i}(x)$ be the irreducible factorization of $\Theta_G(x)$. Then each $\Phi_i(x)$ define the unique character $\chi_i$ and there is a bijection between characters and irreducible factors of $\Theta_G(x)$.
And we have the following orthogonal relations:

\begin{itemize}
\item[(i).]{$\sum_i\ell_i\chi_{ij}\chi_{i'l}=0(l\neq j)$,}
\item[(ii).]{$\sum_i\ell_i\chi_{ij}\chi_{i'j}=\frac{\ell\chi_j(1)}{e_j}$,}
\item[(iii).]{$\sum_l{\frac{e_l}{\chi_l(1)}}\chi_{il}\chi_{jl}=0(i\neq j')$,}
\item[(iv).]{$\sum_l{\frac{e_l}{\chi_l(1)}}\chi_{il}\chi_{i'l}=\frac{\ell}{\ell_i}$,}
\item[(v).]{$\sum_l{e_l}\chi_{il}=0$,}
\item[(vi).]{$\sum_le_l\chi_l(1)=\ell$,}

\end{itemize}
where $e_j$ is defined in the same way as in \ref{sec:realization}. Furthermore, it is shown that both $e_j$ and $\chi_j(1)$ are equal to the degree $f_j$ of $\Phi_j(x)$ in section 9 of \cite{F2}.
Similarly, the following equation holds
$$\frac{\ell \ell_{i_1i_2\cdots i_r}}{\ell_{i1}\ell_{i2}\cdots \ell_{ir}}=\sum_l\frac{e_l}{\chi_l^{r-1}(1)}\chi_{i_1l}\cdots \chi_{i_rl}.$$

We can define the number $p_{i_1,\cdots,i_r}$ by commutators. There is a little difference from the definition in section \ref{sec:FrobPoly}.

\begin{definition}\label{def:pijk}
For $1\leq i_1,\cdots,i_r \leq n$,  $p_{i_1,\cdots,i_r}$ is the number of solutions of the equation \begin{equation*}\label{def:commutators} S^{-1}R^{-1}SR=g_{i_1}\cdots g_{i_r}\end{equation*} where
$S,R$ runs over $G$ and $g_{i_j}\in C_{i_j}$ for $1\leq j\leq r$.
\end{definition}

\begin{lemma}\label{lem:ncommutators}
Give $g\in C_i$, then the number of solutions of the equation $s^{-1}r^{-1}sr=g$ is $\sum_{j}\frac{\ell\ell_{ijj'}}{\ell_i\ell_j}$, where $r,s\in G$.
\end{lemma}

\begin{proof}
The number of solutions of the equation $sr=rsg$ is equal to the number of ways that express $g$ into commutators. The number of solutions of the equation $sr=rsg$ is equal to the number of solutions of the equation $r^{-1}s^{-1}rsg=1$. When $r$ runs over $C_j$, the number of solutions of the equation $r^{-1}s^{-1}rsg=1$ is $\frac{l_{i j j'}}{l_i}$. If $g=r^{-1}s^{-1}rs=[r,s]$, then the equation $g=r^{-1}x^{-1} s^{-1}rsx=[r,sx]$ holds true for any $x\in C_G(s^{-1}rs)$. This implies that given a solution $(r^{-1},s^{-1}rs)\in C_{j'}\times C_{j}$ of the equation $r^{-1}s^{-1}rsg=1$ then it can contribute to $|C_G(r)|$ commutators which are equal to $g$. Thus $C_{j'}\times C_{j}$ contributes to $\frac{|C_G(r)|l_jl_{i j j'}}{l_i l_j}=\frac{ll_{i j j'}}{l_i l_j}$ commutators. We will get all commutators which are equal to g when $j$ runs over $n$ conjugate classes. So the number of commutators equal to $g$ is $\sum_{j=1}^n\frac{ll_{i jj'}}{l_i l_j} $.
\end{proof}

\begin{lemma}
 \[\mathfrak{s}_r(\gamma_1(x),\cdots,\gamma_n(x))=\sum_{1\leq i_1,\cdots,i_r \leq n} \frac{1}{\ell}p_{i_1,\cdots,i_r}x_{i_1}\cdots x_{i_r}.\]
\end{lemma}

\begin{proof}
The coefficient of monomial $x_{i_1}x_{i_2}\cdots x_{i_1}$ in $\mathfrak{s}_r(\gamma_1(x),\cdots,\gamma_n(x))$ is the trace of $A_{i_1}\cdots A_{i_r}$.


The coefficient of $\hat{C}_j$ in $\hat{C}_{i_1}\hat{C}_{i_2}\cdots \hat{C}_{i_r}$ is $\frac{\ell_{j'i_1,\cdots,i_r}}{\ell_j}$. By Lemma \ref{lem:ncommutators},  $s^{-1}r^{-1}sr=g$ has $\sum_{i=1}^n\frac{\ell\ell_{j i i'}}{\ell_i}$ solutions for $g\in C_j$. Then the contribution of $C_j$ to solutions to equation \ref{def:commutators} is $\frac{\ell_{j'i_1,\cdots,i_r}}{\ell_j}\sum_{i=1}^n\frac{\ell\ell_{j i i'}}{\ell_i}$ . So the number of solutions of equation \ref{def:commutators} is
\begin{eqnarray*}
&&p_{i_1,\cdots,i_r}= \sum_{j=1}^n\frac{\ell_{j'i_1,\cdots,i_r}}{\ell_j}\sum_{i=1}^n\frac{\ell\ell_{j i i'}}{\ell_i}\\
&=&\sum_{i=1}^n \frac{\ell}{\ell_i}\sum_{j=1}^n \frac{\ell_{j'i_1,\cdots,i_r}\ell_{j i i'}}{\ell_j}
=\sum_{i=1}^n \frac{\ell\ell_{ii'i_1,\cdots,i_r}}{\ell_i}\\
&=&\sum_{i=1}^n \ell_i\ell_{i_1}\cdots\ell_{i_1}\frac{\ell\ell_{ii'i_1,\cdots,i_r}}{\ell_i\ell_i\ell_{i_1}\cdots\ell_{i_1}}\\
&=&\sum_{i=1}^n \ell_i\ell_{i_1}\cdots\ell_{i_1}\sum_{l=1} \frac{1}{\chi_l(1)^{r}}\chi_{i_1l}\cdots\chi_{i_rl}\chi_{il}\chi_{i'l}\\
&=&\ell_{i_1}\cdots\ell_{i_1}\sum_{l=1} \frac{1}{\chi_l(1)^{r}}\chi_{i_1l}\cdots\chi_{i_rl}\sum_{i=1}^n\ell_i\chi_{il}\chi_{i'l}\\
&=&\ell \sum_{l=1} \frac{\ell_{i_1}\chi_{i_1l}}{\chi_l(1)}\cdots \frac{\ell_{i_r}\chi_{i_rl}}{\chi_l(1)}\\
&=&\ell\sum_{l=1}^n \zeta_{i_1 l}\zeta_{i_2 l}\cdots \zeta_{i_r l}=\ell\mbox{Tr}(A_{i_1}\cdots A_{i_r})
\end{eqnarray*}
\end{proof}

Let $\mathcal{P}$ be the trivial partition of Cl($G$). We will denote $F_G^\mathcal{P}(x)$ by $F_G(s)$. Similar to Theorem \ref{thm:pijk}, we prove Theorem B

\begin{theorem}\label{thm:TFPijk} Let $\mathcal{P}$ be the trivial partition of Cl($G$). Then
\begin{enumerate}
\item[(i)]{The Frobenius polynomial $F_G(x)$ and the ordinary characters table $T_G(\mathcal{P})$   are determined by each other.}
\item[(ii)]{The Frobenius polynomial $F_G(x)$ is determined by the numbers $p_{ijk}$ where $1\leq i,j,k\leq n$.}
\end{enumerate}
\end{theorem}
\begin{proof}
By Proposition \ref{contructureant}, the matrices $\{\mathcal{A}_i\}_{i=1}^n$ can be calculated out by $F_G(x)$.
Then there exists an invertible matrix $\mathcal{U}$ such that
\[\mathcal{U}\mathcal{A}_i\mathcal{U}^{-1}=\begin{pmatrix}
\lambda_{i1}& *& *&\cdots &*\\
0& \lambda_{i2}& *&\cdots &*\\
\cdot& \cdot& \cdot&\cdots &\cdot\\
0& 0& 0&\cdots &\lambda_{in}\\
\end{pmatrix}\leqno{(*)}\] for $1\leq i \leq n$.

Since $k\mathcal{P}$ is semisimple, the following matrix $\mathcal{R}$ is invertible.
\[\mathcal{R}=\begin{pmatrix}
\lambda_{1,1}&\lambda_{2,1}&.&.&.&\lambda_{n,1}\\
\lambda_{1,2}&\lambda_{2,2}&.&.&.&\lambda_{n,2}\\
.&.&.&.&.&.\\
\lambda_{1,i}&\lambda_{2,i}&.&.&.&\lambda_{n,i}\\
.&.&.&.&.&. \\
\lambda_{1,n}&\lambda_{2,n}&.&.&.&\lambda_{n,n}
\end{pmatrix}.
\]
In the cass that $\mathcal{P}$ is trivial partition, it is shown that $e_t=f_t$ in \cite{F2}. Since $e_t=\chi_t(1)$, by  Theorem \ref{thm:degree}, the degrees $\chi_i^\mathcal{P}(1),1\leq i \leq n$ can be calculated out. Then $$T_G=\mbox{Diag}(\chi_1^\mathcal{P}(1),\cdots,\chi_n^\mathcal{P}(1))\mathcal{R}$$ is the character table defined by the partition algebra $k\mathcal{P}$. So  $T_G$ is determined by $F_G(x)$.


Given character table  \[T_G=
\begin{array}{c|rrrrrrrrrrrrrrrrrrrrrrrrrrrrrrr}
&C_1&C_2&.&C_j&.&C_n\\
\hline

\chi^\mathcal{P}_1&\chi^\mathcal{P}_{1,1}&\chi^\mathcal{P}_{1,2}&.&\chi^\mathcal{P}_{1,j}&.&\chi^\mathcal{P}_{1,n}\\
\chi^\mathcal{P}_2&\chi^\mathcal{P}_{2,1}&\chi^\mathcal{P}_{2,2}&.&\chi^\mathcal{P}_{2,j}&.&\chi^\mathcal{P}_{2,n}\\
.&.&.&.&.&.&.\\
\chi^\mathcal{P}_i&\chi^\mathcal{P}_{i,1}&\chi^\mathcal{P}_{i,2}&.&\chi^\mathcal{P}_{i,j}&.&\chi^\mathcal{P}_{i,n}\\
.&.&.&.&.&.&. \\
\chi^\mathcal{P}_n&\chi^\mathcal{P}_{n,1}&\chi^\mathcal{P}_{n,2}&.&\chi^\mathcal{P}_{n,j}&.&\chi^\mathcal{P}_{n,n}
\end{array}.\]
Define matrices $D_j=\mbox{Diag}(\frac{\ell_j\chi^\mathcal{P}_{1,j}}{\chi_1^\mathcal{P}(1)},\cdots,\frac{\ell_j\chi^\mathcal{P}_{n,j}}{\chi_1^\mathcal{P}(1)})$ for $1\leq i \leq n$. Then $F_G^\mathcal{P}(x)=|x_1D_1+\cdots +x_nD_n|$. So  $F_G(x)$ is determined by $T_G$.

\end{proof}

\begin{definition}
Let $G$ and $H$ be finite groups. We say that $F_G(x_1,\cdots,x_n)$ is equal to $F_H(x_1,\cdots,x_n)$ by a permutation if there exists a permutation $\sigma\in S_n$ such that $F_G(x_1,\cdots,x_n)=F_H(x_{1^\sigma},\cdots,x_{n^\sigma})$.
\end{definition}

Let $\{C_{i_1},\cdots,C_{i_r}\}$ be a subset of $cl(G)$. We call $\{C_{i_1},\cdots,C_{i_r}\}$ is closed if the set $\cup_{j=1}^r C_{i_j}$ is closed under the multiplication of $G$.

\begin{theorem}
The lattice of normal subgroups of $G$ can be calculated out by the Frobenius polynomial $F_G(x)$.
\end{theorem}
\begin{proof}
By Proposition \ref{contructureant}, $a_{ij\ell}(1\leq i,j,\ell \leq n)$ can be calculated out from $F_G(x)$. The subset of conjugacy classes $\{C_{i_1},\cdots,C_{i_r}\}$ is closed if and only if $a_{\alpha\beta\gamma}=0$ where only two of $\alpha,\beta,\gamma$ are in $\{i_1,\cdots,i_r\}$. So we can determine whether $\{C_{i_1},\cdots,C_{i_r}\}$ is closed by $F_G(X)$. A subgroup $N$ of $G$ is normal if and only if those conjugacy classes belonging to $N$ are closed. Hence all closed subsets of $Cl(G)$ can be calculated out by $F_G(x)$. The relation of inclusion among closed subset of $Cl(G)$ are obvious. So the lattice of normal subgroups of $G$ can be calculated out by $F_G(x)$.
\end{proof}

Let $\rho_t:Z(\mathbb{C}G)\rightarrow \mathbb{C}$ be one dimensional representation such that $\rho_t(\hat{C_i})=\lambda_{it}$ for $1\leq t\leq n$. Then  $\rho_1,\cdots,\rho_n$ are non-isomorphic representations of $Z(\mathbb{C}G)$ and they are direct summand of the regular represention $\mathfrak{L}$ of $Z(\mathbb{C}G)$. Furthermore, $\mathfrak{L}$ is euivalent to $\oplus_{i=1}^n \rho_i$.

Let $\mathcal{A}_j=\mathfrak{L}(\hat{C}_j)$. Then there is an invertible matrix $U$ such that $$U\mathcal{A}_jU^{-1}= Diag(\lambda_{j1},\cdots,\lambda_{jn}).$$ By Theorem \ref{Property:A}, the matrix $\mathfrak{A}=\sum_{i=1}^n x_i\mathcal{A}_i=(\alpha_{ij})$ has $n$ eigenvalues  :$\gamma_1(x),\cdots,\gamma_n(x)$ and $\gamma_t$ is the  eigenvectors of $\mathfrak{A}$ belonging to $\gamma_t(x)$. This means that $\gamma_t\gamma_t(x)=\gamma_t\mathfrak{A}$.  Let $x_i=\delta_{ij}$. Then $\mathfrak{A}=\mathcal{A}_j$, $\gamma_t(x)=\lambda_{jt}$ and $\lambda_{jt}\gamma_t=\gamma_t \mathcal{A}_j$ for $1\leq t \leq n$. This implies that $\gamma_1,\cdots,\gamma_n$ are common eigenvectors of $\mathcal{A}_1,\cdots, \mathcal{A}_n$. Similarly, by Theorem \ref{Othmatrix}, $\varsigma_1,\cdots,\varsigma_n$ are common eigenvectors of $\mathcal{A}_1',\cdots, \mathcal{A}_n'$. Define diagonal matrices $D_i=\mbox{Diag}(\lambda_{i 1},\lambda_{i 2},\cdots,\lambda_{i n})$ for $1\leq i \leq n$.
Then $R'\mathcal{A}_i(R')^{-1}=D_i$, where $R=(\gamma'_1,\cdots,\gamma'_n)$. So $p_{ij}=Tr(\mathcal{A}_i\mathcal{A}_j)=Tr(D_iD_j)=\sum_{t=1}^n \lambda_{it}\lambda_{jt}$. In particular, $R'\mathcal{A}_{i'}(R')^{-1}=D_{i'}=\mbox{Diag}(\overline{\lambda}_{i 1},\overline{\lambda}_{i 2},\cdots,\overline{\lambda}_{i n})$.
Thus we have the following result:
\begin{proposition} \label{prop:mp'}
The matrix $(p_{ij'})$ is determined by $F_G(x)$.
\end{proposition}

The closed subsets of conjugacy classes form a poset under inclusion.  Let $\mathscr{P}_c(G)$ be a set  consisting of $1$-sets and closed subsets of Cl$(G)$. Under the inclusion, $\mathscr{P}_c(G)$ is a poset. We label each $1$-set $\{C_\ell\}$ by $D_\ell$ and label each closed subset by its size. Then $\mathscr{P}_c(G)$ is a labelled poset and it has a class structure of JH-type(see \cite{KS}).
By \cite{KS}, we have the following result:

\begin{corollary}\label{cor:ks}
\ \

\begin{itemize}
  \item[$(i)$] The labelled poset $\mathscr{P}_c(G)$ is determined by $F_G(x)$.
 \item[$(ii)$] For finite groups $G$ and $H$, if $F_G(x)=F_H(x)$,
 then there is an isomorphism $\tau^c_{G,H}$ of labelled posets
 from $\mathscr{P}_c(G)$ to $\mathscr{P}_c(H)$. In addition,
 $\tau^c_{G,H}$ is a correspondence of JH-type.
\end{itemize}
\end{corollary}

\begin{corollary} Let $G$ and $H$ be finite groups. If $F_G(x)$ is equal to $F_H(x)$ by a permutation  then $G$ and $H$ have the same chief factors and the same multiplicity. In particular, $S$ is isomorphic to $T$ if and only if $F_S(x)$ is isomorphic to $F_T(x)$ and $S$ is determined by its $p_{ijk}$, $1\leq i,j,k\leq n$, in case that $S$ is a finite simple group.\end{corollary}



\begin{proof}
It follows from Theorem \ref{thm:TFPijk}, Corollary \ref{cor:ks} and  Theorem 4 in \cite{KS}.
\end{proof}
\begin{theorem}
Two finite groups $G,H$ have the same character table
if and only if there is an isomorphism $\phi$ of
algebras from $Z(\mathbb{C}G)$ to $Z(\mathbb{C}H)$
such that $\phi$ induces a bijection between Cl$(G)$ and Cl$(H)$.
\end{theorem}

\begin{proof} It is clear that the sizes of conjugacy classes of a finite group
are determined by its character table.
If $G$ and $H$ have the same character table,
then the character table defines a natural bijiection $\phi$ from $Cl(G)$ to $Cl(H)$.
In particular, $|\phi(C)|=|C|$ for $C\in Cl(G)$.
By the equations
\begin{equation}\label{equ:defchar2}
\ell_j\ell_l\chi_{j,t}\chi_{l,t}=\chi_{t,1}\sum_{i=1}^n \ell_{i'}a_{ijl}\chi_{i,t}
\end{equation}
for $1\leq i,j,l,t\leq n$, the construction constants
 $\{a_{ijl}\}_{1\leq i,j,l\leq n}$ of $Z(\mathbb{C}G)$ and $Z(\mathbb{C}H)$ are the same.
Extending $\phi$ to $Z(\mathbb{C}G)$ linearly, we
conclude that
$$\phi(\hat{C_i}\hat{C_j})=\phi(\sum_{\ell=1}^n
a_{\ell,i,j}\hat{C_\ell})=\sum_{\ell=1}^n a_{\ell,i,j}\phi(\hat{C_\ell})
=\phi(\hat{C_i})\phi(\hat{C_j}).$$ So $\phi$ is an isomorphism from
$Z(\mathbb{C}G)$ to $Z(\mathbb{C}H)$.

\vspace{1ex}

Conversely, let $Cl(G)=\{C_1,C_2,\ldots,C_n\}$ and
$Cl(H)=\{\mathfrak{C}_1,\ldots,\mathfrak{C}_n\}$.
Since the algebra isomorphism $\phi$ also induces a
bijection from  $Cl(G)$ to $Cl(H)$,
we can assume that $\phi(C_i)=\mathfrak{C}_i$ for $1\leq i \leq n$.
Then the two algebras $Z(\mathbb{C}G)$ and $Z(\mathbb{C}H)$
have the same construction constants $\{a_{ijl}\}_{1\leq i,j,l\leq n}$.
The coefficient $a_{1jl}\not=0$ if and only if $l=j'$ and $|C_j|=a_{1jl}=|\mathfrak{C}_j|$.
Hence those $\{\ell_{ijl}\}_{1\leq i,j,l\leq n}$
in Equation (\ref{basicEqs}) for $G$ and $H$ are same.
So $G$ and $H$ have the same defining equations of characters,
and thus they have the same character table.
\end{proof}

\section{On the degrees of $k$-characters}\label{sec:k-character}
In  this section we will apply Theorem \ref{Mainthm} to the degrees of  $k$-characters, where $k$ is a subfield of the complex field $\mathbb{C}$. If $k=\mathbb{C}$, the $k$-conjugate classes of $G$ are just the conjugate classes of $G$ and $k$-characters are the ordinary characters of $G$. In this section, some Mckay-type conjectures are reformulated and new problems are also put forward.

At first, we will focus on the case that $k=\mathbb{C}$.

  Let $R=(\gamma_1',\cdots,\gamma_n')$ and $S=(\varsigma_1',\cdots,\varsigma_n')$. Then we  have $R'A_jS=diag(\lambda_{j1},\cdots,\lambda_{jn})=S'A_j'R$ for $1\leq j\leq n$ and $RS'=I$.


Let $M$ be the permutation matrix such that $(p_{ij'})M=(p_{ij})$. Then $M^2=I$.
\begin{definition}
The polynomial $D_G(x)=|xI-Diag(\ell\ell_1,\cdots,\ell\ell_n)MSS'|$ is called degree polynomial of $G$ with respect to partition $\mathcal{P}$ of Cl($G$).

\end{definition}

\begin{theorem}\label{thm:degreeMainthm}
The polynomial $D(x)=|xI-Diag(\ell\ell_1,\cdots,\ell\ell_n)MSS'|$ has the following factorization:
$D(x)=(x-\chi_1^2(1))\cdots(x-\chi_n^2(1))$.

\end{theorem}
\begin{proof}
Since we have  $RR'\mbox{Diag}(\frac{1}{\ell\ell_1},\cdots,\frac{1}{\ell\ell_n})\mbox{Diag}(\ell\ell_1,\cdots,\ell\ell_n)SS'=1$ and $RR'=(p_{ij'})M$, it suffices to show that the characteristic polynomial of $(p_{ij'})\mbox{Diag}(\frac{1}{\ell\ell_1},\cdots,\frac{1}{\ell\ell_n})$ has the following factorization \[\mid xI-(p_{ij'})\mbox{Diag}(\frac{1}{\ell\ell_1},\cdots,\frac{1}{\ell\ell_n})\mid =(x-\frac{1}{\chi_1^2(1)})\cdots(x-\frac{1}{\chi_n^2(1)}).\]

Next to show $\mid xI-(\frac{p_{ij'}}{\ell\ell_{j'}}) \mid =(x-\frac{1}{\chi_1^2(1)})\cdots(x-\frac{1}{\chi_n^2(1)})$.

By equation (\ref{equ:pij}), we have \[\frac{p_{ij'}}{\ell\ell_{j'}}=\sum_{\gamma}\frac{\ell_{ij'\gamma}}{\ell\ell_j}\frac{p_\gamma}{\ell_\gamma}=\frac{1}{\ell}\sum_\gamma a_{ji\gamma}\frac{p_\gamma}{\ell_\gamma}.\] So $(\frac{p_{ij'}}{\ell\ell_{j'}})=\frac{1}{\ell}\sum_\gamma A'_\gamma\frac{p_\gamma}{\ell_\gamma}$. It follows from Theorem \ref{Othmatrix}  that\begin{eqnarray*}
 && \mid xI-(\frac{p_{ij'}}{\ell\ell_{j'}})\mid
 = \prod_t(x-\sum_{\gamma=1}^n \frac{1}{\ell}\zeta_{\gamma t}\frac{p_\gamma}{\ell_\gamma})\\
 &=& \prod_t(x-\sum_{\gamma=1}^n \frac{1}{\ell}\frac{\ell_\gamma\chi_{\gamma t}}{\chi_t(1)}\frac{p_\gamma}{\ell_\gamma})\\
 &=& \prod_t(x-\sum_{\gamma=1}^n \frac{1}{\ell}\frac{\ell_\gamma\chi_{\gamma t}\sum_j\chi_{\gamma'j}}{\chi_t(1)\chi_j(1)})\\
 &=& \prod_t(x-\sum_{j=1}^n \frac{1}{\ell}\frac{\sum_\gamma\ell_\gamma\chi_{\gamma t}\chi_{\gamma'j}}{\chi_t(1)\chi_j(1)})\\
 &=& \prod_t(x- \frac{1}{\chi_t^2(1)}).\end{eqnarray*}


\end{proof}

 Next we will define degree element of $G$ such that the degree polynomial is the characteristic polynomial of some element $(\mathfrak{d})\in Z(\mathbb{C}G)$ and ,  $\mathfrak{d}$ is the degree element of $G$.

Let $\mathscr{L}$ be the regular representation of $\mathbb{C}G$. Then a nonsingular symmetric bilinear form $(x,y)=\frac{tr(\mathscr{L}(xy))}{|G|}$ on $Z(\mathbb{C}G)$ is defined in \cite{Z}. By using this form, Zassenhaus defined Casimir element $c$ of $Z(\mathbb{C}G)$ and show that the characteristic polynomial of $\mathscr{L}(c)$ is $\prod_{\chi\in Irr(G)}(\lambda-\frac{|G|}{\chi^2(1)})^{\chi^2(1)}$.

Since the Casimir element is independent of the choice of basis, $\{e_\chi\}_{\chi\in Irr(G)}$ is used as a basis of $Z(\mathbb{C}G)$ to calculate the Casimir element  in \cite{KM}, where $e_\chi=\frac{\chi(1)}{|G|}\sum_{x\in G}\chi(x^{-1})x$ is the primitive idempotent of $Z(\mathbb{C}G)$ corresponding the irreducible character $\chi$. The Casimir element is written into the form:
\[c=\sum_{\chi\in Irr(G)}(\frac{|G|}{\chi(1)})^2e_\chi=\sum_{x,y\in G}xyx^{-1}y^{-1}=\sum_{x\in G}\theta(x)x,\] where $\theta(x)$ denotes the number of pairs $(a,b)\in G\times G$ such that $x=aba^{-1}b^{-1}$. Then there exists element $c^{-1}\in Z(\mathbb{C}G)$ such that $c^{-1}c=1$ and $c^{-1}= \sum_{\chi\in Irr(G)}(\frac{\chi(1)}{|G|})^2e_\chi$. Let $\mathfrak{d}=|G|^2c^{-1}=\sum_{\chi\in Irr(G)}\chi^2(1)e_\chi$. This element $\mathfrak{d}$ is called degree element of $G$. Then we have

\begin{theorem}

The characteristic polynomial of $\mathfrak{L}(\mathfrak{d})$ is the degree polynomial $D_G(x)$ of $G$.

\end{theorem}

\begin{proof}
Since $\mathfrak{L}$ is isomorphic to $\rho=\rho_1\oplus\cdots \oplus \rho_n$, it suffice to calculate $\rho(\mathfrak{d})$. By definition, $\rho_\chi(e_{\chi'})=\delta_{\chi\chi'}$. So
$\rho(\mathfrak{d})=Diag(\chi_1^2(1),\cdots,\chi_r^2(1))$ and the proof is done.
\end{proof}

Let $M_i(G)$ be the number of irreducible characters of $G$ with degree congruent to $\pm i$ modulo $p$, where $p$ is a prime and $i$ is an integer not divisible by $p$.  In \cite{IN}, I.M.Isaacs and G.Navarro refined McKay conjecture into the following form:

{\bf Conjecture.}(Isaacs and Navarro 2002)
If $G$ is a finite groups with $N=N_G(P)$ for some $P\in\mbox{ Sylow}_p(G)$, then $M_i(G)=M_i(N)$ for arbitrary integer $i$ not divisible by $p$.

 Let $\overline{i}$ denote the image of an integer $i$ in $\mathbb{Z}_p$. Let $D_G(x)=x^\ell-a_{\ell-1}x^{\ell-1}+\cdots+a_1x+a_0$. Then $\overline{D_G(x)}=x^\ell-\overline{a_{\ell-1}}x^{\ell-1}+\cdots+\overline{a_1}x+\overline{a_0}= x^sD^{p'}_G(x)$, where $(x,D^{p'}_G(x))=1$. We call $D^{p'}_G(x)$ the $p'$-degree polynomial of $G$ with respect to $p$. Since $D_G(x)=\prod_{\chi_i\in Irr(G)} (x-\chi_i^2(1))$, we have $D^{p'}_G(x)=\prod_{\substack{ p\nmid \chi_i(1)\\\chi_i\in Irr(G)}}(x-\overline{\chi_i^2(1)})$. So each irreducible character of $G$ with degree congruent to $\pm i$ modulo $p$ contributes $1$ to the multiplicity of $x-\overline{i}^2$ in $D^{p'}_G(x)$. Since $\pm \overline{i}$ are the only roots of $x^2-\overline{i}=0$, the multiplicity of $x-\overline{i}^2$ in $D^{p'}_G(x)$ is $M_i(G)$. Hence the refined McKay conjecture can be reformulated into the form:

{\bf Conjecture.}
Let $G$ be a finite group. Then $D^{p'}_G(x)=D^{p'}_N(x)$, where $N=N_G(P)$ and $P\in\mbox{ Sylow}_p(G)$.

\begin{remark}
The polynomials $D^{p'}_G(x)$ and $D^{p'}_N(x)$ have been calculated out to be equal for the groups: $M_{11},M_{12},M_{22},M_{23},M_{24},Sz(8),Sz(32)$ with each prime divisor of its order in \cite{SL}.
\end{remark}

In the rest, we will consider the case that $k$ is not a splitting field for $G$ and its characteristic is zero.
We will formulate a new conjecture and show its connection with conjecture A in \cite{N}.
 Let $K=k(\xi)$, where $\xi$ is primitive $|G|$-th roots of unity. So $K$ is a splitting field of $G$ and its subgroups. Let $\mathscr{G}=$Gal($K/k$) be the Galois group.  For any $\sigma\in \mbox{Gal}(K/k)$, there exists a unique $t \in \mathbb{Z}/m\mathbb{Z}$ such that $\sigma(\omega)=\omega^t$ for $|G|$-th root $\omega$ of unity. In this way, we define an isomorphism $\tau$ from $\mbox{Gal}(K/k)$ to $(\mathbb{Z}/|G|\mathbb{Z})^*$. Let $\Gamma=\{\tau(\sigma)|\sigma\in \mbox{Gal}(K/k)\}$. We denote by $\sigma_t$ the inverse image of $t\in \Gamma$. We can define $k$-classes as in section \ref{sec:Frobenius theory}. The $k$-classes is a good partition.

Let  $$\Theta_G(x)=|M_G(x)|=\prod_i^r\Phi_i^{e_i}(x)$$ be the factorization of irreducible factors over $K$. Let $\chi_i$ be the character defined by $\Phi_i(x)$.
 Then $\mathscr{G}$ acts on $\Phi_1(x),\cdots,\Phi_n(x)$. Let $\Omega_1,\cdots,\Omega_r$ be the orbits of the action of $\mathscr{G}$. Let $o_i$ be the length of $\Omega_i$. If $\Phi_j(x)$ and $\Phi_l(x)$ belong to the same orbit, then they have the same degree and the same multiplicity i.e. $f_j=\chi_j(1)=\chi_l(1)=f_i=e_j=e_l$ by section 9 in \cite{F2}. Let $\Psi_i=\prod_{\Phi(x)\in \Omega_i}\Phi(x)$.  Then $\Theta_G(x)=\prod_i^r\Psi_i(x)^{e_i}$ is the irreducible factorization of $\Theta_G(x)$ over $k$, where $e_i$ is equal to the multiplicity of $\Phi(x)\Omega_i$ in $\Theta_G(x)$. Assume that $\Omega_j=\{\Phi_{j_1}(x),\cdots,\Phi_{j_{o_j}}(x)\}$. Then $\Psi_j(x)=\Phi_{j_1}(x)\cdots\Phi_{j_{o_j}}(x)$ and the degree of $\Psi_j(x)$ is $f_{j_1}o_j$. Define a function $\chi^k_i$ such that $\chi^k_i(g)$ is the coefficient of $x_E^{f_{j_1}o_j-1}x_g$ in $\Psi_j(x)$ and $\chi_i^k(1)=f_{j_1}o_j$.

 \begin{lemma}
 If $x,y$ belong to the same $k$-conjugate class, then $\chi_j^k(x)=\chi_j^k(y)$ and $\chi_j^k$ is a function over $k$.
 \end{lemma}
 \begin{proof}
 Let $\chi_i$ be the character defined by $\Phi_i(x)$.  Let $T$ be a transversal of the stabilizer of $\chi_{j_1}$ in $\mbox{Gal}(K/k)$.  Then $\chi_j^k=\sum_{j_i=j_1}^{j_{o_j}}\chi_{j_i}=\sum_{\sigma\in T} \chi_{j_1}^\sigma$. So $\chi_j^k=(\chi_j^k)^g$ for any $g\in \mbox{Gal}(K/k)$. If $x$ is $k$-conjugate to $y$ then $y$ is conjugate to $x^q$ for some $q\in \Gamma$. It suffices to show that $\chi_j^k(x)=\chi_i^k(x^q)$. By Section 12 in \cite{F2}, $\chi_{j_1}(x)=\rho_1+\cdots +\rho_r$ and $\chi_{j_1}(x^q)=\rho_1^q+\cdots +\rho_r^q$, where $\rho_1,\cdots ,\rho_r$ are $\circ(x)$-th roots of unity. This means that $\chi_{j_1}(x^q)=\rho_1^{\sigma_q}+\cdots +\rho_r^{\sigma_q}=\chi_{j_1}^{\sigma_q}(x)$. So $\chi_j^k(x^q)=\sum_{\sigma\in T} \chi_{j_1}^{\sigma}(x^q)=\sum_{\sigma\in T} (\chi_{j_1}(x^q))^{\sigma}=\sum_{\sigma\in T} (\chi_{j_1}^{\sigma_q}(x))^{\sigma}=\sum_{\sigma\in T} (\chi_{j_1}^{\sigma_q\sigma}(x))$. Since $\{\sigma_q\sigma|\sigma\in T\}$ is another transversal of the stabilizer of $\chi_{j_j}$ in $\mbox{Gal}(K/k)$, we have $\chi_j^k(x)=\sum_{\sigma\in T} \chi_{j_1}^\sigma(x)=\sum_{\sigma\in T} \chi_{j_1}^{\sigma_q\sigma}(x)=\chi_j^k(x^q)$.
 \end{proof}

Let $\mathcal{P}$ be the $k$-conjugate classes. Let $\mathcal{P}_1,\cdots,\mathcal{P}_n$ be all $k$-classes with $g_j\in \mathcal{P}_j$ as representative.

\begin{lemma}\label{lem:kvaluechip}
Let $\chi^\mathcal{P}$ be the character defined by $\mathcal{P}$ and an irreducible factor $\Phi(x)$ of $\Theta_G(x)$. Then $\chi^\mathcal{P}(g)\in k$ for $g\in G$.
\end{lemma}
\begin{proof}
Assume that $g$ belongs $F$-class $\mathcal{P}_j$ and $\mathscr{C}_{j_1},\cdots,\mathscr{C}_{js}$ are conjugate classes contained in $\mathcal{P}_j$. Then we can choose representatives of $\mathscr{C}_{j_1},\cdots,\mathscr{C}_{js}$ in $\{g^t|t\in \Gamma\}$. Let $g^{t_i}$ be the representative of $\mathscr{C}_{j_i}$.Then $\chi^\mathcal{P}(g)=l(\sum_{i=1}^s\chi(g^{t_i})$, where $l$ is the length of  $\mathscr{C}_{j_1}$ and $\chi$ is a character defined by $\Phi(x)$. Let $\sigma\in \mathscr{G}$ and $q$ is the element in $\Gamma$ corresponding to $\sigma$.  Then there exist some $o(g)$-th root of unity $\rho_1,\cdots ,\rho_r$ such that $\chi(g)=\rho_1+\cdots +\rho_r$ and  $\chi^\sigma(g)=\rho_1^q+\cdots +\rho_r^q=\chi(g^q)$ by section 12 in \cite{F2}. Since any number in $\Gamma$ is prime to the order of $g$, $g^{t_1q},\cdots g^{t_sq}$ are still representatives of $\mathscr{C}_{j_1},\cdots,\mathscr{C}_{js}$. So $(\chi^\mathcal{P}(g))^\sigma=l(\sum_{i=1}^s\chi^\sigma(g^{t_i})=
l(\sum_{i=1}^s\chi(g^{t_iq})=\chi^\mathcal{P}(g)$. So $\chi^\mathcal{P}(g)\in k$ for $g\in G$.
\end{proof}

\begin{lemma}\label{lem:k-eigenvalue}
Set $x_A=x_B$ in the matrix $M_G(x)$ whenever $A$ and $B$ belong to the same $k$-class. Then $M_G(x)$ is similar to a diagonal matrix with linear polynomials of $k[x]$ in the diagonal.
\end{lemma}
\begin{proof}
Define matrix $M_i$ to be $M_G(x)$ by taking $x_g=1$ whenever $g\in \mathcal{P}_i$ and $x_g=0$ otherwise. Then the map $$\mathfrak{r}:k\mathcal{P}\rightarrow M_\ell(k)(\hat{C}_i\mapsto M_i)$$ is a faithful representation of the partition algebra $k\mathcal{P}$.
Under the assumption that $x_A=x_B$ in the matrix $M_G(x)$ whenever $A$ and $B$ belong to the same $F$-class,  $M_G(x)=\sum_{i=1}^n x_{g_i}M_i$.
Since $k\mathcal{P}$ is semisimple,  $M_G(x)$ is similar to a diagonal matrix.

Next it suffices to show $|M_G(x)|=\Theta_G(x)$ is a product of linear polynomials of $k[x]$  under the assumption that $x_A=x_B$ in the matrix $M_G(x)$ whenever $A$ and $B$ belong to the same $k$-class.  By Theorem \ref{thm:abbaExp}, $\Theta_G(x)=|M_G(x)|=\prod_i^r\Phi_i^{e_i}(x)=\prod_i^r(\sum_{j=1}^n \chi_i^\mathcal{P}(g_j)x_{g_j})$. By Lemma \ref{lem:kvaluechip}, $\chi_i^\mathcal{P}$ is a function over $k$. So $\chi_i^\mathcal{P}(g_j)x_{g_j}$ is a linear polynomial over $k$.
\end{proof}

\begin{theorem}\label{thm:k-abbaext}
Let $\Omega_j=\{\Phi_{j_1}(x),\cdots,\Phi_{j_{o_j}}(x)\}$ and $\Psi_j(x)=\prod_{\Phi_{j_i}(x)\in \Omega} \Phi_{j_i}(x)$. If  we set $x_A=x_B$ in the $\Psi_j(x)$ whenever $A$ and $B$ belong to the same $k$-class, then
\begin{itemize}
\item[(i)]{$\Psi_j(x)=(\frac{1}{o_jf}\sum_{i=1}^n\chi_j^k(g_i)x_{g_i})^{o_jf}$, where $f$ is the degree of one polynomial in $\Omega_j$;}
\item[(ii)]{$\chi^\mathcal{P}_{j_m}=\chi^\mathcal{P}_{j_n}$ for $1\leq m,n \leq o_j$;}
\item[(iii)]{$o_j\chi^\mathcal{P}_{j_m}=\chi^k_j$ for $1\leq m \leq o_j$.}
\end{itemize}
\end{theorem}
\begin{proof}
With the assumption that $x_A=x_B$ in the matrix $M_G(x)$ whenever $A$ and $B$ belong to the same $k$-class, the eigenvalues of  $M_G(x)$
are linear polynomials over $k$ by Lemma \ref{lem:k-eigenvalue}. By the similar proof of Theorem \ref{thm:abbaExp}, we have $\Psi_j(x)=\frac{1}{o_jf}(\sum_{i=1}^n\chi_j^k(g_i)x_{g_i})^{o_jf}$.

Set $x_A=x_B$ in the $\Phi_{j_m}(x)$ and $\Phi_{j_n}$ whenever $A$ and $B$ belong to the same $k$-class, then $\Phi_{j_m}(x)=\frac{1}{f}\sum_{i=1}^n\chi_{j_m}^{\mathcal{P}}(g_i)x_{g_i}$ and $\Phi_{j_n}(x)=\frac{1}{f}\sum_{i=1}^n\chi_{j_n}^{\mathcal{P}}(g_i)x_{g_i}$ are polynomials over $k$ by Lemma \ref{lem:kvaluechip} and \ref{thm:abbaExp}. Since $\Psi_j(x)=\Phi_{j_1}(x)\cdots\Phi_{j_{o_j}}(x)$, it follows from (i) that $\chi^\mathcal{P}_{j_m}=\chi^\mathcal{P}_{j_n}=\frac{1}{o_j}\chi^k_j$.
\end{proof}

With the assumption that $x_A=x_B$ when $A$ and $B$ belongs to the same $k$-class, the multiplicity of $\frac{1}{f_{j_m}}\sum_{i=1}^n\chi^\mathcal{P}_{j_m}(g_i)x_{g_i} $ in $\Theta_G(x)=\prod_i^r\Psi_i(x)^{e_i}$ is $o_je_{j_m}f_{j_m}$ by Theorem \ref{thm:abbaExp} and Theorem \ref{thm:k-abbaext}. Since $e_{j_m}=f_{j_m}=\chi_{j_m}(1)$, by Theorem \ref{Mainthm} and \ref{cor:etchi1}, $D_\mathcal{P}(x)=(x-\chi_1(1)^2o_1)\cdots(x-\chi_n(1)^2o_n)$.  Let $P$ be a Sylow$_p$-subgroup of $G$ with $N=N_G(P)$. Let $\mathcal{P}'$ be the partition consisting of $k$-classes of $N$.
Similar to definition in section \ref{sec:degreepoly}, we define $p'$-degree polynomials $D_G^{k,p'}(x)$ of $D_\mathcal{P}(x)$ for $G$ and $p'$-degree polynomials $D_N^{k,p'}(x)$ of $D_{\mathcal{P}'}(x)$ for $N$. We call $D_G^{k,p'}(x)$ and $D_N^{k,p'}(x)$
the $p'$-degree polynomials with respect to $(G,\mathcal{P})$ and $(N,\mathcal{P}')$ respectively.
Similar to refined McKay's conjecture, it is natural to ask whether $D_G^{k,p'}(x)=D_G^{k,p'}(x)$ for $k$-classes of $G$ and $N$.

{\bf Conjecture.} Let $P$ be a Sylow$_p$-subgroup of $G$ with $N=N_G(P)$. Let $\mathcal{P}$ and $\mathcal{P'}$ be the good partitions $F$-conjugate classes of $G$ and $N$ respectively. Then $D_G^{k,p'}(x)=D_N^{k,p'}(x)$, where $D_G^{p'}(x)$ and $D_N^{p'}(x)$
are $p'$-degree polynomials with respect to $(G,\mathcal{P})$ and $(N,\mathcal{P}')$.

Let  $n$ be the exponent of $G$ and let $\mathbb{Q}_n$ be the cyclotomic field. Then $\mathbb{Q}_n$ is a splitting field of $G$ and its subgroups. Therefore, the Galois group $\mathscr{G} =
\mbox{Gal}(Q_n/Q)$ permutes the set Irr$_\mathbb{C}(G)$ of the irreducible complex characters of $G$. Let Irr$_{p'}(G)$ be the set of complex irreducible characters of $G$ with degree not divisible by $p$. Navarro proposed the following conjecture in \cite{N}.

{\bf Conjecture.} \cite{N} Let $G$ be a finite group of order $n$ and let $p$ be a prime.
Let $e$ be a nonnegative integer and let $\sigma \in \mbox{Gal}(Q_n/Q)$ be any Galois automorphism sending every $p'$-root of unity $\xi$ to $\xi^{p^e}$. Then $\sigma$ fixes the same number
of characters in Irr$_{p'}(G)$ as it does in Irr$_{p'}(N_G(P))$, where $P$ is a Sylow$_p$-subgroup of $G$.

 Let $\xi\in \mathbb{C}$  be a primitive $n$-th root of unity.  Then $\mathbb{Q}_n =\mathbb{Q}(\xi)$ is the
cyclotomic field.  We write $\xi=\omega\delta$,
where  the order of $\omega$ is $p^d$ ,the order of $\delta$ is $m$ , $n=p^dm$ and $(p,m)=1$.  Let $\mathcal{H}$ be the subgroup of $\mathscr{G}=\mbox{Gal}(\mathbb{Q}_n / \mathbb{Q}) $ consisting
of those elements $\sigma\in \mathscr{G}$ for which there is a nonnegative integer $e$ such that $\sigma(\lambda)=\lambda^{p^e}$
whenever
$\lambda$ is a $p'$-root of unity in $<\xi>$.
Let $\mathcal{K} = \{\tau\in \mathscr{G}|\tau(\delta) = \delta\}$.  Then $\mathcal{K}$ is isomorphic to
the group Gal$(\mathbb{Q}_{p^d} /\mathbb{Q})$ of order $\phi(p^d)$. Let $\mathcal{ J} = \{\tau\in \mathscr{G}|\tau(\omega)=\omega\}$.
 Then $\mathcal{ H} = \mathcal{K}\times <\sigma>$, where $\sigma\in \mathcal{J}$ such that $\sigma(\delta)=\delta^p$,
and the order of $\sigma$ is the order of $p$ modulo $m$.

\begin{proposition} Let $\mathbb{Q}^\mu$ be the fixed field of $\mu$ in $\mathbb{Q}_n$ for $\mu \in \mathscr{G}$. Let $N$ be the normalizer of some Sylow$_p$-subgroup of $G$. Let $\mathcal{P}$ and $\mathcal{P'}$ be the $\mathbb{Q}^\mu$-classes of $G$ and $N$ respectively. If  $\mu$ is a $p$-element contained in $\mathcal{K}$,
then the following are equivalent:
\begin{itemize}
  \item[(i)] {The automorphism $\mu$ fixes the same number
of characters in Irr$_{p'}(G)$ as it does in Irr$_{p'}(N)$;}
  \item[(ii)]{The polynomial $D_G^{p'}(x)$ is equal to the polynomial $D_N^{p'}(x)$, where $D_G^{p'}(x)$ and $D_N^{p'}(x)$
are $p'$-degree polynomials with respect to $(G,\mathcal{P})$ and $(N,\mathcal{P}')$.}
\end{itemize}

\end{proposition}

\begin{proof} Let $\Theta_G(x)=\prod_i^r\Phi_i^{e_i}(x)$ be the irreducible factorization over $\mathbb{Q}_n$. The Galois group Gal$(\mathbb{Q}_n/ \mathbb{Q}^\mu)$ is isomorphic to $<\mu>$ and it is of $p$-power order. Then Gal$(\mathbb{Q}_n/ \mathbb{Q}^\mu)$ acts on $\{\Phi_i(x)|i=1,\cdots,r\}$. Let $\Omega_1,\cdots,\Omega_t$ be the orbits of this action. Then
 the length of each orbits  is $1$ or $p^a$ for some $a\not=1$. By Theorem \ref{Mainthm} and \ref{cor:etchi1}, $D_\mathcal{P}(x)=(x-\chi_1(1)^2o_1)\cdots(x-\chi_n(1)^2o_n)$, where $o_i$ is the length of $\Omega_i$. So $$D_G^{p'}(x)=\prod_{\stackrel{p\nmid \chi_i(1)}{ o_i=1}}(x-\overline{\chi_i(1)^2o_i}).$$ We can get similar expression of $D_N^{p'}(x)$. So (i) is equivalent to (ii).
\end{proof}

Given a good partition $\mathcal{P}$ of Cl($G$). We have define $D_\mathcal{P}(x)$ in section \ref{sec:realization}. Similarly, we can define $p'$-degree polynomial $D_\mathcal{P}^p(x)$ of $D_\mathcal{P}(x)$. The following problem seems reasonable.

 {\bf Problem.} Let $P$ be a Sylow$p$-subgroup of $G$ with $N=N_G(P)$.
Given a good partition $\mathcal{P},\mathcal{P'}$ of  Cl($G$) and Cl($N$) respectively. If each partition class $\mathcal{P}_i$ of $\mathcal{P}$ satisfies that $\mathcal{P}_i\cap N$ is an union of some partition classes of $\mathcal{P'}$, then $D_\mathcal{P}^p(x)=D_\mathcal{P'}^p(x)$.

\subsubsection{References}
\begingroup

\renewcommand{\section}[2]{}%
\def\refname{{REFERENCES}}

\endgroup


\begin{thebibliography}{fgts}
\bibitem[A]{A}
C.A.M.Andr\'{e},
The basic character table of the unitriangular group.
{\em J. Algebra} \textbf{241} (2001), no. 1, 437-471.
\bibitem[BCFS]{BCFS}W. Bosma, J. J. Cannon, C. Fieker, A. Steel (eds.), {\em Handbook of Magma functions}, Edition 2.16 (2010), 5017 pages.







\bibitem[BO]{BO} N. Bourbaki, Algebra. II. Chapters 4-7. Springer-Verlag, Berlin, 1990.






\bibitem[DI]{DI}
P.Diaconis and I. M.Isaacs, Supercharacters and superclasses for algebra groups,{\em Trans. Amer. Math. Soc.} \textbf{360} (2008), no. 5, 2359-2392
\bibitem[FS]{FS} E.Formanek and D.Sibley,
The group determinant determines the group.
{\em Proc. Amer. Math. Soc.} {\textbf 112} (1991), no. 3, 649-656.

\bibitem[F1]{F1} G. Frobenius, \"{U}ber Gruppencharaktere, {\em Sitzungsber.Preuss.Akad.Wiss.Berlin} (1896), 985--1021.

\bibitem[F2]{F2} G. Frobenius, \"{U}ber die Primfaktoren der gruppendeterminante, {\em Sitzungsber.Preuss.Akad.Wiss.Berlin} (1896), 1343--1382.
\bibitem[G]{G}M.Geck, An introduction to algebraic geometry and algebraic groups. Oxford Graduate Texts in Mathematics, 10. Oxford University Press, Oxford, 2003
\bibitem[H1]{H1}
H.J.Hoehnke,\"{U}ber komponierbare Formen und konkordante hyperkomplexe Gr\"{o}ssen. {\em Math. Z.} \textbf{70} (1958),1-12
\bibitem[HJ]{HJ} H.J.Hoehnke and K. W. Johnson,  The $1$-, $2$-, and $3$-characters determine a group.{\em Bull. Amer. Math. Soc.} (N.S.){\textbf 27} (1992), no. 2, 243-245
\bibitem[I]{I} I.M.Isaacs,   Character theory of finite groups.  Dover Publications, Inc., New York, 1994
\bibitem[IN]{IN}
I.M.Isaacs and G.Navarro, New refinements of the McKay conjecture for arbitary finite groups.  {\em Ann.of Math} (2) \textbf{156} (2002),333-344
\bibitem[J]{J}K.W. Johnson, Group matrices, group determinants and representation theory. The mathematical legacy of Frobenius.{\em Lecture Notes in Mathematics}\textbf{ 2233}  (2019)

\bibitem[KM]{KM}
P.Kellersch and K.Meyber, On a casimir element of a finite group. {\em Comm.Alg} 25\textbf{6}(1997),1695-1702



\bibitem[KS]{KS}
W.Kimmerle and R.Sandling, Group theoretic and group ring theoretic determination of certain Sylow and Hall subgroups and the resolution of a question of R.Brauer.  {\em J.Algebra}  \textbf{171} (1995),329-346





\bibitem[MA]{MA}
 I. G.Macdonald, Symmetric functions and Hall polynomials. Second edition. With contributions by A. Zelevinsky. Oxford Mathematical Monographs. Oxford Science Publications. The Clarendon Press, Oxford University Press, New York, 1995.
\bibitem[M]{M}
R.Mansfield, A group determinant determines its group. {\em Proc.Amer.Math.Soc} \textbf{116}(4)  (1992), 939-941

\bibitem[MK]{MK}
J.McKay, Irreducible representations of odd degree. {\em J. Algebra } \textbf{20}  (1972), 416-418

\bibitem[NT]{NT}
H. Nagao and Y. Tsushima, {\em Representation of finite groups}, Academic Press, New York, 1988.
\bibitem[N]{N}G.Navarro,
The McKay conjecture and Galois automorphisms.{\em Ann. of Math} (2) \textbf{160} (2004), no. 3, 1129¨C1140.

\bibitem[S]{S}
J.P. Serre, Linear representations of finite groups.  Graduate Texts in Mathematics, Vol. 42. Springer-Verlag, New York-Heidelberg, 1977.

\bibitem[SL]{SL}
L.Song, McKay conjecture and its checking by computers, M.SC.Thesis, Peking University, (2019)

\bibitem[Z]{Z}
H.Zassenhaus, An equation for the degrees of the ablsolutely irreducible representations of a group of finite order. {\em Canad.J.Math } \textbf{2}  (1950), 166-167
\end{thebibliography}
\end{document}